\documentclass[11pt]{article}

\usepackage[colorlinks, linkcolor=blue, citecolor=blue]{hyperref}            
\usepackage{color}
\usepackage{graphicx,subfigure,amsmath,amssymb,amsfonts,bm,epsfig,epsf,url,dsfont}
\usepackage{amsthm}
\usepackage{tikz}
\usepackage{bbm}      
\usepackage{booktabs}
\usepackage{cases}
\usepackage{fullpage}
\usepackage[small,bf]{caption}
\usepackage{natbib}
\usepackage[top=1in,bottom=1in,left=1in,right=1in]{geometry}
\usepackage{fancybox}
\usepackage[bottom]{footmisc}
\usepackage{appendix}

\newcommand{\mL}{\mathcal{L}}

\newcommand{\mG}{\mathcal{G}}
\newcommand{\mN}{\mathcal{N}}

\newcommand{\mP}{\mathcal{P}}
\newcommand{\mQ}{\mathcal{Q}}

\newcommand{\TV}{d_{\textnormal{TV}}}

\newcommand{\bP}{\mathbb{P}}
\newcommand{\bE}{\mathbb{E}}

\newcommand{\pr}[1]{\textsc{#1}}

\newcommand\numberthis{\addtocounter{equation}{1}\tag{\theequation}}

\makeatletter
\newtheorem*{rep@theorem}{\rep@title}
\newcommand{\newreptheorem}[2]{%
\newenvironment{rep#1}[1]{%
 \def\rep@title{#2 \ref{##1}}%
 \begin{rep@theorem}}%
 {\end{rep@theorem}}}
\makeatother

\newreptheorem{theorem}{Theorem}
\newreptheorem{corollary}{Corollary}

\newtheorem{definition}{Definition}[section]
\newtheorem{question}{Question}[section]

\newtheorem{proposition}{Proposition}[section]
\newtheorem{theorem}{Theorem}[section]
\newtheorem{corollary}{Corollary}[section]
\newtheorem{conjecture}{Conjecture}[section]
\newtheorem{lemma}{Lemma}[section]

\newenvironment{fminipage}%
  {\begin{Sbox}\begin{minipage}}%
  {\end{minipage}\end{Sbox}\fbox{\TheSbox}}

\usepackage{bbm}

\newcommand{\unif}{\mathsf{unif}}

\newcommand{\tv}{d_\mathrm{TV}}

\newcommand{\nurgg}{\nu^{\mathsf{rgg}}}

\newcommand{\kl}{\mathsf{KL}}

\newcommand{\grgg}{G^\mathsf{rgg}}

\newcommand{\sphere}{\mathbb{S}}

\theoremstyle{remark}

\numberwithin{equation}{section}

\begin{document}

\title{Phase Transitions for Detecting Latent Geometry \\ in Random Graphs}

\author{Matthew Brennan\thanks{Massachusetts Institute of Technology. Department of EECS. Email: \texttt{brennanm@mit.edu}.}
\and 
Guy Bresler\thanks{Massachusetts Institute of Technology. Department of EECS. Email: \texttt{guy@mit.edu}.}
\and
Dheeraj Nagaraj\thanks{Massachusetts Institute of Technology. Department of EECS. Email: \texttt{dheeraj@mit.edu}.}}
\date{\today}

\maketitle

\begin{abstract}
Random graphs with latent geometric structure are popular models of social and biological networks, with applications ranging from network user profiling to circuit design. These graphs are also of purely theoretical interest within computer science, probability and statistics. A fundamental initial question regarding these models is: when are these random graphs affected by their latent geometry and when are they indistinguishable from simpler models without latent structure, such as the Erd\H{o}s-R\'{e}nyi graph $\mG(n, p)$? We address this question for two of the most well-studied models of random graphs with latent geometry -- the random intersection and random geometric graph. Our results are as follows:
\begin{itemize}
\item The random intersection graph is defined by sampling $n$ random sets $S_1, S_2, \dots, S_n$ by including each element of a set of size $d$ in each $S_i$ independently with probability $\delta$, and including the edge $\{i, j\}$ if $S_i \cap S_j \neq \emptyset$. We prove that the random intersection graph converges in total variation to an Erd\H{o}s-R\'{e}nyi graph if $d = \tilde{\omega}(n^3)$, and does not if $d = o(n^3)$, for both dense and sparse edge densities $p$. This resolves an open problem in \cite{fill2000random, rybarczyk2011equivalence, kim2018total}. The same result was obtained simultaneously and independently by Bubeck, Racz and Richey \cite{bubeck2019geometry}.
\item We strengthen the preceding argument to show that the matrix of random intersection sizes $|S_i \cap S_j|$ converges in total variation to a symmetric matrix with independent Poisson entries. This yields the first total variation convergence result for $\tau$-random intersection graphs, where the edge $\{i, j\}$ is included if $|S_i \cap S_j| \ge \tau$. More precisely, our results imply that, if $p$ is bounded away from $1$, then the $\tau$-random intersection graph with edge density $p$ converges to $\mG(n, p)$ if $d = \omega(\tau^3 n^3)$.
\item The random geometric graph on $\mathbb{S}^{d - 1}$ is defined by sampling $X_1, X_2, \dots, X_n$ uniformly at random from $\mathbb{S}^{d - 1}$ and including the edge $\{i, j\}$ if $\| X_i - X_j \|_2 \le \tau$. A result of \cite{bubeck2016testing} showed that this model converges to $\mG(n, p)$ in total variation, where $p$ is chosen so that the models have matching edge densities, as long as $d = \omega(n^3)$. It was conjectured in \cite{bubeck2016testing} that this threshold decreases drastically for $p$ small. We make the first progress towards this conjecture by showing convergence if $d = \tilde{\omega}\left(\min\{ pn^3, p^2 n^{7/2} \} \right)$.
\end{itemize}
Our proofs are a hybrid of combinatorial arguments, direct couplings and applications of information inequalities. Previous upper bounds on the total variation distance between random graphs with latent geometry and $\mG(n, p)$ have typically not been both combinatorial and information-theoretic, while this interplay is essential to the sharpness of our bounds.
\end{abstract}

\pagebreak

\tableofcontents

\pagebreak

\section{Introduction}

Random graphs have emerged as ubiquitous objects of interest in a variety of fields. The topic of inference on random graphs encompasses a number of important statistical problems with applications ranging from social, genetic and biological networks to network user profiling and circuit design. Random graphs are also primary combinatorial objects of interest within the computer science, probability and statistics communities. Many contemporary random graphs in applications arise as models of pairwise relationships between latent points $X_1, X_2, \dots, X_n$ drawn at random from a high-dimensional feature space. This feature space is referred to as the social space in random models of social networks \cite{hoff2002latent}. In these models with latent structure, the points $X_i$ capture the underlying attributes of nodes in the network that determine the formation of edges. The underlying geometry of the feature space influences the emergent properties of the network, often leading to desirable properties of real-world networks such as the small-world phenomenon and clustering.

An initial fundamental question regarding these models is: when are random graphs with latent geometric structure actually influenced by their geometry? In other words, when are these models capturing more than simpler models without any latent structure? As the dimension $d$ of the latent feature space increases, it is often the case that the numerous degrees of freedom in the points $X_i$ cause the connections in the graph to appear less correlated and more independent. In the high-dimensional limit $d \to \infty$, these models begin to resemble the simplest random graph without latent structure, the Erd\H{o}s-R\'{e}nyi graph $\mG(n, p)$, wherein each edge is included independently with probability $p$. This leads to the following precise reformulation of our general question.

\begin{question}
Given a random graph model with $n$ nodes, latent geometry in dimension $d = d(n)$ and edge density $p = p(n)$, for what triples $(n, d, p)$ is the model indistinguishable from $\mG(n, p)$?
\end{question}

We address this question for two of the most popular models of random graphs with latent geometry -- the random intersection and random geometric graph. The random intersection graph $G$ is defined by sampling $n$ random sets $S_1, S_2, \dots, S_n$ by including each element of a set of size $d$ in each $S_i$ independently with probability $\delta$, and including the edge $\{i, j\}$ in $G$ if $S_i$ and $S_j$ have nonempty intersection. Random intersection graphs were introduced in \cite{karonski1999random}. A long line of research has examined the combinatorial properties of random intersection graphs, including their diameter, connectivity and large components \cite{blackburn2009connectivity, rybarczyk2011diameter}, independent sets \cite{nikoletseas2008large}, degree distribution \cite{stark2004vertex, jaworski2006degree, bloznelis2013degree} and threshold functions \cite{rybarczyk2011sharp}. Random intersection graphs have also found a range of applications, including to epidemics \cite{ball2014epidemics, britton2008epidemics}, circuit design \cite{karonski1999random}, network user profiling \cite{marchette2005random}, the security of wireless sensor networks \cite{bloznelis2009component}, key predistribution \cite{yagan2008random} and cluster analysis \cite{godehardt2003two}. A line of work directly relevant to our question of focus has examined common properties between random intersection and Erd\H{o}s-R\'{e}nyi graphs \cite{fill2000random, rybarczyk2011equivalence, kim2018total}. For a more extensive survey of properties and applications of random intersection graphs, see Chapter 12 of \cite{frieze2016introduction} or \cite{bloznelis2015recenta} and \cite{bloznelis2015recentb}. A more general model is the $\tau$-random intersection graph, where each edge $\{i, j\}$ is included if $|S_i \cap S_j| \ge \tau$. The $\tau$-random intersection graph was introduced in \cite{godehardt2003two} and its clique number has recently been examined in \cite{bloznelis2017large}. For $\tau > 1$, this model is analytically more difficult than the ordinary random intersection graph and the question of its convergence to $\mG(n, p)$ in total variation has not yet been studied. 

The random geometric graph on $\mathbb{S}^{d - 1}$ is defined by sampling $X_1, X_2, \dots, X_n$ uniformly at random from $\mathbb{S}^{d - 1}$ and including the edge $\{i, j\}$ if $\| X_i - X_j \|_2 \le \tau$. A large body of literature has been devoted to studying the properties of low-dimensional random geometric graphs, and a survey of this literature can be found in the monograph \cite{penrose2003random}. Random geometric graphs have many well-studied applications, including to wireless networks \cite{santi2005topology, haenggi2009stochastic}, gossip algorithms \cite{boyd2006randomized} and optimal planning \cite{karaman2011sampling}. In this work, we focus on the high-dimensional setting where the dimension $d$ of the latent space $\mathbb{S}^{d - 1}$ grows as a function of $n$. A recent line of work has studied properties of high-dimensional random geometric graphs, including their clique number \cite{devroye2011high, arias2015detecting}, edge and triangle statistics \cite{bubeck2016testing, grygierek2019poisson}, convergence to Erd\H{o}s-R\'{e}nyi graphs in total variation \cite{devroye2011high, bubeck2016testing, eldan2016information, racz2017basic} and birthday inequalities \cite{perkins2016birthday}. This prior work as it relates to our question is discussed in Section \ref{subsec:relwork}. In \cite{bubeck2016testing}, it was shown that random geometric graphs on $\mathbb{S}^{d - 1}$ with marginal edge density $p$ converge to $\mG(n, p)$ in total variation as long as $d = \omega(n^3)$. It was conjectured that if $p$ decays as a function of $n$, this threshold should also decrease from $O(n^3)$. However, as will be discussed in Section \ref{subsec:relwork}, the techniques in \cite{bubeck2016testing} rely on a coupling of random matrices that fails if $d = O(n^3)$. We make the first progress towards the conjecture of \cite{bubeck2016testing} by showing that the threshold of $O(n^3)$ decreases substantially for small $p$.

Our main results and the overall structure of the paper are as follows:
\begin{itemize}
\item We prove in Section \ref{sec:rig} that the random intersection graph converges in total variation to an Erd\H{o}s-R\'{e}nyi graph if $d = \tilde{\omega}(n^3)$. Furthermore, we show in the dense and sparse regimes of $p$, that these two random graphs have total variation $1 - o(1)$ if $d = o(n^3)$. This resolves an open problem in \cite{fill2000random}, \cite{rybarczyk2011equivalence} and \cite{kim2018total}.
\item In Section \ref{sec:rim}, we strengthen our argument for random intersection graphs to show that the matrix of random intersection sizes $|S_i \cap S_j|$ converges in total variation to a symmetric matrix with independent Poisson entries. This implies the first total variation convergence result for $\tau$-random intersection graphs. More precisely, our results show that the $\tau$-random intersection graph with edge density $p$ converges to $\mG(n, p)$ if $d = \omega(\tau^3 n^3)$, if $p$ is bounded away from $1$.
\item In Section \ref{sec:rgg}, we prove that if $d = \tilde{\omega}\left(\min\{ pn^3, p^2 n^{7/2} \} \right)$, then the random geometric graph on $\mathbb{S}^{d - 1}$ and edge density $p$ converges in total variation to $\mG(n, p)$. This marks the first progress towards the conjecture of \cite{bubeck2016testing} that the threshold $d = \omega(n^3)$ decreases drastically for $p$ polynomially small with respect to $n$.
\end{itemize}
The first result above was obtained simultaneously by Bubeck, Racz and Richey \cite{bubeck2019geometry}. While working on the convergence of sparse random geometric graphs to Erd\H{o}s-R\'{e}nyi graphs, we were informed by M. Racz that they had solved the problem for the random intersection graphs. We then found an alternative proof for this result during our work on the random geometric graphs, and extended our techniques to random intersection matrices and $\tau$-random intersection graphs. We have not seen any portion of their arguments so our solution is independent of theirs.

In Section \ref{subsec:relwork}, we discuss work related to our question for random intersection and geometric graphs. In Section \ref{sec:results}, we formally introduce the models we consider, our results, the techniques used to prove these results and several problems that remain open. Our proofs are a hybrid of combinatorial arguments, direct couplings and applications of information inequalities, which is novel to this problem. Previous upper bounds on the total variation distance between random graphs with latent geometry and $\mG(n, p)$ have typically not been both combinatorial and information-theoretic, while this interplay is essential to the sharpness of our bounds.

\subsection{Related Work}
\label{subsec:relwork}

The question of convergence of random intersection graphs to Erd\H{o}s-R\'{e}nyi random graphs in total variation was first examined in \cite{fill2000random}. In \cite{fill2000random}, it was shown that if $d = n^\alpha$ where $\alpha > 6$, then the two graphs converge in total variation. In the recent paper \cite{kim2018total}, this result was improved to show that convergence occurs as long as $d \gg n^4$. In \cite{rybarczyk2011equivalence}, a weaker property than convergence in total variation was shown to hold as long as $d = n^\alpha$ where $\alpha > 3$. In particular, \cite{rybarczyk2011equivalence} shows that, for any monotone property $\mathcal{A}$, the probabilities $\bP[G \in \mathcal{A}]$ are essentially the same regardless of whether $G$ is sampled from a random intersection graph or an Erd\H{o}s-R\'{e}nyi random graph at a matching edge density. We remark that this result differs from convergence in total variation because of the monotonicity requirement on $\mathcal{A}$. The first main result of our work is to strengthen these prior results by showing that total variation convergence occurs as long as $d = \tilde{\omega}(n^3)$, and that this is the best bound possible. We also extend our techniques to show that this total variation convergence occurs when $d = \tilde{\omega}(\tau^3 n^3)$ for $\tau$-random intersection graphs with edge density bounded away from $1$, yielding the first convergence result of this type for this model.

The convergence of high-dimensional random geometric graphs to Erd\H{o}s-R\'{e}nyi random graphs in total variation was first examined in \cite{devroye2011high}. In \cite{devroye2011high}, the authors identified the clique number of random geometric graphs and showed an asymptotic convergence result through central limit theorem-based methods -- specifically, they proved that the two graphs converge in total variation if $n$ is fixed and $d \to \infty$. \cite{bubeck2016testing} strengthened this result significantly, by showing that if $d = \omega(n^3)$ then the two graphs converge in total variation. Their main technique was to show that the adjacency matrices of random geometric graphs and Erd\H{o}s-R\'{e}nyi graphs can be approximately generated by thresholding the entries of an $n \times n$ Wishart matrix with $d$ degrees of freedom and an $n \times n$ $\pr{goe}$ matrix, respectively. By directly comparing their density functions on the set of symmetric matrices in $\mathbb{R}^{n \times n}$, \cite{bubeck2016testing} showed that these two random matrix ensembles converge in total variation if $d = \omega(n^3)$. \cite{bubeck2016testing} conjectured that, if the marginal edge density $p$ of the graphs tends to zero as a function of $n$, then the threshold of $d = \omega(n^3)$ should decrease drastically. Specifically, they conjectured that if $p = \Theta(1/n)$, then random geometric graphs converge to Erd\H{o}s-R\'{e}nyi random graphs in total variation as long as $d = \omega(\log^3 n)$.

We remark that the argument in \cite{bubeck2016testing} thresholding a pair of coupled Wishart and $\pr{goe}$ matrices breaks down as soon as $d = O(n^3)$, since these two matrix ensembles no longer converge in total variation. Our third main result overcomes this technical difficulty, making the first progress towards the conjecture of \cite{bubeck2016testing} by showing that convergence occurs as long as $d = \tilde{\omega}\left(\min\{ pn^3, p^2 n^{7/2} \} \right)$. The argument in \cite{bubeck2016testing} sparked a line of research examining the total variation convergence of Wishart and \pr{goe} matrices \cite{bubeck2016entropic, nourdin2018asymptotic, chetelat2019middle, racz2019smooth}. \cite{eldan2016information} extended the Erd\H{o}s-R\'{e}nyi total variation convergence result in \cite{bubeck2016testing} to anisotropic random geometric graphs. We also note that the same result on Wishart and \pr{goe} matrices as in \cite{bubeck2016testing} was obtained independently in \cite{jiang2015approximation}. An exposition on some of these results on the convergence of random geometric graphs and matrix ensembles can be found in \cite{racz2017basic}.

The general topic of showing total variation convergence between high-dimensional objects has emerged as a common technical problem in a number of areas. \cite{janson2010asymptotic} provides some initial general results on showing total variation convergence for pairs of random graph distributions. Total variation convergence often underlies information-theoretic lower bounds for detection and estimation problems in statistics \cite{mossel2015reconstruction, berthet2013optimal, abbe2015community, bresler2018optimal}. The total variation convergence of high-dimensional objects also is the principal technical content of average-case reductions between statistical problems \cite{berthet2013complexity, ma2015computational, hajek2015computational, wu2018statistical, brennan2018reducibility, brennan2019optimal, brennan2019universality, brennan2019average, brennan2020reducibility}. The recent reduction in \cite{brennan2019optimal} between the planted clique problem and sparse principal component analysis directly uses the random matrix ensemble convergence of \cite{bubeck2016testing} and \cite{jiang2015approximation} to construct efficient reductions.

\subsection{Notation}

Throughout, we let $G = ([n], E)$ be a simple graph on the vertex set $[n] = \{1, 2, \dots, n\}$. All other quantities, unless stated otherwise, will be viewed as functions of $n$. For example, $p = p(n)$ and $d = d(n)$ will typically denote edge density and latent dimension parameters associated with $G$. The asymptotic notation $\omega_n(\cdot), o_n(\cdot), \Omega_n(\cdot)$ and $O_n(\cdot)$ refers to its standard meaning with all quantities that are not functions of $n$ viewed as constant. The notation $\tilde{\omega}_n(\cdot)$ and $\tilde{O}_n(\cdot)$ are shorthands for $\omega_n(\cdot)$ and $O_n(\cdot)$, respectively, up to $\text{polylog}(n)$ factors. The inequalities $\ll$ and $\lesssim$ will serve as shorthands for $o_n(\cdot)$ and $O_n(\cdot)$, respectively. Throughout, equalities involving $O_n(\cdot)$ will be used to denote two-sided estimates of error terms. More precisely, the statement $A = B + O_n(C)$ will be a shorthand for $|A - B| = O_n(C)$. Given a random variable $X$, we let $\mL(X)$ denote its law. Given a measure $\nu$ over graphs $G$ and edge $e$, we let $\nu_{\sim e}$ denote the marginal measure of the graph restricted to edges other than $e$ and we let $\nu_{\sim e}^+$ denote the measure of the rest of the graph conditioned on the event $\{e \in E(G)\}$. We let $\mathbbm{1}(A)$ denote the indicator for an event $A$. Total variation, KL divergence and $\chi^2$ divergence are denoted as $\TV(\cdot, \cdot), \kl(\cdot \| \cdot)$ and $\chi^2(\cdot, \cdot)$, respectively. Given measures $\mu_1, \mu_2, \dots, \mu_n$ over a measurable space $(\mathcal{X}, \mathcal{B})$, then $\mu = \mu_1 \otimes \mu_2 \otimes \cdots \otimes \mu_n$ denotes the product measure with marginals $\mu_i$.

\section{Random Graphs with Latent Geometry}
\label{sec:results}

Now, we formally introduce the models we consider and state our main results. We remark that all of the graphs we consider -- random intersection graphs, $\tau$-random intersection graphs and random geometric graphs on $\mathbb{S}^{d - 1}$ -- can be viewed as specific instantiations of random inner product graphs. These are also referred to as dot product graphs in the literature and are defined as follows.

\begin{definition}[Random Inner Product Graphs]
Let $\mu$ be a measure on a set $\mathcal{H}$ equipped with a real-valued inner product $\langle \cdot, \cdot \rangle$ and let $X_1, X_2, \dots, X_n \sim_{\textnormal{i.i.d.}} \mu$. The random inner product graph $G$ over the vertex set $[n]$ is then constructed by connecting $i$ and $j$ if and only if $\langle X_i, X_j \rangle \ge t$ for some threshold $t \in \mathbb{R}$.
\end{definition}

\subsection{Random Intersection Graphs and Matrices}

An intersection graph of a sequence of sets is defined as follows. We remark that an intersection graph can be viewed as an inner product graph over the set $\{0, 1\}^d$ equipped with the inner product on $\mathbb{R}^d$ by identifying a set with its corresponding indicator vector.

\begin{definition}[Intersection Graph]
Given finite sets $S_1, S_2, \dots, S_n$, let $\pr{ig}_{\tau}(S_1, S_2, \dots, S_n)$ denote the graph $G$ on the vertex set $[n]$ where $\{i, j \} \in E(G)$ if and only if $|S_i \cap S_j| \ge \tau$.
\end{definition}

This leads to a natural distribution of random intersection graphs formed by constructing $n$ subsets of $[d]$ where each element of $[d]$ is included in each subset independently with a fixed probability $\delta$.

\begin{definition}[Random Intersection Graph]
Let $\pr{rig}(n, d, p)$ denote the distribution of the graph $\pr{ig}_1(S_1, S_2, \dots, S_n)$ where $S_1, S_2, \dots, S_n$ are random subsets of $[d]$ generated by including each element of $[d]$ in each $S_i$ independently with probability $\delta$ where $p = 1 - (1 - \delta^2)^d$.
\end{definition}

Here $p = 1 - (1 - \delta^2)^d$ corresponds to the marginal probability of an edge between two vertices in $\pr{rig}(n, d, p)$. We remark that our notation differs from the standard notation for random intersection graphs, which are typically parameterized directly by $\delta$ rather than their marginal edge density $p$. We choose the latter for consistency with our notation for random geometric graphs on $\mathbb{S}^{d - 1}$.

Below is our main theorem identifying conditions under which $\pr{rig}$ and $\mG(n, p)$ converge in total variation. This resolves an open problem in \cite{kim2018total}, \cite{fill2000random} and \cite{rybarczyk2011equivalence}. Bubeck, Racz and Richey independently proved the same result through alternate techniques simultaneous to this work \cite{bubeck2019geometry}.

\begin{theorem} \label{thm:introdenserig}
Suppose $p = p(n) \in (0, 1)$ and $d$ satisfy that
$$d \gg n^3 \left( 1 + \min \left\{ \log n, \log(1 - p)^{-1} \right\} \right)^3$$
Then it follows that
$$\TV\left( \pr{rig}(n, d, p), \mG(n, p) \right) \to 0 \quad \text{as } n \to \infty$$
\end{theorem}

For the sake of exposition, we first show that this theorem holds under the assumption that $1 - p = \Omega_n(n^{-1/2})$. The proof of the theorem under these conditions can be found in Section \ref{sec:rig}. The theorem when $1 - p = o_n(n^{-1/2})$ will be implied in Section \ref{sec:rim} by a later result on the convergence of random intersection matrices and Poisson matrices. The main ideas in the proof are as follows. A Poissonization argument yields that $G \sim \pr{rig}(n, d, p)$ can approximately be generated as the union of independently chosen cliques. Through the second-moment method, we obtain a tight upper bound on the total variation distance induced by planting a small clique in $\mG(n, p)$ to $\mG(n, p')$ where $p'$ is chosen so that the two models have matching edge densities. Another Poissonization step and then applying this bound inductively yields the desired convergence in total variation. Making this argument rigorous requires a number of additional technical steps.

We also show that, in sparse and dense regimes of edge densities $p$, the condition on the dimension $d$ in Theorem \ref{thm:introdenserig} is the best that can be hoped for. This follows by comparing the number of triangles and a signed variant of the number of triangles in each of $\pr{rig}(n, d, p)$ and $\mG(n, p)$. The signed triangle statistic we consider was introduced in \cite{bubeck2016testing} to show a similar theorem for random geometric graphs.

\begin{theorem} \label{thm:rig-triangles}
Suppose $p = p(n)$ satisfies that $1 - p = \Omega(1)$ and either $p = \Theta(1)$ or $p = \Theta(1/n)$. It follows that if $n^2 \ll d \ll n^3$, then
$$\TV\left( \pr{rig}(n, d, p), \mG(n, p) \right) \to 1 \quad \text{as } n \to \infty$$
\end{theorem}

Our second main result extends the proof of Theorem \ref{thm:introdenserig} to directly couple the full matrix of intersection sizes between the sets $S_i$ to a matrix with i.i.d. Poisson entries. More precisely, consider the following pair of random matrices.

\begin{definition}[Random Intersection Matrix]
Let $\pr{rim}(n, d, \delta)$ denote the distribution of $n \times n$ matrices $M$ with entries
$$M_{ij} = \left\{ \begin{array}{cl} |S_i \cap S_j| &\textnormal{if } i \neq j \\ 0 &\textnormal{otherwise} \end{array} \right.$$
where $S_1, S_2, \dots, S_n$ are random subsets of $[d]$ generated by including each element of $[d]$ in each $S_i$ independently with probability $\delta$.
\end{definition}

\begin{definition}[Poisson Matrix]
Given $\lambda \in \mathbb{R}_{\ge 0}$, let $\pr{poim}(n, \lambda)$ denote the distribution of symmetric $n \times n$ matrices $M$ such that $M_{ii} = 0$ for all $1 \le i \le n$ and $M_{ij}$ are i.i.d. $\textnormal{Poisson}(\lambda)$ for all $1 \le i < j \le n$.
\end{definition}

Our second main result coupling $\pr{rim}$ and $\pr{poim}$ can now be formally stated as follows. Its proof can be found in Section \ref{sec:rim}.

\begin{theorem} \label{thm:rim}
Suppose that $\delta = \delta(n) \in (0, 1)$ and $d$ satisfies that $d \gg n^3$ and $\delta \ll d^{-1/3} n^{-1/2}$. Then it holds that
$$\TV\left( \pr{rim}(n, d, \delta), \pr{poim}\left(n, d\delta^2\right) \right) \to 0 \quad \text{as } n \to \infty$$
\end{theorem}

Applying the data-processing inequality to these matrices at $\tau > 1$ now yields a natural extension of Theorem \ref{thm:introdenserig} to random intersection graphs defined at higher thresholds than $1$. These graphs are formally defined as follows.

\begin{definition}[Random Intersection Graphs at Higher $\tau$] \label{defn:righigh}
Let $\pr{rig}_\tau(n, d, p)$ denote the distribution of the graph $\pr{ig}_\tau(S_1, S_2, \dots, S_n)$ where $S_1, S_2, \dots, S_n$ are random subsets of $[d]$ generated by including each element of $[d]$ in each $S_i$ independently with probability $\delta$ where
$$1 - p = \sum_{k = 0}^{\tau - 1} \binom{d}{k} \delta^{2k}(1 - \delta^2)^{d - k}$$
\end{definition}

We obtain the following corollary of Theorem \ref{thm:rim} yielding conditions for the convergence of $\mG(n, p)$ and $\pr{rig}_\tau(n, d, p)$ in total variation.

\begin{corollary} \label{cor:higherthres}
Suppose $p = p(n) \in (0, 1)$, $\delta = \delta(n) \in (0, 1)$, $\tau \in \mathbb{Z}_+$ and $d$ satisfy that
$$1 - p = \sum_{k = 0}^{\tau - 1} \binom{d}{k} \delta^{2k}(1 - \delta^2)^{d - k}$$
Furthermore suppose that
$$d \gg n^3, \quad \delta \ll d^{-1/3} n^{-1/2} \quad \textnormal{and} \quad n^2 \delta^4 \ll p(1 - p)$$
Then it follows that
$$\TV\left( \pr{rig}_\tau(n, d, p), \mG(n, p) \right) \to 0 \quad \text{as } n \to \infty$$
\end{corollary}

When $1 - p$ is bounded below by a constant, this corollary can be restated with the simple condition of $d \gg \tau^3 n^3$ as shown below. This is our main result on the convergence of $\mG(n, p)$ and $\pr{rig}_\tau(n, d, p)$ in total variation.

\begin{corollary} \label{cor:morerig}
Suppose $p = p(n) \in (0, 1)$ satisfies that $1 - p = \Omega_n(1)$ and $d$ and $\tau = \tau(n) \in \mathbb{Z}_+$ satisfy $d \gg \tau^3 n^3$. Then it follows that
$$\TV\left( \pr{rig}_\tau(n, d, p), \mG(n, p) \right) \to 0 \quad \text{as } n \to \infty$$
\end{corollary}

Note that the condition $d \gg \tau^3 n^3$ is more restrictive at larger $\tau$. Qualitatively, this arises because of the relation between $p, \tau$ and $\delta$ in Definition \ref{defn:righigh}. If $p$ remains fixed and $\tau$ increases while satisfying that $\tau \ll d$, then $\delta$ must also increase. This can be seen by expressing the given relation as $1 - p = \bP[\text{Binom}(d, \delta^2) < \tau]$. At larger $\delta$, the conditions for the underlying $\pr{rim}$ and $\pr{poim}$ matrices to converge in Theorem \ref{thm:rim} are then stricter, leading to a more restrictive condition on $d$ in Corollary \ref{cor:morerig}. As we will discuss further in Section \ref{subsec:open}, it is unclear if the conditions in Corollaries \ref{cor:higherthres} and \ref{cor:morerig} are tight. Whether they can be improved or there is a statistic distinguishing between $\pr{rig}_\tau$ and $\mG(n, p)$ when these conditions are violated is a question left open by this work.

\subsection{Random Geometric Graphs on $\mathbb{S}^{d - 1}$}

A geometric graph of a sequence of points in $\mathbb{R}^d$ is defined as follows.

\begin{definition}[Geometric Graph]
Given $n$ points $X_1, X_2, \dots, X_n \in \mathbb{R}^d$ and a threshold $t \in \mathbb{R}$, let $\pr{gg}_{t}(X_1, X_2, \dots, X_n)$ denote the graph $G$ on the vertex set $[n]$ where $\{i, j \} \in E(G)$ if and only if $\langle X_i, X_j \rangle \ge t$.
\end{definition}

Note that when the points $X_1, X_2, \dots, X_n$ are on the sphere $\mathbb{S}^{d - 1}$, the inner product condition $\langle X_i, X_j \rangle \ge t$ is equivalent to $\| X_i - X_j \|_2^2 \le \tau = 2 - 2t$, yielding the standard definition of geometric graphs in which points close in $\ell_2$ distance are joined by an edge. This leads to a natural distribution of random geometric graphs by taking $X_1, X_2, \dots, X_n$ to be sampled independently and uniformly at random from the sphere $\mathbb{S}^{d - 1}$. 

\begin{definition}[Random Geometric Graph]
Let $\pr{rgg}(n, d, p)$ denote the distribution of the graph $\pr{gg}_{t_{p, d}}(X_1, X_2, \dots, X_n)$ where $X_1, X_2, \dots, X_n$ are sampled independently from the Haar measure on $\mathbb{S}^{d - 1}$ and $t_{p, d} \in \mathbb{R}$ is such that $\bP[\langle X_1, X_2 \rangle \ge t_{p, d}] = p$.
\end{definition}

Our main result on random geometric graphs is the following theorem, yielding the first progress towards a conjecture of \cite{bubeck2016testing} that the regime of parameters $(n, d, p)$ in which $\pr{rgg}(n, d, p)$ to $\mG(n, p)$ converge in total variation increases quickly as $p$ decays with $n$. This theorem also tightly recovers the result of \cite{bubeck2016testing} on the total variation convergence of $\pr{rgg}(n, d, p)$ to $\mG(n, p)$ in the dense regime when $p$ is constant.

\begin{theorem} \label{thm:rgg}
Suppose $p = p(n) \in (0, 1/2]$ satisfies that $p = \Omega_n(n^{-2} \log n)$ and
$$d \gg \min\left\{ pn^3 \log p^{-1}, p^2 n^{7/2} (\log n)^3 \sqrt{\log p^{-1}} \right\}$$
where $d$ also satisfies that $d \gg n \log^4 n$. Then it follows that
$$\TV\left( \pr{rgg}(n, d, p), \mG(n, p) \right) \to 0 \quad \text{as } n \to \infty$$
\end{theorem}

When $p = \Theta(n^{-\alpha})$ where $\alpha \in [0, 1]$, this yields convergence in total variation as long as $d = \tilde{\omega}(\min\{n^{3 - \alpha}, n^{7/2 - 2\alpha} \})$, yielding the first improvement over the $d = \omega(n^3)$ result of \cite{bubeck2016testing}. In particular, when $p = c/n$ where $c > 0$ is a constant, this theorem shows convergence in total variation if $d = \tilde{\omega}(n^{3/2})$. We remark that our argument still yields convergence results if $p = o_n(n^{-2} \log n)$. However, for the sake of maintaining a simple main theorem statement, we relegate these results to the propositions in the subsections of Section \ref{sec:rgg}.

The main ideas in the proof of Theorem \ref{thm:rgg} are as follows. We first reduce bounding the total variation between $\pr{rgg}(n, d, p)$ and $\mG(n, p)$ to bounding the expected value of the $\chi^2$ divergence between the conditional distribution $Q$ of an edge of $\pr{rgg}$ given the rest of the graph and $\text{Bern}(p)$. We then bound this $\chi^2$ divergence using two different coupling arguments. The first directly couples $X_1, X_2, \dots, X_n$ with a set of orthogonal vectors and independent random variables, approximately expressing the conditional distribution $Q$ in terms of one of these variables. This argument yields tight bounds in the regime of dense marginal edge probabilities $p$. The second argument reduces bounding this $\chi^2$ divergence to bounding the total variation between $\pr{rgg}$ with the edge $\{1, 2\}$ marginalized out and $\pr{rgg}$ conditioned on the presence of $\{1, 2\}$. This is then done by directly coupling to $X_1, X_2, \dots, X_n$ to $X_1, X_2', \dots, X_n$ where $X_2'$ is conditioned to be such that the edge $\{1, 2\}$ is present. This argument yields tighter bounds in the regime of sparse marginal edge probabilities $p$.

\subsection{Techniques and Information Inequalities}

In this section, we briefly review the key properties of the $f$-divergences $\TV(\cdot, \cdot), \kl(\cdot \| \cdot)$ and $\chi^2(\cdot, \cdot)$ used in our arguments. Given two measures $\nu$ and $\mu$ on a measurable space $(\mathcal{X}, \mathcal{B})$ where $\nu$ is absolutely continuous with respect to $\mu$, these divergences are given by
\begin{align*}
&\TV\left( \nu, \mu \right) = \frac{1}{2} \cdot \bE_{x \sim \mu} \left| \frac{d\nu}{d\mu}(x) - 1 \right|, \quad \quad \kl\left( \nu \| \mu \right) = \bE_{x \sim \mu} \left[ \frac{d\nu}{d\mu}(x) \cdot \log \frac{d\nu}{d\mu}(x) \right] \quad \text{and} \quad \\
&\chi^2\left( \nu, \mu \right) = \bE_{x \sim \mu} \left( \frac{d\nu}{d\mu}(x) - 1 \right)^2
\end{align*}
where $\tfrac{d\nu}{d\mu} : \mathcal{X} \to \mathbb{R}_{\ge 0}$ denotes the Radon-Nikodym derivative of $\nu$ with respect to $\mu$. A key property of these divergences is that they satisfy data-processing inequalities. Specifically, if $K$ is a Markov transition from the measurable space $(\mathcal{X}, \mathcal{B})$ to another measurable space $(\mathcal{X}', \mathcal{B}')$, then $\TV(K\nu, K\mu) \le \TV(\nu, \mu)$. Analogous inequalities hold for $\kl$ and $\chi^2$. The following inequalities also hold
$$2 \cdot \TV\left( \nu, \mu \right)^2 \le \kl(\nu \| \mu) \le \chi^2(\nu, \mu)$$
Note that the first inequality is Pinsker's inequality and the second is Cauchy-Schwarz. A survey of these inequalities and the relationships between different probability metrics can be found in \cite{gibbs2002choosing}. These divergences have different characterizations and properties that make them amenable to different contexts. Total variation satisfies the triangle inequality and is symmetric, while $\kl(\cdot \| \cdot)$ and $\chi^2(\cdot, \cdot)$ are not. Furthermore, total variation can alternatively be characterized in terms of couplings and differences in event probabilities as
$$\TV\left( \nu, \mu \right) = \sup_{S \in \mathcal{B}} \left| \bP_{\nu}[S] - \bP_{\mu}[S] \right| = \inf_{\rho \in \mathcal{C}} \bP_{(X, Y) \sim \rho}[X \neq Y]$$
where $\mathcal{C}$ is the set of couplings $(X, Y)$ over the product space $\mathcal{X} \times \mathcal{X}$ where $X \sim \nu$ and $Y \sim \mu$. When $\mathcal{X}$ is a product set $\mathcal{X} = \mathcal{S}^k$ and $\mu$ and $\nu$ are product measures with $\nu = \nu_1 \otimes \nu_2 \otimes \cdots \otimes \nu_k$ and $\mu = \mu_1 \otimes \mu_2 \otimes \cdots \otimes \mu_k$, then $\kl$ tensorizes with
$$\kl\left( \nu \| \mu \right) = \sum_{i = 1}^k \kl( \nu_i \| \mu_i)$$
A similar property holds when $\nu$ is not necessarily a product distribution. Given a measure $\nu$ on $(X_1, X_2, \dots, X_k) \in \mathcal{S}^k$, let $\nu_i$ denote the marginal measure of $X_i$ and $\nu_i( \cdot | x_{\sim i})$ denote the conditional measure of $X_i$ given $(X_j : j \neq i) = x_{\sim i}$ where $x_{\sim i} \in \mathcal{S}^{k - 1}$. When $\mu$, but not $\nu$, is a product measure then $\kl$ satisfies the following tensorization inequality, which will be a key part of our argument for random geometric graphs.

\begin{lemma}[See e.g. Lemma 3.4 in \cite{kontorovich2017concentration}]
\label{lem:kl_tensorization}
Suppose $\mu$ is a product measure on $\mathcal{S}^k$ with $\mu = \mu_1\otimes \mu_2 \otimes \dots \otimes \mu_k$, then it holds that
$$\kl(\nu||\mu) \leq \sum_{i=1}^{k}\mathbb{E}_{x  \sim \nu} \left[ \kl\bigr(\nu_i(\cdot|x_{\sim i})\bigr|\bigr|\mu_i\bigr) \right]$$
\end{lemma}

The divergences also have important properties related to mixtures. Suppose that $\nu = \bE_{s \sim \rho} \nu_s$ where $\rho$ is a distribution on elements of a set $T$ and $\{ \nu_s : s \in T \}$ is a collection of measures on $(\mathcal{X}, \mathcal{B})$ absolutely continuous with respect to $\mu$. Convexity of the divergences yield that $\TV(\nu, \mu) \le \bE_{s \sim \rho} [\TV(\nu_s, \mu)]$ and analogous inequalities hold for $\kl$ and $\chi^2$. A particularly useful property of $\chi^2$ divergence follows by applying Fubini's theorem to mixtures $\nu = \bE_{s \sim \rho} \nu_s$ as follows:
$$1 + \chi^2(\nu, \mu) = \bE_{x \sim \mu} \left[ \bE_{s \sim \rho} \left[ \frac{d\nu_s}{d\mu}(x) \right]^2 \right] = \bE_{(s, s') \sim \rho \otimes \rho} \left[ \bE_{x \sim \mu} \left[ \frac{d\nu_s}{d\mu}(x) \cdot \frac{d\nu_{s'}}{d\mu}(x) \right] \right]$$
This is the main idea behind the second moment method and will be crucial in our arguments for random intersection graphs and matrices. Furthermore, if $\mu = \bE_{s \sim \rho'} \mu_s$, then we have the following conditioning property of total variation
$$\TV(\nu, \mu) \le \TV(\rho, \rho') + \bE_{s \sim \rho} [\TV(\nu_s, \mu_s)]$$
If $E$ is an event on $\mathcal{X}$ and $\nu_E$ is the distribution of $\nu$ given $E$, then we also have the conditioning property that $\TV(\nu_E, \nu) = \bP_\nu[E^c]$. A final useful property is that if the Radon-Nikodym derivative is controlled, then all of these divergences can be bounded from above in terms of one another. For example, if $|\tfrac{d\nu}{d\mu} - 1|$ is upper bounded by $C_1$ on an event $E$ and is upper bounded by $C_2$ in general, then it holds that
$$\chi^2(\nu, \mu) \le 2C_1 \cdot \TV(\nu, \mu) + C_2^2 \cdot \bP_\mu[E^c]$$
This allows $\chi^2$ to be upper bounded in terms of both concentration of $\tfrac{d\nu}{d\mu}$ and a coupling of $\nu$ and $\mu$, which will be essential to our arguments for random geometric graphs in the sparse case.

\subsection{Open Problems and Conjectures}
\label{subsec:open}

As previously mentioned, it is unclear if the conditions in Corollaries \ref{cor:higherthres} and \ref{cor:morerig} are tight. Given the similarities between the proofs of Theorems \ref{thm:introdenserig} and \ref{thm:rim}, and the fact that triangles and signed triangles certify that the conditions in Theorem \ref{thm:introdenserig} are tight in certain regimes of $p$, it is possible that triangles and signed triangle identify the optimal conditions on $d$ needed for the convergence in Corollary \ref{cor:morerig}. However, analyzing these statistics for ordinary $\pr{rig}$ to prove the relatively simple Theorem \ref{thm:rig-triangles} is even computationally involved. Carrying out similar computations for $\pr{rig}_\tau$ seems as though it would be substantially more difficult.

Another outstanding problem related to Corollary \ref{cor:morerig} concerns the convergence of $\tau$-random intersection graphs and $\mG(n, p)$ in the large $\tau$ regime. A limiting case of our argument is when $\tau$ is set to be $\tau = \Theta_n(d^{1/3 - \kappa})$, in which case our condition reduces to $d \gg n^{1/\kappa}$. In particular, our argument fails to show any convergence if $\tau = \Omega_n(d^{1/3})$. This leads to the following question left open by this work.

\begin{question}
For what parameters $(\tau, n, d, p)$ do $\pr{rig}_\tau(n, d, p)$ and $\mG(n, p)$ converge in total variation if $\tau = \Omega_n(d^{1/3})$?
\end{question}

We suspect that improving on the condition $d \gg \tau^3 n^3$ in Corollary \ref{cor:morerig} would require a substantially different argument for bounding the total variation between $\pr{rig}_\tau$ and $\mG(n, p)$. We conjecture that for $\tau = d/4 + O_n(\sqrt{d \log n})$ and $\delta = 1/2$, $\tau$-random intersection graphs should behave approximately like random geometric graphs on the sphere $\mathbb{S}^{d - 1}$. A direction for future work is to directly compare these two models in total variation distance, showing that they can converge to one another even when they are both far from Erd\H{o}s-R\'{e}nyi. Another open problem is the optimal dependence on $p$ in the phase transition for detecting geometry in random geometric graphs on the sphere $\mathbb{S}^{d - 1}$. In particular, the following conjecture of \cite{bubeck2016testing} about this dependence when $p = \Theta(1/n)$ remains open.

\begin{conjecture}[\cite{bubeck2016testing}]\label{conj:BDER}
If $c > 0$ is a constant and $d \gg \log^3 n$, then it follows that
$$\TV\left( \pr{rgg}(n, d, c/n), \mG(n, c/n) \right) \to 0 \quad \text{as } n \to \infty$$
\end{conjecture}
 
 Even showing total variation convergence at any $p = \Omega_n(1/n)$ for some $d < n$ seems like a technically challenging open problem, given that all known techniques require that the Wishart matrix $W_{ij} = \langle X_i, X_j\rangle$ of the latent points $X_1, X_2, \dots, X_n$ is non-singular. A first question to answer is as follows.

\begin{question}
Is there a parameter $d = d(n) \ll n$ such that $\pr{rgg}(n, d, c/n)$ and $\mG(n, c/n)$ converge in total variation for any constant $c > 0$?
\end{question}

\section{Random Intersection Graphs}
\label{sec:rig}

The purpose of this section is to identify the regime of parameters $(n, d, p)$ in which $\pr{rig}(n, d, p)$ converges in total variation to $\mG(n, p)$. The following theorem identifies the regime in which the two random graphs do not converge for the main edge densities $p$ of interest. This theorem is a restatement of Theorem \ref{thm:rig-triangles} from the previous section. Its proof involves analyzing the triangle and signed triangle counts in $\pr{rig}(n, d, p)$. Further discussion of this result and its proof are in Section \ref{subsec:signedtriangles} and Appendix \ref{subsec:rig-triangles}. 

\begin{reptheorem}{thm:rig-triangles}
Suppose $p = p(n)$ satisfies that $1 - p = \Omega(1)$ and either $p = \Theta(1)$ or $p = \Theta(1/n)$. It follows that if $n^2 \ll d \ll n^3$, then
$$\TV\left( \pr{rig}(n, d, p), \mG(n, p) \right) \to 1 \quad \text{as } n \to \infty$$
\end{reptheorem}

The main purpose of this and the next two subsections is to prove the following theorem, identifying conditions under which the two graph distributions converge in total variation. This theorem is Theorem \ref{thm:introdenserig} in the case where $1 - p = \Omega_n(n^{-1/2})$, which captures the main regime of interest. This theorem resolves an open problem in \cite{kim2018total}, \cite{fill2000random} and \cite{rybarczyk2011equivalence}. Bubeck, Racz and Richey independently proved the same result through alternate techniques simultaneous to this work \cite{bubeck2019geometry}. Subsequent sections will also introduce the techniques we use to show the stronger equivalence of random intersection matrices and random Poisson matrices with i.i.d. entries in the following section. 

\begin{theorem} \label{thm:denserig}
Suppose $p = p(n) \in (0, 1)$ satisfies $1 - p = \Omega_n(n^{-1/2})$ and $d$ satisfies that
$$d \gg n^3 \left( 1 + \log(1 - p)^{-1} \right)^3$$
Then it follows that
$$\TV\left( \pr{rig}(n, d, p), \mG(n, p) \right) \to 0 \quad \text{as } n \to \infty$$
\end{theorem}

The proof of this theorem will be extended in Section \ref{sec:rim} to show that the entire random matrix of intersection sizes $|S_i \cap S_j|$ converges to a Poisson matrix with independent entries. This more general result will then imply an analogue of Theorem \ref{thm:denserig} for random intersection graphs at higher thresholds $\tau$. It will also complete the proof of Theorem \ref{thm:introdenserig}. Throughout the proof of this Theorem \ref{thm:denserig}, let $L = 1 + \log(1 - p)^{-1}$ and note that $1 < L = O_n(\log n)$. The proof approximately proceeds as follows:
\begin{enumerate}
\item The distribution $\pr{rig}(n, d, p)$ can be viewed as a union of cliques, each corresponding to an element of $[d]$. A Poissonization argument yields that $G \sim \pr{rig}(n, d, p)$ can approximately be generated as the union of a Poisson number of independently chosen cliques of small sizes.
\item We show that planting a triangle in an Erd\H{o}s-R\'{e}nyi graph of edge density $p$ yields a graph within total variation $O_n(p^{-3/2} n^{-3/2})$ of an Erd\H{o}s-R\'{e}nyi graph of an appropriately updated edge density. The expected number of triangles planted in the process described in the previous step is $O_n(n^3 p^{3/2} L^{3/2} d^{-1/2})$.
\item An induction now shows that adding in these triangles induces a $O_n(n^{3/2} L^{3/2} d^{-1/2}) = o_n(1)$ total variation distance from a mixture of Erd\H{o}s-R\'{e}nyi graphs. A similar argument applies to larger cliques, whose contribution to this total variation distance ends up being smaller.
\item Directly comparing their edge counts shows that the resulting mixture of Erd\H{o}s-R\'{e}nyi graphs is close to $\mG(n, p)$, and the theorem then follows from the triangle inequality.
\end{enumerate}
However, making this argument rigorous requires a number of additional technical steps. In the next section, we establish the bound in Step 2 above -- we obtain tight bounds on the optimal error probability of testing for a small planted clique in a density-corrected Erd\H{o}s-R\'{e}nyi graph. This will serve as a key technical component in our proof of Theorem \ref{thm:denserig}. We remark that the fact that the triangles in Step 3 have the largest contribution to the resulting total variation distance intuitively is consistent with the fact that analyzing triangle and signed triangle counts suffices to prove Theorem \ref{thm:rig-triangles}.

\subsection{Testing for Planted Cliques in Density-Corrected Random Graphs}
\label{subsec:plantingcliques}

We first observe that each element of $[d]$ forces a clique on the vertices whose sets it is a member of. Thus an alternative view of a random intersection graph is as a union of $d$ randomly chosen cliques. A precise description of this union and the numbers of cliques of each size is given later in this section. It is natural to consider whether forcing a randomly chosen clique on an Erd\H{o}s-R\'{e}nyi random graph yields a graph distribution close to some other Erd\H{o}s-R\'{e}nyi random graph. Obtaining a tight bound on the total variation distance between the resulting distributions is the content of the next lemma. Let $\mG(n, t, q)$ denote the planted clique distribution, generated by:
\begin{enumerate}
\item Sampling a graph $G \sim \mG(n, q)$, and then
\item Choosing $t$ vertices uniformly at random from $[n]$ and inserting a clique on these vertices.
\end{enumerate}
A key component of our method is the following precise bound on the total variation between $\mG(n, t, q)$ and the Erd\H{o}s-R\'{e}nyi random graph with matching edge density, which we obtain by a careful estimate of the corresponding $\chi^2$ divergence between these two distributions through the second moment method.

\begin{lemma} \label{lem:pcdist}
For any constant $t \ge 2$ and $q = q(n) \in \left(t^4 n^{-2}, 1 \right)$, it holds that
\begin{align*}
&\TV\left( \mG(n, t, q), \mG\left(n, q + (1 - q) \binom{t}{2} \binom{n}{2}^{-1} \right) \right) = O_n\left(q^{-1/2} n^{-3/2} + q^{-1} n^{-2} + \max_{2 < k \le t} q^{-\frac{1}{2}\binom{k}{2}} n^{-k/2} \right)
\end{align*}
\end{lemma}

\begin{proof}
Let $\tau = (1 - q) \binom{t}{2} \binom{n}{2}^{-1}$ and $p = q + \tau$. Observe that $p - q = \tau = O_n(n^{-2})$ and hence that $p = \Theta_n(q)$ since $q > t^4 n^{-2}$. Furthermore, it follows that
\begin{equation} \label{eq1}
1 < \frac{1 - q}{1 - p} = \frac{1}{1 - \binom{t}{2} \binom{n}{2}^{-1}} = 1 + O_n\left(n^{-2}\right)
\end{equation}
and thus $1 - p = \Theta_n(1 - q)$. Given a set $S \subseteq [n]$, let $\mG(n, S, q)$ denote the graph distribution formed by planting a clique on the vertices of $S$ and including each other edge independently with probability $q$. Let $\mathcal{U}_t$ denote the uniform distribution on the size $t$ subsets of $[n]$ and note that $\mG(n, t, q) =_d \bE_{S \sim \mathcal{U}_t} \mG(n, S, q)$. Let $\mathcal{S}_n$ denote the set of all simple undirected graphs on the labelled vertex set $[n]$ and observe that
\allowdisplaybreaks
\begin{align*} 
1 + \chi^2(\mG(n, t, q), \mG(n, p)) &= \sum_{G \in \mathcal{S}_n} \frac{\bP[\mG(n, t, q) = G]^2}{\bP[\mG(n, p) = G]} = \sum_{G \in \mathcal{S}_n} \frac{\bE_{S \sim \mathcal{U}_t}\left[ \bP[\mG(n, S, q) = G] \right]^2}{\bP[\mG(n, p) = G]} \\
&\quad \quad = \bE_{S, T \sim \mathcal{U}_t} \left[\sum_{G \in \mathcal{S}_n} \frac{\bP[\mG(n, S, q) = G] \cdot \bP[\mG(n, T, q) = G]}{\bP[\mG(n, p) = G]} \right]
\end{align*}
where the last equality holds by linearity of expectation and because $S$ and $T$ are independent. Since $\mG(n, S, q)$, $\mG(n, T, q)$ and $\mG(n, p)$ are product distributions, the above quantity is equal to
\allowdisplaybreaks
\begin{align*} 
&\bE_{S, T \sim \mathcal{U}_t} \left[ \prod_{e \in \binom{[n]}{2}} \left( \frac{\bP[e \not \in E(\mG(n, S, q))] \cdot \bP[e \not \in E(\mG(n, T, q))]}{\bP[e \not \in E(\mG(n, p))]} \right. \right. \\
&\quad \quad \quad \quad \quad \quad \quad \quad \quad \quad \quad \left. \left. + \frac{\bP[e \in E(\mG(n, S, q))] \cdot \bP[e \in E(\mG(n, T, q))]}{\bP[e \in E(\mG(n, p))]} \right) \right] \\
&\quad \quad = \bE_{S, T \sim \mathcal{U}_t} \left[ p^{-\binom{|S \cap T|}{2}} \left( \frac{q}{p} \right)^{2 \binom{t}{2} - 2\binom{|S \cap T|}{2}} \left( \frac{(1 - q)^2}{1 - p} + \frac{q^2}{p} \right)^{\binom{n}{2} - 2\binom{t}{2} + \binom{|S \cap T|}{2}} \right] \\
&\quad \quad = \bE_{S, T \sim \mathcal{U}_t} \left[ p^{-\binom{|S \cap T|}{2}} \left( 1 - \frac{\tau}{p} \right)^{2 \binom{t}{2} - 2\binom{|S \cap T|}{2}} \left( 1 + \frac{\tau^2}{p(1 - p)} \right)^{\binom{n}{2} - 2\binom{t}{2} + \binom{|S \cap T|}{2}} \right] \numberthis \label{eq2}
\end{align*}
Now fix two subsets $S, T \subseteq [n]$ of size $t$ and note that $|S \cap T| \le t = O_n(1)$. If $N_1 = 2 \binom{t}{2} - 2\binom{|S \cap T|}{2}$, then
\allowdisplaybreaks
\begin{align*}
\left( 1 - \frac{\tau}{p} \right)^{N_1} &= \sum_{k = 0}^{N_1} \binom{N_1}{k} (-1)^k \left( \frac{\tau}{p} \right)^k = 1 - \frac{N_1 \tau}{p} + O_n(q^{-2} n^{-4}) \\
&= 1 - \frac{(1 - p) \cdot \left( 2 \binom{t}{2} - 2\binom{|S \cap T|}{2} \right) \binom{t}{2}}{p \binom{n}{2}} - \frac{(p - q) \cdot \left( 2 \binom{t}{2} - 2\binom{|S \cap T|}{2} \right) \binom{t}{2}}{p \binom{n}{2}} + O_n(q^{-2} n^{-4}) \\
&= 1 - \frac{(1 - p) \cdot \left( 2 \binom{t}{2} - 2\binom{|S \cap T|}{2} \right) \binom{t}{2}}{p \binom{n}{2}} + O_n(q^{-2} n^{-4}) \numberthis \label{eq3}
\end{align*}
where the second equality follows from the fact that each summand with $k \ge 2$ is $O_n(q^{-2}n^{-4})$ since $\tau/p = O_n(q^{-1} n^{-2}) = O_n(1)$, and the sum has $O_n(1)$ many summands. Now let $N_2 = \binom{n}{2} - 2\binom{t}{2} + \binom{|S \cap T|}{2} \le \binom{n}{2}$ since $|S \cap T| \le t$. Observe that for sufficiently large $n$, we have that
\begin{equation} \label{eq4}
\frac{\binom{n}{2} \tau^2}{p(1 - p)} = \frac{(1 - q)^2}{1 - p} \cdot \frac{\binom{t}{2}^{2} \binom{n}{2}^{-1}}{p} \le \left( 1 + O_n\left(n^{-2} \right) \right) \cdot \frac{\binom{t}{2}^2 \binom{n}{2}^{-1}}{q} < \frac{1}{2}
\end{equation}
by Equation \ref{eq1}, the fact that $q > t^4 n^{-2}$ and the fact that $t = O_n(1)$. Furthermore, these inequalities also show that this quantity is $O_n(q^{-1} n^{-2})$. Now note that
\allowdisplaybreaks
\begin{align*}
\left( 1 + \frac{\tau^2}{p(1 - p)} \right)^{N_2} &= \sum_{k = 0}^{N_2} \binom{N_2}{k} \left( \frac{\tau^2}{p(1 - p)} \right)^k = 1 + \frac{N_2 \tau^2}{p(1 - p)} + O_n\left(q^{-2} n^{-4}\right) \\
&= 1 + \frac{(1 - p) \cdot \binom{t}{2}^2}{p \binom{n}{2}} - \frac{(1 - p) \cdot \left( 2\binom{t}{2} - \binom{|S \cap T|}{2} \right) \binom{t}{2}^2}{p \binom{n}{2}^2} \\
&\quad \quad + \frac{\left[(1 - q)^2 - (1 - p)^2 \right] \cdot \left( \binom{n}{2} - 2\binom{t}{2} + \binom{|S \cap T|}{2} \right) \binom{t}{2}^2}{p(1 - p) \binom{n}{2}^2} + O_n\left(q^{-2} n^{-4}\right) \\
&= 1 + \frac{(1 - p) \cdot \binom{t}{2}^2}{p \binom{n}{2}} + O_n\left(q^{-2} n^{-4}\right) \numberthis \label{eq5}
\end{align*}
where $(1 - q)^2 - (1 - p)^2 = 2\tau (1 - q) - \tau^2 = O_n(n^{-2})$. The second equality follows from the following inequality
\begin{align*}
0 < \sum_{k = 2}^{N_2} \binom{N_2}{k} \left( \frac{\tau^2}{p(1 - p)} \right)^k &\le \sum_{k = 2}^{N_2} \binom{n}{2}^k \left( \frac{\tau^2}{p(1 - p)} \right)^k \\
&\le \left( \frac{\binom{n}{2} \tau^2}{p(1 - p)} \right)^2 \left[ 1 - \left( \frac{\binom{n}{2} \tau^2}{p(1 - p)} \right) \right]^{-1} = O_n\left(q^{-2} n^{-4}\right)
\end{align*}
where the third inequality above follows from Equation \ref{eq4}. Multiplying the approximations in Equations \ref{eq3} and \ref{eq5} and simplifying yields that
\begin{align*}
&p^{-\binom{|S \cap T|}{2}} \left( 1 - \frac{\tau}{p} \right)^{2 \binom{t}{2} - 2\binom{|S \cap T|}{2}} \left( 1 + \frac{\tau^2}{p(1 - p)} \right)^{\binom{n}{2} - 2\binom{t}{2} + \binom{|S \cap T|}{2}} \\
&\quad \quad = p^{-\binom{|S \cap T|}{2}} \left(1 - \frac{(1 - p) \cdot \left( \binom{t}{2} - 2\binom{|S \cap T|}{2} \right) \binom{t}{2}}{p \binom{n}{2}} \right) + O_n\left( q^{-2 - \binom{|S \cap T|}{2}} n^{-4}\right) \numberthis \label{eq6}
\end{align*}
since $|S \cap T| \le t = O_n(1)$. Now observe that if $S, T \sim \mathcal{U}_t$ and are independent then $|S \cap T|$ is distributed as $\text{Hypergeometric}(n, t, t)$ and, in particular, it holds that $\bP[|S \cap T| = k] = \binom{t}{k} \binom{n - t}{t - k} \binom{n}{t}^{-1} = O_n(n^{-k})$. Furthermore, observe that the first term above is $O_n\left(q^{-\binom{|S \cap T|}{2}}\right)$ since $p = \Omega_n(n^{-2})$, $p = \Theta_n(q)$ and $|S \cap T| = O_n(1)$. Combining these estimates with the formula for $\chi^2(\mG(n, t, q), \mG(n, p))$ in Equation \ref{eq2} yields that
\begin{align*}
&\chi^2(\mG(n, t, q), \mG(n, p)) \\
&\quad \quad = \sum_{k = 0}^t \frac{\binom{t}{k} \binom{n - t}{t - k}}{\binom{n}{t}} \cdot \left[p^{-\binom{k}{2}} \left(1 - \frac{(1 - p) \cdot \left( \binom{t}{2} - 2\binom{k}{2} \right) \binom{t}{2}}{p \binom{n}{2}} \right) - 1 + O_n\left( q^{-2 - \binom{k}{2}} n^{-4}\right)\right] \\
&\quad \quad = \sum_{k = 0}^2 \frac{\binom{t}{k} \binom{n - t}{t - k}}{\binom{n}{t}} \cdot \left[p^{-\binom{k}{2}} \left(1 - \frac{(1 - p) \cdot \left( \binom{t}{2} - 2\binom{k}{2} \right) \binom{t}{2}}{p \binom{n}{2}} \right) - 1 \right] + O_n\left(q^{-2} n^{-4}\right) \\
&\quad \quad \quad \quad + O_n\left(q^{-3} n^{-6}\right) + O_n\left( \left(1 + q^{-2} n^{-4} \right) \max_{2 < k \le t} q^{-\binom{k}{2}} n^{-k} \right) \\
&\quad \quad = - \left[ \frac{\binom{n - t}{t}}{\binom{n}{t}} + \frac{t \binom{n - t}{t - 1}}{\binom{n}{t}} \right] \cdot \frac{(1 - p) \cdot \binom{t}{2}^2}{p \binom{n}{2}} + \frac{\binom{t}{2} \binom{n - t}{t - 2}}{\binom{n}{t}} \cdot \left[p^{-1} \left(1 - \frac{(1 - p) \cdot \left( \binom{t}{2} - 2 \right) \binom{t}{2}}{p \binom{n}{2}} \right) - 1 \right] \\
&\quad \quad \quad \quad + O_n\left(q^{-2} n^{-4}\right) + O_n\left( \max_{2 < k \le t} q^{-\binom{k}{2}} n^{-k} \right) \numberthis \label{eq7}
\end{align*}
Note that the second equality holds since there are $O_n(1)$ summands and the last equality since $q = \Omega_n(n^{-2})$. Now note that $\binom{n - t}{t} + t \binom{n - t}{t - 1} - \binom{n}{t} = O_n(n^{t - 2})$ and therefore
\begin{equation} \label{eq8}
\left[ \frac{\binom{n - t}{t}}{\binom{n}{t}} + \frac{t \binom{n - t}{t - 1}}{\binom{n}{t}} \right] \cdot \frac{(1 - p) \cdot \binom{t}{2}^2}{p \binom{n}{2}} = \frac{(1 - p) \cdot \binom{t}{2}^2}{p \binom{n}{2}} + O_n\left(q^{-1} n^{-4}\right)
\end{equation}
Furthermore
\begin{equation} \label{eqnbin}
\frac{\binom{n - t}{t - 2}}{\binom{n}{t}} - \frac{\binom{t}{2}}{\binom{n}{2}} = \frac{\binom{t}{2}}{\binom{n}{2} \cdot \frac{(n - 2)(n - 3) \cdots (n - t + 1)}{(n - t)(n - t -1) \cdots (n - 2t + 3)}}  - \frac{\binom{t}{2}}{\binom{n}{2}} = O_n(n^{-3})
\end{equation}
Therefore it follows that
\begin{align*}
&\frac{\binom{t}{2} \binom{n - t}{t - 2}}{\binom{n}{t}} \cdot \left[p^{-1} \left(1 - \frac{(1 - p) \cdot \left( \binom{t}{2} - 2 \right) \binom{t}{2}}{p \binom{n}{2}} \right) - 1 \right] \\
&\quad \quad = \frac{\binom{t}{2} \binom{n - t}{t - 2}}{\binom{n}{t}} \cdot \frac{(1 - p)}{p} + O_n\left(q^{-2} n^{-4}\right) \\
&\quad \quad = \frac{(1 - p) \cdot \binom{t}{2}^2}{p \binom{n}{2}} + O_n\left(q^{-1}n^{-3} + q^{-2} n^{-4} \right) \numberthis \label{eq9}
\end{align*}
Combining the estimates in Equations \ref{eq7}, \ref{eq8} and \ref{eq9} yields that
$$\chi^2(\mG(n, t, q), \mG(n, p)) = O_n\left( q^{-1} n^{-3} + q^{-2} n^{-4} + \max_{2 < k \le t} q^{-\binom{k}{2}} n^{-k} \right)$$
Now applying Cauchy-Schwarz yields that
\begin{align*}
\TV(\mG(n, p), \mG(n, t, q)) &\le \sqrt{\frac{1}{2} \cdot \chi^2(\mG(n, t, q), \mG(n, p))} \\
&= O_n\left(q^{-1/2} n^{-3/2} + q^{-1} n^{-2} + \max_{2 < k \le t} q^{-\frac{1}{2}\binom{k}{2}} n^{-k/2} \right)
\end{align*}
which completes the proof of the lemma.
\end{proof}

\subsection{Proof of Theorem \ref{thm:denserig}}

Having established the bound in Lemma \ref{lem:pcdist}, we now proceed to prove Theorem \ref{thm:denserig}.
We first adapt a Poissonization argument from \cite{kim2018total} and \cite{rybarczyk2011equivalence} in order to apply Lemma \ref{lem:pcdist}. Observe that given some element $i \in [d]$, the number of sets $S_j$ containing $i$ is distributed as $\text{Bin}(n, \delta)$. In other words, the number of vertices in the clique forced by element $i$ is distributed independently as $\text{Bin}(n, \delta)$ for each $i$. Furthermore, conditioned on the number of vertices in the clique forced by $i$, this clique is distributed uniformly at random over all subsets of $[n]$ of that size. Now for each $k \le n$, let $M_k$ denote the number of $i \in [d]$ with $| \{ j : i \in S_j\}| = k$. It follows that $(M_0, M_1, \dots, M_n)$ is distributed as a multinomial distribution with $d$ trials and probabilities $(p_0, p_1, \dots, p_n)$ where $p_k = \bP[| \{ j : i \in S_j\}| = k] = \binom{n}{k} \delta^k (1 - \delta)^{n - k}$. This implies that the marginals of the $M_k$ are distributed as $\text{Bin}(d, p_k)$. This view yields the following alternative procedure for generating a sample from $\pr{rig}(n, d, p)$:
\begin{enumerate}
\item Sample $(M_0, M_1, \dots, M_n) \sim \text{Multinomial}(d, p_0, p_1, \dots, p_n)$ and initialize $G$ to be the empty graph with $V(G) = [n]$; and
\item For each $2 \le k \le n$: independently sample a subset of size $k$ from $[n]$ a total of $M_k$ times and plant a clique in $G$ on each of these subsets.
\end{enumerate}
Consider instead generating $(M_2, M_3, \dots, M_n)$ as follows: sample $X \sim \text{Poisson}(d(1 - p_0 - p_1))$ and then generating $(M_2, M_3, \dots, M_n) \sim \text{Multinomial}(X, \gamma p_2, \gamma p_3, \dots, \gamma p_n)$ where $\gamma = (1 - p_0 - p_1)^{-1}$. Applying Step 2 to the tuple $(M_2, M_3, \dots, M_n)$ generated in this way induces a distribution $\pr{rig}_P(n, d, p)$ on the generated graph $G$. We now show that this distribution is close to $\pr{rig}(n, d, p)$. We remark that the next proposition is only slightly different from Proposition 3.2 in \cite{kim2018total} and Lemma 5 in \cite{rybarczyk2011equivalence}.

\begin{proposition} \label{prop:poissonization}
If $d \gg n^2 \log (1 - p)^{-1}$ as $n \to \infty$, then it holds that
$$\TV\left( \pr{rig}(n, d, p), \pr{rig}_P(n, d, p) \right) = O_n \left( \frac{n^2 \log(1 - p)^{-1}}{d} \right)$$
\end{proposition}

\begin{proof}
Observe that the marginal of $(M_0, M_1, \dots, M_n) \sim \text{Multinomial}(d, p_0, p_1, \dots, p_n)$ on the variables $(M_2, M_3, \dots, M_n)$ can also be generated by first generating $Y \sim \text{Bin}(d, 1 - p_0 - p_1)$ and then generating $(M_2, M_3, \dots, M_n) \sim \text{Multinomial}(Y, \gamma p_2, \gamma p_3, \dots, \gamma p_n)$. Since the distributions of $\pr{rig}(n, d, p)$ conditioned on $Y = z$ and $\pr{rig}_P(n, d, p)$ conditioned on $X = z$ are equal for any $0 \le z \le d$, it follows by the conditioning property of total variation that
\begin{align*}
\TV\left( \pr{rig}(n, d, p), \pr{rig}_P(n, d, p) \right) &= \TV\left( \mL(X), \mL(Y) \right) \\
&= \TV\left( \text{Poisson}(d(1 - p_0 - p_1)), \text{Bin}(d, 1 - p_0 - p_1) \right)
\end{align*}
Now note that $\delta^2 = 1 - (1 - p)^{1/d}$ satisfies that
\begin{equation} \label{eqndelta}
\frac{p}{d} \le \delta^2 = 1 - (1 - p)^{1/d} \le \frac{\log(1 - p)^{-1}}{d}
\end{equation}
The lower bound follows from Bernoulli's inequality and the upper bound follows from rearranging $(1 - x/d)^d \le e^{-x}$ applied with $x = \log(1 - p)^{-1}$. Therefore $\delta = O_n\left(\sqrt{\log(1 - p)^{-1}/d}\right)$ and thus $n^2 \delta^2 \ll 1$. Now by Theorem 2.1 in \cite{barbour1989some}, we have that
\begin{align*}
\TV\left( \text{Poisson}(d(1 - p_0 - p_1)), \text{Bin}(d, 1 - p_0 - p_1) \right) &\le 1 - p_0 - p_1 = \sum_{k = 2}^n p_k \\
&= \sum_{k = 2}^n \binom{n}{k} \delta^k (1 - \delta)^{n - k} \le \sum_{k = 2}^n n^k \delta^k = O(n^2 \delta^2)
\end{align*}
which completes the proof of the proposition.
\end{proof}

Let $\mL_P$ denote the law of the $M_i$ used to generate $\pr{rig}_P(n, d, p)$. In the remainder of this section, we will let $(M_2, M_3, \dots, M_n)$ denote a sample from $\mL_P$. Furthermore, let $G_2$ denote the graph $G$ after Step 2 above has been applied with only $k = 2$ in the process of generating $\pr{rig}_P(n, d, p)$. In other words, $G_2$ is generated by planting edges on $M_2$ randomly chosen edges. We now reap the benefits of this Poissonization argument by applying Poisson splitting in two separate cases.
\begin{enumerate}
\item Under $\mL_P$, we have that $(M_2, M_3, \dots, M_n) \sim \text{Multinomial}(X, \gamma p_2, \gamma p_3, \dots, \gamma p_n)$ where $X \sim \text{Poisson}(d(1 - p_0 - p_1))$. Poisson splitting implies that $M_i$ is distributed as $\text{Poisson}(dp_i)$ and that $M_2, M_3, \dots, M_n$ are independent.
\item Let $X_{\{i, j\}}$ denote the number of times the edge $\{i, j\}$ is planted in $G$ in the part of Step 2 where $k = 2$. Then the $X_{ij}$ are distributed as a multinomial distribution with $M_2$ trials and $\binom{n}{2}$ categories, each with a probability $\binom{n}{2}^{-1}$ of success. Poisson splitting implies that
$$X_{\{i, j\}} \sim_{\text{i.i.d.}} \text{Poisson}\left(\binom{n}{2}^{-1} d p_2\right) =_d \text{Poisson}\left( d \delta^2 (1 - \delta)^{n - 2} \right)$$
Now since $E(G_2) = \{ \{i, j \} : X_{\{i, j\}} \ge 1 \}$, we have that $G_2 \sim \mG(n, q)$ where $q = 1 - e^{-d \delta^2 (1 - \delta)^{n - 2}}$. Since $M_2, M_3, \dots, M_n$ are independent, it also follows that $G_2, M_3, M_4, \dots, M_n$ are independent.
\end{enumerate}
The second application of Poisson splitting above is especially important to this argument. Note that the number of edges in $G_2$ is distributed as the number of coupons collected among $\binom{n}{2}$ total coupons with $M_2$ trials, a distribution that seems very difficult to work with directly in total variation. The Poisson splitting argument above essentially shows that the coupon collector distribution with a Poisson number of trials is binomially distributed.

Now let $\pr{rig}_P(n, d, p, m_3, m_4, \dots, m_n)$ denote the law of $\pr{rig}_P(n, d, p)$ conditioned on the event that $M_i = m_i$ for each $3 \le i \le n$. For notational convenience, we let $\pr{rig}_P(n, d, p, m_3, m_4, \dots, m_K)$ denote $\pr{rig}_P(n, d, p, m_3, m_4, \dots, m_K, 0, 0, \dots, 0)$ for $K < n$. We now will repeatedly apply Lemma \ref{lem:pcdist} to bound the total variation between $\pr{rig}_P(n, d, p, m_3, \dots, m_n)$ and an Erd\H{o}s-R\'{e}nyi random graph with an appropriately chosen edge density. We remark that we will only need this proposition for $K = 5$, as cliques of size six or larger are sufficiently rare in $\pr{rig}_P$ to have a negligible contribution to the final total variation distance.

The proof of this proposition first carries out a straightforward induction to bound the desired total variation distance as a sum of the upper bounds in Lemma \ref{lem:pcdist}, and then bounds this sum directly. The latter bounding step involves some casework as different ranges of the edge density $p$ need to be handled separately to obtain the desired bound.

\begin{proposition} \label{prop:unionbound}
Let $w = w(n) \to \infty$ as $n \to \infty$ be such that $w \ll n$ and $d \gg w^2 n^3 L^3$ as $n \to \infty$. Let $m_i = m_i(n) \ge 0$ satisfy that $m_i = O_n(wdp_i)$ for each $3 \le i \le K$ for some constant positive integer $K$. Let $q(n, d, p, m_3, m_4, \dots, m_K) \in (0, 1)$ be given by
$$1 - q(n, d, p, m_3, m_4, \dots, m_K) = e^{-d \delta^2(1 - \delta)^{n - 2}} \prod_{i = 3}^K \left( 1 - \binom{i}{2} \binom{n}{2}^{-1} \right)^{m_i}$$
where $p = 1 - (1 - \delta^2)^d$. Then it holds that
$$\TV\left( \pr{rig}_P(n, d, p, m_3, m_4, \dots, m_K), \mG\left(n, q(n, d, p, m_3, m_4, \dots, m_K)\right) \right) = o_n(1)$$
\end{proposition}

\begin{proof}
We begin by handling the case where $p \ge wn^{-2}$. Observe that if $(m_1, m_2, \dots, m_K) \in \mathbb{Z}_{\ge 0}^K$, then $q(n, d, p, m_1, m_2, \dots, m_K) \in [q_{\min}, 1)$ where $q_{\min} = 1 - e^{-d \delta^2(1 - \delta)^{n - 2}}$. By Equation \ref{eqndelta}, we have that $\sqrt{p/d} \le \delta \le 1$ and therefore
$$q_{\min} = 1 - e^{-d \delta^2(1 - \delta)^{n - 2}} \ge \frac{d\delta^2(1 - \delta)^{n - 2}}{1 + d\delta^2(1 - \delta)^{n - 2}} = \Omega_n(p) = \omega_n(n^{-2})$$
Suppose that $n$ is sufficiently large so that $q_{\min} > K^4 n^{-2}$. Now let
$$E_i = \min\left\{1, C\left( q_{\min}^{-1/2} n^{-3/2} + q_{\min}^{-1} n^{-2} + \max_{2 < k \le i} q_{\min}^{-\frac{1}{2}\binom{k}{2}} n^{-k/2} \right) \right\}$$
for a large enough constant $C > 0$ such that $E_i$ is an upper bound in Lemma \ref{lem:pcdist} for all cliques of size $3 \le i \le K$ and $q \in [q_{\min}, 1)$. We will prove by a routine induction on $m_3 + m_4 + \cdots + m_K$ that
\begin{equation} \label{eqn11}
\TV\left( \pr{rig}_P(n, d, p, m_3, m_4, \dots, m_K), \mG\left(n, q(n, d, p, m_3, m_4, \dots, m_K)\right) \right) \le \sum_{i = 3}^K m_i E_i
\end{equation}
for all tuples $(m_3, m_4, \dots, m_K) \in \mathbb{Z}_{\ge 0}^K$. Consider $G \sim \pr{rig}_P(n, d, p)$ generated as described above. Since $G_2, M_3, \dots, M_n$ are independent, $G$ conditioned on the events $M_i = 0$ for $3 \le i \le n$ is distributed as $G_2 \sim \mG(n, q_{\min})$. This completes the base case of the induction. Now suppose that $m_3 + m_4 + \cdots + m_K > 0$ and $3 \le t \le K$ is such that $m_t \ge 1$. Let $m_i' = m_i$ if $i \neq t$ and $m_t' = m_t - 1$. For notational convenience, let $q = q(n, d, p, m_3, m_4, \dots, m_K)$ and $q' = q(n, d, p, m_3', m_4', \dots, m_K')$. Since $q' \ge q_{\min} \ge K^4 n^{-2}$, Lemma \ref{lem:pcdist} implies
$$\TV\left(\mG(n, t, q'), \mG\left(n, q \right) \right) \le E_t$$
Note that $t \le K = O_n(1)$ and it holds by the definition of $q(n, d, p, m_3, m_4, \dots, m_K)$ that
$$q = q' + (1 - q') \binom{t}{2} \binom{n}{2}^{-1}$$
Now observe that the distributions $\pr{rig}_P(n, d, p, m_3, m_4, \dots, m_K)$ and $\mG(n, t, q')$ are both obtained by planting a uniformly at random chosen $t$-clique in samples from $\pr{rig}_P(n, d, p, m_3', m_4', \dots, m_K')$ and $\mG(n, q')$, respectively. It therefore follows by the data-processing inequality that
$$\TV\left(\pr{rig}_P(n, d, p, m_3, m_4, \dots, m_K), \mG(n, t, q') \right) \le \TV\left( \pr{rig}_P(n, d, p, m_3', m_4', \dots, m_K'), \mG(n, q') \right)$$
By the triangle inequality and induction hypothesis, we have that
\begin{align*}
\TV\left( \pr{rig}_P(n, d, p, m_3, m_4, \dots, m_K), \mG\left(n, q\right) \right) &\le \TV\left( \pr{rig}_P(n, d, p, m_3, m_4, \dots, m_K), \mG\left(n, t, q' \right) \right) \\
&\quad \quad + \TV\left( \mG\left(n, t, q' \right), \mG(n, q) \right) \\
&\le \TV\left( \pr{rig}_P(n, d, p, m_3', m_4', \dots, m_K'), \mG(n, q') \right) + E_t \\
&\le E_t + \sum_{i = 3}^K m_i' E_i = \sum_{i = 3}^K m_i E_i
\end{align*}
which completes the induction. Now observe that if $3 \le k \le K$ then Equation \ref{eqndelta} implies that
\begin{equation} \label{eqn12}
dp_k = d\binom{n}{k} \delta^k (1 - \delta)^{n - k} \le dn^k \left(\frac{\log(1 - p)^{-1}}{d} \right)^{k/2} \le dn^k \left( \frac{pL}{d} \right)^{k/2}
\end{equation}
The second inequality above follows from rearranging $\log(1 - p)^{-1} \le p/(1 - p)$ to obtain $\log(1 - p)^{-1} \le pL$. Recall that $L$ denotes $L = 1 + \log(1 - p)^{-1}$. Note that if $k \ge 6$, then the fact that $d \gg w^2n^3 L^3$ implies that $wdp_k \le wn^k L^{k/2} /d^{k/2 - 1} = o_n(1)$ and it also holds that $wdp_k E_k = o_n(1)$ since $E_k \le 1$. We now will bound $wdp_k E_k$ for $3 \le k \le 5$. Since $q_{\min} = \Omega_n(p)$, we have that
\begin{align*}
wdp_3 E_3 &\lesssim wdn^3 \cdot \left(\frac{pL}{d} \right)^{3/2} \cdot \left(p^{-1/2} n^{-3/2} + p^{-1} n^{-2} + p^{-\frac{1}{2}\binom{3}{2}} n^{-3/2} \right) \\
&\lesssim \frac{wn^{3/2}L^{3/2}}{d^{1/2}} + \frac{wp^{1/2} nL^{3/2}}{d^{1/2}} = o_n(1)
\end{align*}
Note that if $p \le n^{-1/2}$, then it follows that
\begin{align*}
wdp_4 E_4 &\le wdp_4 \lesssim \frac{wn^4 p^2L^2}{d} \le \frac{wn^3L^2}{d} = o_n(1) \quad \text{and} \\
wdp_5 E_5 &\le wdp_5 \lesssim \frac{wn^5 p^{5/2}L^{5/2}}{d^{3/2}} \le \frac{wn^{15/4}L^{5/2}}{d^{3/2}} = o_n(1)
\end{align*}
If $p > n^{-1/2}$, it follows that
\begin{align*}
wdp_4 E_4 &\lesssim wdn^4 \cdot \left(\frac{pL}{d} \right)^{2} \cdot \left(p^{-1/2} n^{-3/2} + p^{-1} n^{-2} + p^{-\frac{1}{2}\binom{3}{2}} n^{-3/2} + p^{-\frac{1}{2}\binom{4}{2}} n^{-2} \right) \\
&\lesssim \frac{wp n^2L^2}{d} + \frac{wp^{1/2}n^{5/2}L^2}{d} + \frac{wp^{-1} n^2L^2}{d} = o_n(1)
\end{align*}
since $p = \Omega_n(n^{-1})$. Similarly, it follows that
\begin{align*}
wdp_5 E_5 &\lesssim wdn^5 \cdot \left(\frac{pL}{d} \right)^{5/2} \cdot \left(p^{-1/2} n^{-3/2} + p^{-1} n^{-2} + \sum_{k = 3}^5 p^{-\frac{1}{2}\binom{k}{2}} n^{-k/2} \right) \\
&\lesssim \frac{wp^{3/2} n^3L^{5/2}}{d^{3/2}} + \frac{wp n^{7/2}L^{5/2}}{d^{3/2}} + \frac{wp^{-1/2} n^3L^{5/2}}{d^{3/2}} + \frac{wp^{-5/2} n^{5/2}L^{5/2}}{d^{3/2}} = o_n(1)
\end{align*}
since $p = \Omega_n(n^{-1/2})$. In summary, we have that $m_k E_k = O_n(wdp_k E_k) = o_n(1)$ for each $3 \le k \le K$. Substituting this into Equation \ref{eqn11} proves the proposition if $p \ge n^{-2}$. Now note that if $p < wn^{-2}$, it follows that
$$wdp_k \le wdn^k \cdot \left(\frac{pL}{d} \right)^{k/2} \le w^{k/2 + 1} L^{k/2} d^{-(k/2 - 1)} = o_n(1)$$
for all $3 \le k \le K$ since $d \gg w^5 L^3$. Thus it must follow that $m_k = 0$ for sufficiently large $n$ and all $3 \le k \le K$. This implies that $\pr{rig}_P(n, d, p, m_3, m_4, \dots, m_K)$ is distributed as $G_2 \sim \mG(n, q_{\min})$, in which case the proposition also holds.
\end{proof}

We now are ready to complete the proof of Theorem \ref{thm:denserig}. We will need the following standard upper bound on the total variation between binomial distributions. This is a corollary of Theorem 2.2 in \cite{janson2010asymptotic} stated in \cite{kim2018total}.

\begin{lemma}[Corollary 5.1 in \cite{kim2018total}] \label{lem:binomtv}
For a positive integer $N$ and $0 < p < q < 1$, we have that
$$\TV\left( \textnormal{Bin}(N, p), \textnormal{Bin}(N, q) \right) \le \gamma + 3\gamma^2$$
where $\gamma = (q - p) \sqrt{\frac{N}{p(1 - p)}}$.
\end{lemma}

Suppose that $|p - q| = o(N^{-1})$ and let $f = f(N) \to \infty$ as $N \to \infty$ be such that $|p - q| \le f^{-1} N^{-1}$. If $f^{-1} N^{-1} \le p \le 1 - f^{-1} N^{-1}$ then it follows that $\gamma = O(f^{-1/2}) = o(1)$. If $p < f^{-1} N^{-1}$, then it follows that both $p, q = o(N^{-1})$ and both $\text{Bin}(N, p)$ and $\text{Bin}(N, q)$ are zero with probability $1 - o(1)$. Similarly if $p > 1 - f^{-1} N^{-1}$, both distributions are $N$ with probability $1 - o(1)$. In summary, we have that if $|p - q| = o(N^{-1})$ then
$$\TV\left( \textnormal{Bin}(N, p), \textnormal{Bin}(N, q) \right) = o(1)$$
Combining this with the lemma above and the triangle inequality yields that if $p, q$ and $\gamma$ are as in the lemma and $q' = q + o(N^{-1})$, then
$$\TV\left( \textnormal{Bin}(N, p), \textnormal{Bin}(N, q') \right) \le \gamma + 3\gamma^2 + o(1)$$
This is the form of the lemma we will apply in our proof of Theorem \ref{thm:denserig} below. The remainder of the proof of Theorem \ref{thm:denserig} combines the results in this section, proceeding as follows. We will apply the Poissonization argument above to reduce from considering $\pr{rig}$ to $\pr{rig}_P$. We then further reduce to $\pr{rig}_P$ conditioned on a high probability event $E$, over which each $M_k$ lies in an appropriately chosen confidence interval. We then will apply Proposition \ref{prop:unionbound} to show that it suffices to bound the total variation distance between $\mG(n, p)$ and a mixture of Erd\H{o}s-R\'{e}nyi graphs. This can be upper bounded by a supremum over distances between binomial distributions, and the proof concludes by applying Lemma \ref{lem:binomtv}.

\begin{proof}[Proof of Theorem \ref{thm:denserig}]
We will assume throughout that $n$ is sufficiently large. Fix some $w = w(n) \to \infty$ as $n \to \infty$ such that $d \gg w^2 n^3 L^3$ and $w \ll n$. We will show that 
$$\TV\left( \pr{rig}_P(n, d, p), \mG(n, p) \right) = o_n(1)$$
Combining this with Proposition \ref{prop:poissonization} and the triangle inequality then implies Theorem \ref{thm:denserig}. Recall that $(M_2, M_3, \dots, M_n)$ are a sample from $\mL_P$, are independent and satisfy that $M_i \sim \text{Poisson}(dp_i)$. Now let $E$ be the event that all of the following inequalities hold
\begin{align*}
dp_k - \sqrt{wdp_k} \le M_k \le dp_k + \sqrt{wdp_k} \quad &\text{for } k \ge 3 \text{ with } dp_k > w^{-1/2} \\
M_k = 0 \quad &\text{for } k \ge 3 \text{ with } dp_k \le w^{-1/2}
\end{align*}
As in Equation \ref{eqn12}, if $k \ge 6$ then
\begin{equation} \label{eqn13}
dp_k \lesssim dn^k \cdot \left(\frac{pL}{d} \right)^{k/2} \lesssim n^k L^{k/2} d^{-(k/2 - 1)} = o_n(w^{-1})
\end{equation}
and thus $M_k = 0$ for all $k \ge 6$ on the event $E$, if $n$ is sufficiently large. Since $M_k \sim \text{Poisson}(dp_k)$ under $\mL_P$ and thus $M_k$ mean and variance $dp_k$, Chebyshev's inequality implies that
$$\bP_{\mL_P}\left[|M_k - dp_k| > \sqrt{wdp_k} \right] \le w^{-1}$$
Now let $A \subseteq \{3, 4, 5\}$ be the set of indices $k$ such that $dp_k > w^{-1/2}$. A union bound now implies
\allowdisplaybreaks
\begin{align*}
\bP_{\mL_P} \left[ E^c \right] &\le \sum_{k \in A} \bP_{\mL_P}\left[|M_k - dp_k| > \sqrt{wdp_k} \right] + \sum_{k \not \in A} \bP_{\mL_P}[M_k \neq 0] \\
&\lesssim 3w^{-1} + \sum_{k \in A^c \cap \{3, 4, 5\}} (1 - e^{-dp_k}) + \sum_{k = 6}^n (1 - e^{-dp_k}) \\
&\lesssim 3w^{-1} + 3\left(1 - e^{-w^{-1/2}}\right) + \sum_{k = 6}^n dp_k
\end{align*}
since $1 - e^{-x} \le x$ for all $x \ge 0$. Substituting the bounds from Equation \ref{eqn13} yields
\allowdisplaybreaks
\begin{align*}
\bP_{\mL_P} \left[ E^c \right] &\lesssim 3w^{-1} + 3w^{-1/2} + \sum_{k = 6}^n d \left( \frac{n L^{1/2}}{d^{1/2}} \right)^k \\
&\lesssim 3w^{-1} + 3w^{-1/2} + \frac{n^6 L^3 d^{-2}}{1 - \frac{nL^{1/2}}{d^{1/2}}} = o_n(1) \numberthis \label{eqncondition}
\end{align*}
Now let $\pr{rig}_E(n, d, p)$ and $\mL_E$ denote the distributions of $\pr{rig}_P(n, d, p)$ and $\mL_P$, respectively, conditioned on the event $E$ holding. Since $\mL_P$ is a product distribution and $E$ is the intersection of events over each of the $M_k$, it follows that $\mL_E$ is also a product distribution. Observe that
$$\pr{rig}_E(n, d, p) = \bE_{(m_3, m_4, m_5) \sim \mL_E} \, \pr{rig}_P(n, d, p, m_3, m_4, m_5)$$
Note that $dp_k + \sqrt{wdp_k} = O_n(wdp_k)$ if $k \in A$ since $dp_k > w^{-1/2}$. Proposition \ref{prop:unionbound} applied with $K = 5$ and the conditioning property of total variation yield that
$$\TV\left( \pr{rig}_E(n, d, p), \bE_{(m_3, m_4, m_5) \sim \mL_E} \, \mG(n, q(n, d, p, m_3, m_4, m_5)) \right) = o_n(1)$$
Now observe that
$$\TV\left( \pr{rig}_E(n, d, p), \pr{rig}_P(n, d, p) \right) \le \bP_{\mL_P} \left[ E^c \right] = o_n(1)$$
By the triangle inequality, it now suffices to show that
$$\TV\left( \mG(n, p), \bE_{(m_3, m_4, m_5) \sim \mL_E} \, \mG(n, q(n, d, p, m_3, m_4, m_5)) \right) = o_n(1)$$
Note that both of these distributions are uniformly distributed conditioned on their edge counts. This implies that
\begin{align*}
&\TV\left( \mG(n, p), \bE_{(m_3, m_4, m_5) \sim \mL_E} \, \mG(n, q(n, d, p, m_3, m_4, m_5)) \right) \\
&\quad \quad = \TV\left( \text{Bin}(N, p), \bE_{(m_3, m_4, m_5) \sim \mL_E} \, \text{Bin}(N, q(n, d, p, m_3, m_4, m_5)) \right) \\
&\quad \quad \le \sup_{(m_3, m_4, m_5) \sim \text{supp}(\mL_E)} \TV\left( \text{Bin}(N, p), \text{Bin}(N, q(n, d, p, m_3, m_4, m_5)) \right)
\end{align*}
where $N = \binom{n}{2}$. The remainder of the proof applies the constraints defining $E$ to deduce that the two binomial distributions in the $\TV$ expression above are close. Let $p_1, p_2 \in (0, 1)$ be such that
\begin{align*}
1 - p_1 &= e^{-d \delta^2(1 - \delta)^{n - 2}} \prod_{k \in A} \left( 1 - \binom{k}{2} \binom{n}{2}^{-1} \right)^{dp_k} \\
1 - p_2 &= e^{-d \delta^2(1 - \delta)^{n - 2}} \prod_{k = 3}^5 \left( 1 - \binom{k}{2} \binom{n}{2}^{-1} \right)^{dp_k}
\end{align*}
First note that
$$\left| \log(1 - p_1) - \log(1 - p_2) \right| = - \sum_{k \in A^c \cap \{3, 4, 5\}} dp_k \log \left( 1 - \binom{k}{2} \binom{n}{2}^{-1} \right) \lesssim w^{-1/2} n^{-2}$$
since $dp_k < w^{-1/2}$ if $k \not \in A$. Now note that since $p = 1 - (1 - \delta^2)^d$, we have that
\begin{align*}
\frac{1}{d} \left( \log(1 - p) - \log(1 - p_2) \right) &= \log(1 - \delta^2) + \delta^2(1 - \delta)^{n - 2} - \sum_{k = 3}^5 p_k \log\left( 1 - \binom{k}{2} \binom{n}{2}^{-1} \right) \\
&= - \delta^2 + O_n(\delta^4) + \delta^2(1 - \delta)^{n - 2} + \sum_{k = 3}^5 \left[ p_k \binom{k}{2} \binom{n}{2}^{-1} + O_n\left(p_k n^{-4}\right) \right]
\end{align*}
This quantity can be further simplified to
\begin{align*}
&-\delta^2 + \delta^2 \cdot \left( (1 - \delta)^{n - 2} + \binom{n}{3} \binom{3}{2} \binom{n}{2}^{-1} \delta (1 - \delta)^{n - 3} +  \binom{n}{4} \binom{4}{2} \binom{n}{2}^{-1} \delta^2 (1 - \delta)^{n - 4} \right. \\
&\quad \quad \quad \quad  \left. +  \binom{n}{5} \binom{5}{2} \binom{n}{2}^{-1} \delta^3(1 - \delta)^{n - 5} \right) + O_n(L^2 d^{-2}) + O_n\left( n^{-1} L^{3/2} d^{-3/2} \right) \\
&\quad \quad = -\delta^2 + \delta^2 \cdot \left( 1 - \sum_{t = 4}^n \binom{n - 2}{t} \delta^t (1 - \delta)^{n - t} \right) + O_n\left( n^{-1} L^{3/2} d^{-3/2} \right) \\
&\quad \quad = O_n(n^4 \delta^6) + O_n\left( n^{-1} L^{3/2} d^{-3/2} \right) \\
&\quad \quad = O_n\left(n^4 L^3 p^3 d^{-3} + n^{-1} L^{3/2} d^{-3/2} \right)
\end{align*}
The first equality holds from the binomial theorem and since $\binom{n}{k} \binom{k}{2} \binom{n}{2}^{-1} = \binom{n - 2}{k - 2}$. Here, we also have combined Equations \ref{eqndelta} and \ref{eqn12} to obtain
$$\delta = O_n\left(\sqrt{\frac{pL}{d}}\right) = O_n(L^{1/2}d^{-1/2}) = o_n(n^{-3/2})$$
and $p_k = \binom{n}{k} \delta^k (1 - \delta)^{n - k} = O_n(n^k \delta^k) = O_n(n^3 L^{3/2}d^{-3/2})$ for each $3 \le k \le 5$. The second last equality above follows from
$$0 < \sum_{t = 4}^n \binom{n - 2}{t} \delta^t (1 - \delta)^{n - t} \le \sum_{t = 4}^n n^t \delta^t = O_n(n^4 \delta^4)$$
since $n\delta = o_n(1)$. Finally, also observe that if $(m_3, m_4, m_5) \in \text{supp}(\mL_E)$, then it follows that
\begin{align*}
\left| \log(1 - p_1) - \log(1 - q(n, d, p, m_3, m_4, m_5)) \right| &\le - \sum_{k \in A} \left| m_k - dp_k \right| \cdot \log \left( 1 - \binom{k}{2} \binom{n}{2}^{-1} \right) \\
&\lesssim \sum_{k = 3}^5 n^{-2} \cdot \sqrt{wdp_k} \lesssim \frac{p^{3/4} w^{1/2} L^{3/4}}{n^{1/2} d^{1/4}}
\end{align*}
Combining these three inequalities with the fact that $d \gg n^3 L^3$ yields that
\begin{align*}
\left| \log(1 - p) - \log(1 - q(n, d, p, m_3, m_4, m_5)) \right| &\lesssim w^{-1/2} n^{-2} + n^4 L^3 p^3 d^{-2} + n^{-1} L^{3/2} d^{-1/2} + \frac{p^{3/4} w^{1/2} L^{3/4}}{n^{1/2} d^{1/4}} \\
&\lesssim w^{-1/2} n^{-2} + \frac{p^{3/4} w^{1/2} L^{3/4}}{n^{1/2} d^{1/4}}
\end{align*}
for all $(m_3, m_4, m_5) \in \text{supp}(\mL_E)$. Observe that since the upper bound on the right hand side above tends to zero as $n \to \infty$, it follows that $1 - q(n, d, p, m_3, m_4, m_5) = \Theta_n(1 - p)$. Now note that, since $e^x$ is $1$-Lipschitz for $x \le 0$, the inequality above also implies that
$$\left| p - q(n, d, p, m_3, m_4, m_5)) \right| \lesssim w^{-1/2} n^{-2} + \frac{p^{3/4} w^{1/2} L^{3/4}}{n^{1/2} d^{1/4}}$$
Note that the first term is $o_n(n^{-2}) = o_n(N^{-1})$ and let
$$\gamma = \frac{p^{3/4} w^{1/2} L^{3/4}}{n^{1/2} d^{1/4}} \cdot \sqrt{\frac{N}{p(1-p)}} \lesssim \frac{p^{1/4} n^{1/2} w^{1/2} L^{3/4}}{(1 - p)^{1/2} d^{1/4}} \lesssim \frac{n^{3/4} w^{1/2} L^{3/4}}{d^{1/4}} = o_n(1)$$
since $1 - p = \Omega_n(n^{-1/2})$ and $d \gg w^2 n^3 L^3$. As argued above, we have $1 - q(n, d, p, m_3, m_4, m_5) = \Theta_n(1 - p)$ and that $q(n, d, p, m_3, m_4, m_5) = \Theta_n(p)$ for all $(m_3, m_4, m_5) \in \text{supp}(\mL_E)$. Applying the previous lemma on the total variation between binomial distributions now yields that
$$\sup_{(m_3, m_4, m_5) \in \text{supp}(\mL_E)} \TV\left( \text{Bin}(N, p), \text{Bin}(N, q(n, d, p, m_3, m_4, m_5)) \right) = o_n(1)$$
which completes the proof of the theorem.
\end{proof}

\subsection{Signed Triangle Count in RIG$(n, d, p)$}
\label{subsec:signedtriangles}

The purpose of this section is to prove Theorem \ref{thm:rig-triangles}. We first establish some notation that will be used throughout this section. Given a simple graph $G$ on the vertex set $[n]$, let $e_{ij} = \mathbbm{1}(\{i, j\} \in E(G))$ for each $1 \le i < j \le n$. Let $T_s(G)$ denote the signed triangle count of a graph $G$ given by
$$T_s(G) = \sum_{1 \le i < j < k \le n} (e_{ij} - p)(e_{ik} - p)(e_{jk} - p)$$
This statistic was introduced in \cite{bubeck2016testing} to show an analogue of Theorem \ref{thm:rig-triangles} for random geometric graphs on $\mathbb{S}^{d - 1}$. Let $T(G)$ denote the ordinary triangle count of $G$. Also recall that Equation \ref{eqndelta} states that
$$\frac{p}{d} \le \delta^2 \le \frac{\log(1 - p)^{-1}}{d}$$
Therefore the condition that $1 - p = \Omega_n(1)$ implies that $\delta = O_n(d^{-1/2})$. In order to establish Theorem \ref{thm:rig-triangles}, we will need the following results on the distribution of $T_s(G)$ and $T(G)$ for $G \sim \pr{rig}(n, d, p)$ and $G \sim \mG(n, p)$. Recall that the statement $A = B + O_n(C)$ is a shorthand for the two-sided estimate $|A - B| = O_n(C)$.

\begin{lemma} \label{lem:signedtriangleexp}
If $G \sim \pr{rig}(n, d, p)$ where $1 - p = (1 - \delta^2)^d = \Omega_n(1)$, then
$$\bE\left[ T_s(G) \right] = \binom{n}{3} \cdot (1 - p)^3 \cdot \left[ d \delta^3 + O_n\left(d\delta^4 \right) \right]$$
\end{lemma}

\begin{lemma} \label{lem:rigsignedtrianglesvar}
If $G \sim \pr{rig}(n, d, p)$ where $1 - p = (1 - \delta^2)^d = \Omega_n(1)$, then
$$\textnormal{Var}[T_s(G)] = \binom{n}{3} \cdot p^3(1 - p)^3 + O_n\left( n^4 d\delta^3 + n^5 d\delta^4 \right)$$
\end{lemma}

\begin{lemma} \label{lem:sparsetriangles}
If $n \ll d \ll n^3$ and $G \sim \pr{rig}(n, d, p)$ where $p = \Theta(1/n)$, then it follows that $T(G) \ge \tfrac{n^3}{12} \cdot \sqrt{p^3/d}$ with probability $1 - o_n(1)$.
\end{lemma}

The following lemma summarizes analogous calculations for $\mG(n, p)$. These calculations are elementary and can be found in Section 3 of \cite{bubeck2016testing}.

\begin{lemma} \label{lem:ercounts}
If $G \sim \mG(n, p)$, then it follows that
$$\bE[T_s(G)] = 0 \quad \quad \quad \bE[T(G)] = \binom{n}{3} \cdot p^3 \quad \quad \quad \textnormal{Var}[T_s(G)] = \binom{n}{3} \cdot p^3(1 - p)^3$$
\end{lemma}

Given these lemmas, the proof of Theorem \ref{thm:rig-triangles} is a straightforward consequence of the definition of total variation.

\begin{proof}[Proof of Theorem \ref{thm:rig-triangles}]
First consider the case in which $1 - p = \Omega(1)$, $p = \Theta(1)$ and $n^2 \ll d \ll n^3$. Equation \ref{eqndelta} implies that $\delta = \Theta(d^{-1/2}) = o(1/n)$. Combining this with Lemma \ref{lem:rigsignedtrianglesvar} yields that $\text{Var}[T_s(G)] = \left( 1 + o_n(1) \right) \binom{n}{3} \cdot p^3(1 - p)^3$ if $G \sim \pr{rig}(n, d, p)$. Therefore it follows that
$$\frac{\bE\left[ T_s(G) \right]}{\sqrt{\text{Var}[T_s(G)]}} = \left( 1 + o_n(1) \right) d\delta^3 \cdot \sqrt{p^{-3}(1 - p)^3 \binom{n}{3}} \gtrsim d^{-1/2} n^{3/2} = \omega_n(1)$$
Therefore if $E$ is the event
$$E = \left\{ T_s(G) \ge \frac{1}{2} \binom{n}{3} \cdot (1 - p)^3 \cdot d\delta^3 \right\}$$
then it follows by Chebyshev's inequality that $\bP_{G \sim \pr{rig}(n, d, p)}[E] = 1 - o_n(1)$. Now consider the case where $G \sim \mG(n, p)$. Lemma \ref{lem:ercounts} implies that $\bE[T_s(G)] = 0$ and $\textnormal{Var}[T_s(G)] = \binom{n}{3} \cdot p^3(1 - p)^3$. Chebyshev's inequality now implies that $\bP_{G \sim \mG(n, p)}[E] = o_n(1)$. The definition of total variation implies that
$$\TV\left( \pr{rig}(n, d, p), \mG(n, p) \right) \ge \left| \bP_{G \sim \pr{rig}(n, d, p)}[E] - \bP_{G \sim \mG(n, p)}[E] \right| = 1 - o_n(1)$$
which completes the proof of the theorem if $p = \Theta(1)$. Now consider the case in which $p = \Theta(1/n)$ and assume that $n \ll d \ll n^3$. Lemma \ref{lem:sparsetriangles} yields $\bP_{G \sim \pr{rig}(n, d, p)}[E'] = 1 - o_n(1)$ where
$$E' = \left\{ T(G) \ge \frac{n^3}{12} \cdot \sqrt{\frac{p^3}{d}} \right\}$$
Note that $T(G) \ge 0$ and thus Markov's inequality implies that
$$\bP_{G \sim \mG(n, p)}[E'] \le \frac{\bE_{G \sim \mG(n, p)}[T(G)]}{\frac{n^3}{12} \cdot \sqrt{\frac{p^3}{d}}} \lesssim \sqrt{dp^3} = o_n(1)$$
since $\bE_{G \sim \mG(n, p)}[T(G)] = \binom{n}{3} \cdot p^3$ by Lemma \ref{lem:ercounts}, $d \ll n^3$ and $p = \Theta(1/n)$. Similarly, this implies that $\TV\left( \pr{rig}(n, d, p), \mG(n, p) \right) = 1 - o_n(1)$, proving the theorem.
\end{proof}

We remark that the proof above more generally shows that the two graphs do not converge in total variation if $p^{-3} n^2 \ll d \ll n^3$ or if $p = \Theta(1/n)$ and $n \ll d \ll n^3$. These extended conditions are omitted from Theorem \ref{thm:rig-triangles} for simplicity.

In the rest of this section, we prove Lemmas \ref{lem:signedtriangleexp} and \ref{lem:sparsetriangles}. We now present the proof of Lemma \ref{lem:signedtriangleexp}, which computes the expectation of $T_s(G)$ for $G \sim \pr{rig}(n, d, p)$. This expectation is $\binom{n}{3}$ times the expected value of a single signed triangle, which a priori is a fairly intractable combinatorial sum. Our main trick is to write this expectation as the linear combination of the probabilities of subsets of edges being omitted from $G$. These probabilities are products over the elements of the base set $[d]$, from which we obtain a fairly simple explicit expression for the desired expectation. The proof of Lemma \ref{lem:rigsignedtrianglesvar} uses similar ideas but is considerably more computationally involved. This proof is deferred to Appendix \ref{subsec:rig-triangles}. In Appendix \ref{subsec:rig-triangles}, we also show how to adapt the method below to compute $\bE[T(G)]$ for $G \sim \pr{rig}(n, d, p)$.

\begin{proof}[Proof of Lemma \ref{lem:signedtriangleexp}]
First observe that linearity of expectation and symmetry yields that
\begin{equation} \label{eqn:signedexp}
\bE\left[ T_s(G) \right] = \sum_{1 \le i < j < k \le n} \bE[(e_{ij} - p)(e_{ik} - p)(e_{jk} - p)] = \binom{n}{3} \cdot \bE[(e_{12} - p)(e_{13} - p)(e_{23} - p)]
\end{equation}
For each $x \in \{0, 1\}^3$, let $P(x_1, x_2, x_3)$ denote the probability that $e_{12} = x_1$, $e_{13} = x_2$ and $e_{23} = x_3$. Now define $Q : \{0, 1\}^3 \to [0, 1]$ as
$$Q(x_1, x_2, x_3) = \sum_{y \subseteq x} P(y_1, y_2, y_3)$$
Note that $Q(x_1, x_2, x_3)$ is the probability that edges among $\{1, 2\}, \{1, 3\}$ and $\{2, 3\}$ that are present in $G$ form a subset of the support of $(x_1, x_2, x_3)$. Therefore we have that
\begin{equation} \label{eqn:expidentity}
Q(x_1, x_2, x_3) = \bE\left[ \prod_{\{i, j\} \in C(x)} (1 - e_{ij}) \right]
\end{equation}
where $C(x)$ is the set of edges among $\{1, 2\}, \{1, 3\}, \{2, 3\}$ with corresponding indicators among $x_1, x_2, x_3$ equal to zero. We now compute $Q(x)$ for each $x \in \{0, 1\}^3$. Note that $P(0, 0, 0) = Q(0, 0, 0)$ is the probability that none of these three edges is present. If $S_1, S_2$ and $S_3$ are the latent sets for vertices $1, 2$ and $3$, respectively, then this is the same as the event that each $i \in [d]$ is present in at most one of $S_1, S_2$ and $S_3$ for each $i$. Note that these events are independent for different $i$. The probability that any given $i \in [d]$ is in at most one of $S_1, S_2$ and $S_3$ is
$$\bP\left[i \text{ is in at most one of } S_1, S_2, S_3 \right] = (1 - \delta)^3 + 3\delta(1 - \delta)^2 = (1 + 2\delta) (1 - \delta)^2$$
Independence for different $i$ now implies that
$$P(0, 0, 0) = Q(0, 0, 0) = \prod_{i = 1}^d \bP\left[i \text{ is in at most one of } S_1, S_2, S_3 \right] = (1 + 2\delta)^d (1 - \delta)^{2d}$$
Note that $Q(1, 0, 0)$ is the probability that each $i \in [d]$ is either in at most one of $S_1, S_2, S_3$ or is in both of $S_1$ and $S_2$. Thus
$$Q(1, 0, 0) = \left[ (1 - \delta)^3 + 3\delta(1 - \delta)^2 + \delta^2 (1 - \delta) \right]^d = (1 - \delta)^d(1 + \delta - \delta^2)^d$$
Generalizing this to other $x \in \{0, 1\}^3$ implies that
$$Q(x_1, x_2, x_3) = (1 - \delta)^d(1 + \delta - \delta^2)^d$$
if $|x| = x_1 + x_2 + x_3 = 1$. By a similar argument, if $|x| = x_1 + x_2 + x_3 = 2$, then
$$Q(x_1, x_2, x_3) = \left[ (1 - \delta)^3 + 3\delta(1 - \delta)^2 + 2\delta^2 (1 - \delta) \right]^d = (1 - \delta)^d(1 + \delta)^d$$
Note that $Q(1, 1, 1) = 1$. Now using Equation \ref{eqn:expidentity}, we have that
\allowdisplaybreaks
\begin{align*}
\bE[(e_{12} - p)(e_{13} - p)(e_{23} - p)] &= -\bE\left[\left((1 - e_{12}) - (1 - p) \right) \left((1 - e_{13}) - (1 - p) \right) \left((1 - e_{23}) - (1 - p) \right) \right] \\
&\quad \quad = - \sum_{x \in \{0, 1\}^3} (-1)^{|x|} (1 - p)^{|x|} \cdot Q(x_1, x_2, x_3)
\end{align*}
Directly expanding the $Q(x_1, x_2, x_3)$ and the fact that $1 - p = (1 - \delta^2)^d$ simplifies this quantity to
\allowdisplaybreaks
\begin{align*}
&(1 - p)^3 - 3(1 - p)^2 (1 - \delta)^d(1 + \delta)^d + 3(1 - p) (1 - \delta)^d(1 + \delta - \delta^2)^d - (1 + 2\delta)^d (1 - \delta)^{2d} \\
&\quad \quad = (1 - p)^3 \cdot \left[ -2 + 3(1 - \delta^2)^{-d} \left( 1 - \frac{\delta^2}{1 + \delta} \right)^d - (1 - \delta^2)^{-d} \left( 1 - \frac{\delta^2}{(1 + \delta)^2} \right)^d \right] \\
&\quad \quad = (1 - p)^3 \cdot \left[ -2 + 3 \left( 1 + \frac{\delta^3}{(1 - \delta^2)(1 + \delta)} \right)^d - \left( 1 + \frac{2\delta^3 + \delta^4}{(1 - \delta^2)(1 + \delta)^2} \right)^d \right] \numberthis \label{eqn:finalsigned}
\end{align*}
Now let
$$\Delta_1 = \frac{\delta^3}{(1 - \delta^2)(1 + \delta)} \quad \text{and} \quad \Delta_2 = \frac{2\delta^3 + \delta^4}{(1 - \delta^2)(1 + \delta)^2}$$
Since $\delta = O_n(d^{-1/2})$, it follows that $d\Delta_1, d\Delta_2 = O_n(d\delta^3) = o_n(1)$. Therefore we have that for sufficiently large $d$,
\allowdisplaybreaks
\begin{align*}
\left| 3 \left( 1 + \frac{\delta^3}{(1 - \delta^2)(1 + \delta)} \right)^d - 3 - 3d\delta^3 \right| &\le 3 \left| d\Delta_1 - d\delta^3 \right| + 3\sum_{k = 2}^d \binom{d}{k} \Delta_1^k (1 - \Delta_1)^{d - k} \\
&\le 3 \cdot \left| \frac{d(\delta^4 - \delta^5 - \delta^6)}{(1 - \delta^2)(1 + \delta)} \right| + 3 \sum_{k = 2}^d d^k \Delta_1^k \\
&\le \frac{3d(\delta^4 + \delta^5 + \delta^6)}{(1 - \delta^2)(1 + \delta)} + \frac{3d^2 \Delta_1^2}{1 - d\Delta_1} \\
&\lesssim d\delta^4 + d^2 \Delta_1^2 \\
&\lesssim d\delta^4
\end{align*}
Therefore it follows that
\begin{equation} \label{eqn:firstapprox}
3 \left( 1 + \frac{\delta^3}{(1 - \delta^2)(1 + \delta)} \right)^d - 3 = 3d \delta^3 + O_n\left(d\delta^4 \right)
\end{equation}
By a similar computation, it follows that
\begin{equation} \label{eqn:secondapprox}
\left( 1 + \frac{2\delta^3 + \delta^4}{(1 - \delta^2)(1 + \delta)^2} \right)^d - 1 = 2d \delta^3 + O_n\left( d\delta^4 \right)
\end{equation}
Substituting these bounds into Equation \ref{eqn:finalsigned}, we have that
\begin{align*}
\bE[(e_{12} - p)(e_{13} - p)(e_{23} - p)] &= (1 - p)^3 \cdot \left[ d\delta^3 + O_n(d\delta^4) \right]
\end{align*}
Now combining this with Equation \ref{eqn:signedexp} completes the proof of the lemma.
\end{proof}

We conclude this section by proving Lemma \ref{lem:sparsetriangles}. This is a simple consequence of the planting cliques view of $\pr{rig}(n, d, p)$ in Section \ref{subsec:plantingcliques}.

\begin{proof}[Proof of Lemma \ref{lem:sparsetriangles}]
We use the same notation as in the proof of Theorem \ref{thm:denserig}. Observe that $T(G) \ge M_3$ where $M_3$ is the number of 3-cliques planted in the construction of $G$. Furthermore, $M_3 \sim \text{Bin}(d, p_3)$ where $p_3 = \binom{n}{3} \delta^3(1 - \delta)^{n - 3}$. Now note that since $p = \Theta(1/n)$, it follows that
$$\frac{p}{d} \le \delta^2 \le \frac{\log(1 - p)^{-1}}{d} = \frac{p}{d} + O_n\left( \frac{p^2}{d} \right)$$
Thus $\delta = \Theta_n(1/\sqrt{nd}) = o_n(1/n)$, which implies that $(1 - \delta)^{n - 3} = 1 - o_n(1)$. Therefore
$$dp_3 = d\binom{n}{3} \delta^3(1 - \delta)^{n - 3} = \left(1 + o_n(1)\right) \cdot \frac{n^3}{6} \cdot \sqrt{\frac{p^3}{d}} = \omega_n(1)$$
since $p = \Theta(1/n)$. Since $dp_3 \to \infty$, standard concentration inequalities for the binomial distribution then imply that $M_3 \ge 3dp_3/4$ with probability $1 - o_n(1)$ where $3dp_3/4 \ge \tfrac{n^3}{12} \cdot \sqrt{p^3/d}$ for sufficiently large $n$. This completes the proof of the lemma.
\end{proof}

\section{Random Intersection Matrices and Higher Thresholds $\tau$}
\label{sec:rim}

In this section, we extend the approach used to prove Theorem \ref{thm:denserig} to directly couple the full matrix of intersection sizes between the sets $S_i$ to a matrix with i.i.d. Poisson entries and prove Theorem \ref{thm:rim}. Applying the data-processing inequality to thresholding this matrix at $\tau > 1$ will then yield a natural extension of Theorem \ref{thm:introdenserig} to random intersection graphs defined at higher thresholds than $1$ and prove Corollary \ref{cor:higherthres}.

The main results of this section are Theorem \ref{thm:rim} and Corollary \ref{cor:higherthres} identifying triples of $(n, d, \delta)$ for which $\pr{rim}$ and $\pr{poim}$ converge and quadruples of $(n, d, p, \tau)$ for which $\pr{rig}_\tau$ and $\mG(n, p)$ converge, respectively. These are restated here for convenience.

\begin{reptheorem}{thm:rim}
Suppose that $\delta = \delta(n) \in (0, 1)$ and $d$ satisfies that $d \gg n^3$ and $\delta \ll d^{-1/3} n^{-1/2}$. Then it holds that
$$\TV\left( \pr{rim}(n, d, \delta), \pr{poim}\left(n, d\delta^2\right) \right) \to 0 \quad \text{as } n \to \infty$$
\end{reptheorem}

\begin{repcorollary}{cor:higherthres}
Suppose $p = p(n) \in (0, 1)$, $\delta = \delta(n) \in (0, 1)$, $\tau \in \mathbb{Z}_+$ and $d$ satisfy that
$$1 - p = \sum_{k = 0}^{\tau - 1} \binom{d}{k} \delta^{2k}(1 - \delta^2)^{d - k}$$
Furthermore suppose that
$$d \gg n^3, \quad \delta \ll d^{-1/3} n^{-1/2} \quad \textnormal{and} \quad n^2 \delta^4 \ll p(1 - p)$$
Then it follows that
$$\TV\left( \pr{rig}_\tau(n, d, p), \mG(n, p) \right) \to 0 \quad \text{as } n \to \infty$$
\end{repcorollary}

The proof of Theorem \ref{thm:rim} proceeds in analogous steps to those in the proof of Theorem \ref{thm:denserig}. A key ingredient is a sharp analysis of the total variation distance between planted and non-planted Poisson matrices, an intermediary object defined below that will appear in our argument.

\begin{definition}[Planted Poisson Matrix]
Given $\lambda \in \mathbb{R}_{\ge 0}$ and a positive integer $t \ge 2$, let $\pr{poim}_P(n, t, \lambda)$ denote the distribution of symmetric $n \times n$ matrices $M$ generated in the steps:
\begin{enumerate}
\item select a subset $S \subseteq [n]$ of size $|S| = t$ uniformly at random; and
\item form the symmetric matrix $M$ with $M_{ii} = 0$ for $1 \le i \le n$ and entries $M_{ij}$ with $1 \le i < j \le n$ conditionally independent given $S$ and distributed as
$$M_{ij} \sim \left\{ \begin{array}{cl} 1 + \textnormal{Poisson}(\lambda) &\textnormal{if } i, j \in S \\ \textnormal{Poisson}(\lambda) &\textnormal{otherwise} \end{array} \right.$$
\end{enumerate}
\end{definition}

The next lemma is an analogue of Lemma \ref{lem:pcdist} for random intersection matrices. Its proof is deferred to Appendix \ref{subsec:plantedpois}.

\begin{lemma} \label{lem:plantedpois}
Let $t \ge 3$ be a constant positive integer and $\lambda = \lambda(n) \in \mathbb{R}_{\ge 0}$ be such that $\lambda = \omega_n(n^{-2})$. Then it follows that
\begin{align*}
&\TV\left( \pr{poim}_P\left( n, t, \lambda \right), \pr{poim}\left( n, \lambda + \binom{t}{2} \binom{n}{2}^{-1} \right) \right) \\
&\quad \quad = O_n \left( (1 + \lambda^{-1}) n^{-2} + \max_{2 < k \le t} n^{-k/2} \left( 1 + \lambda^{-1} \right)^{\frac{1}{2} \binom{k}{2}} \right)
\end{align*}
\end{lemma}

Using this lemma, the proof of Theorem \ref{thm:rim} follows the same steps as the proof of Theorem \ref{thm:denserig} -- it applies the above lemma inductively for elements of $[d]$ in at least three sets after several Poissonization steps. The full details of the remainder of the proof of Theorem \ref{thm:rim} can be found in Appendix \ref{subsec:rim-poim}. Now, by thresholding instances of $\pr{rim}$ and applying the data-processing inequality, we can use Theorem \ref{thm:rim} to prove Corollary \ref{cor:higherthres}.

\begin{proof}[Proof of Corollary \ref{cor:higherthres}]
First note that if $p = o_n(n^{-2})$, then a union bound yields that both $\pr{rig}_\tau(n, d, p)$ and $\mG(n, p)$ are the empty graph with probability $1 - o_n(1)$. In this case, the corollary follows. Similarly, if $1 - p = o_n(n^{-2})$ then both graphs are complete with probability $1 - o_n(1)$ and the corollary also follows. In particular, we may assume that $\min(p, 1 - p) \gg n^{-3}$.

Consider the graph $G$ with an adjacency matrix formed by thresholding the entries of a matrix $X \in \mathcal{M}_n$ each at $\tau$, or in other words with $\{i, j \} \in E(G)$ if and only if $X_{ij} \ge \tau$. If $X \sim \pr{rim}(n, d, \delta)$, then $G \sim \pr{rig}_\tau(n, d, p)$ by definition. Furthermore, if $X \sim \pr{poim}(n, d\delta^2)$, then it follows that $G \sim \mG(n, p')$ where $p' \in (0, 1)$ is given by
$$p' = \bP\left[ \textnormal{Poisson}(d\delta^2) \ge \tau \right]$$
The data processing inequality together with Theorem \ref{thm:rim} yield that
\begin{equation} \label{eqn28}
\TV\left( \pr{rig}_\tau(n, d, p), \mG(n, p') \right) \le \TV\left( \pr{rim}(n, d, \delta), \pr{poim}\left(n, d\delta^2\right) \right) = o_n(1)
\end{equation}
Now observe that $p = \bP\left[ \text{Binom}(d, \delta^2) \ge \tau \right]$. By Theorem 2.1 in \cite{barbour1989some}, it follows that
\begin{align*}
|p - p'| &= \left| \bP\left[ \text{Binom}(d, \delta^2) \ge \tau \right] - \bP\left[ \textnormal{Poisson}(d\delta^2) \ge \tau \right] \right| \\
&\le \TV\left( \text{Binom}(d, \delta^2), \textnormal{Poisson}(d\delta^2) \right) \le \delta^2 \numberthis \label{eqn29}
\end{align*}
Let $N = \binom{n}{2}$. Since the distribution of any $\mG(n, q)$ is the same conditioned on its total edge count, it follows that
\begin{equation} \label{eqn30}
\TV\left( \mG(n, p), \mG(n, p') \right) = \TV\left( \text{Binom}(N, p), \text{Binom}(N, p') \right)
\end{equation}
Lemma \ref{lem:binomtv} now applies with $\gamma$ upper bounded by
$$\gamma \le |p - p'| \cdot \sqrt{\frac{N}{\min(p, p') (1 - \max(p, p'))}} \lesssim \frac{n \delta^2}{\sqrt{(p - \delta^2)(1 - p - \delta^2)}} \lesssim \frac{n \delta^2}{\sqrt{p(1 - p)}} = o_n(1)$$
The third inequality follows since $\delta^2 \ll d^{-2/3} n^{-1} \ll n^{-3}$ which is both $o_n(p)$ and $o_n(1 - p)$. This implies that $(p - \delta^2)(1 - p - \delta^2) = \Theta_n(p(1 - p))$. The last inequality follows since $n^2 \delta^4 \ll p(1 - p)$ by assumption. Lemma \ref{lem:binomtv} therefore implies that the total variation in Equation \ref{eqn30} is $o_n(1)$. Combining this with the triangle inequality and Equation \ref{eqn28} proves the corollary.
\end{proof}

We now apply Corollary \ref{cor:higherthres} to different parameter regimes of $p$ and $\tau$. If $\tau = 1$, then Corollary \ref{cor:higherthres} recovers and slightly extends the result in Theorem \ref{thm:denserig}. Observe that if $\tau = 1$, then $1 - p = (1 - \delta^2)^d$ and
$$\frac{p}{d} \le \delta^2 \le \frac{\log(1 - p)^{-1}}{d}$$
as in Equation \ref{eqndelta}. Given these bounds, the conditions in Corollary \ref{cor:higherthres} are satisfied when $d \gg n^3$
$$\frac{\log(1 - p)^{-1}}{d} \ll d^{-2/3} n^{-1} \quad \text{and} \quad n \cdot \frac{\log(1 - p)^{-1}}{d} \ll \sqrt{p(1 - p)}$$
The first condition is the threshold in Theorem \ref{thm:denserig}. The left-hand side in the second condition is $o_n(n^{-2})$, which is always $o_n(\sqrt{p(1 - p)})$ unless one of $p$ or $1 - p$ is $o_n(n^{-4})$. However, in this case, $\pr{rig}(n, d, p)$ and $\mG(n, p)$ are either both empty or complete with probability $1 - o_n(1)$ and still converge in total variation. Thus we have the following corollary extending Theorem \ref{thm:denserig}.

\begin{corollary}
Suppose $p = p(n) \in (0, 1)$ satisfies $1 - p = O_n(n^{-1/2})$ and $d$ satisfies that $d \gg n^3 \log^3 n$. Then it follows that
$$\TV\left( \pr{rig}(n, d, p), \mG(n, p) \right) \to 0 \quad \text{as } n \to \infty$$
\end{corollary}

Corollary \ref{cor:higherthres} also applies to other $p$ and $\tau$. If $\tau$ is constant and $1 - p = \Omega_n(1)$, then it follows that $d\delta^2 = O_n(1)$ and the conditions in Corollary \ref{cor:higherthres} reduce to $d \gg n^3$. If $\tau = \tau(n)$ is growing and $1 - p = \Omega_n(1)$, then the central limit theorem applied to $\text{Binom}(d, \delta^2)$ implies that $d\delta^2 = O_n(\tau)$. In this case, the conditions in Corollary \ref{cor:higherthres} are satisfied when $d \gg n^3$
$$\frac{\tau}{d} \ll d^{-2/3} n^{-1} \quad \text{and} \quad \frac{n\tau}{d} \ll \sqrt{p(1 - p)}$$
By the same argument as in the case when $\tau = 1$, the first condition subsumes the second. Thus Corollary \ref{cor:higherthres} holds when $d \gg \tau^3 n^3$ if $1 - p = \Omega_n(1)$. This is stated formally in the following corollary, which is Corollary \ref{cor:morerig} reproduced from Section \ref{sec:results}.

\begin{repcorollary}{cor:morerig}
Suppose $p = p(n) \in (0, 1)$ satisfies that $1 - p = \Omega_n(1)$ and $d$ and $\tau = \tau(n) \in \mathbb{Z}_+$ satisfy $d \gg \tau^3 n^3$. Then it follows that
$$\TV\left( \pr{rig}_\tau(n, d, p), \mG(n, p) \right) \to 0 \quad \text{as } n \to \infty$$
\end{repcorollary}

\section{Random Geometric Graphs on $\mathbb{S}^{d - 1}$}
\label{sec:rgg}

The main purpose of this section is to prove Theorem \ref{thm:rgg}, yielding the first progress towards a conjecture of \cite{bubeck2016testing} that the regime of parameters $(n, d, p)$ in which $\pr{rgg}(n, d, p)$ to $\mG(n, p)$ converge in total variation increases quickly as $p$ decays with $n$. This theorem is restated below for convenience.

\begin{reptheorem}{thm:rgg}
Suppose $p = p(n) \in (0, 1/2]$ satisfies that $p = \Omega_n(n^{-2} \log n)$ and
$$d \gg \min\left\{ pn^3 \log p^{-1}, p^2 n^{7/2} (\log n)^3 \sqrt{\log p^{-1}} \right\}$$
where $d$ also satisfies that $d \gg n \log^4 n$. Then it follows that
$$\TV\left( \pr{rgg}(n, d, p), \mG(n, p) \right) \to 0 \quad \text{as } n \to \infty$$
\end{reptheorem}

We remark that our argument still yields convergence results if $p = o_n(n^{-2} \log n)$. However, for the sake of maintaining a simple main theorem statement, we relegate these results to the propositions in the next subsections. We begin this section with some preliminary observations and then proceed to the main arguments to establish this theorem in the two subsequent subsections. More precisely, the proof of Theorem \ref{thm:rgg} will roughly proceed as follows:
\begin{enumerate}
\item We reduce bounding the total variation between $\pr{rgg}(n, d, p)$ and $\mG(n, p)$ to bounding the expected value of the $\chi^2$ divergence between the conditional distribution $Q$ of an edge of $\pr{rgg}$ given the rest of the graph and $\text{Bern}(p)$.
\item We introduce a coupling of the variables $X_1, X_2, \dots, X_n$ with a collection of random vectors and variables $(Y_1, Y_2, \dots, Y_n, \Gamma_2, \dots, \Gamma_n)$ with the following properties. The vectors $Y_1, Y_2, \dots, Y_n$ are an orthonormal basis of the span $\mathsf{span}(X_1, X_2, \dots, X_n)$ and $\Gamma_2, \Gamma_3, \dots, \Gamma_n$ are i.i.d. real-valued coefficients, derived from expressing the $X_i$ over this basis, such that the conditional distribution of the edge $\{1, 2\}$ in $\pr{rgg}$ given the rest of the graph can approximately be captured by $\Gamma_2$. Bounding the $\chi^2$ of the conditional distribution $Q$ then reduces to large deviation principles for $\Gamma_2$. This leads to a proof that the theorem holds if $d \gg pn^3 \log p^{-1}$.
\item We refine the bounds obtained in the preceding argument by introducing an alternate coupling between the distribution of $\pr{rgg}$ given the presence of edge $\{1, 2\}$ and the distribution of $\pr{rgg}$ marginalizing out the presence of $\{1, 2\}$. This refines our total variation bound in the sparse case, proving the theorem holds if $d \gg p^2 n^{7/2} (\log n)^3 \sqrt{\log p^{-1}}$.
\end{enumerate}

Before proceeding to the proof of Theorem \ref{thm:rgg}, we make several remarks on the tightness of our argument. As shown by the results in \cite{bubeck2016testing}, Theorem~\ref{thm:rgg} is sharp when $p \in (0,1)$ is a constant. However, the resulting bound in the case when $p = c/n$ is a factor of $p^{3/2}$ off from Conjecture~\ref{conj:BDER}. We believe that this difference may arise at any one of several parts of our argument: the use of of Pinsker's inequality to upper bound TV with KL divergence, the application of tensorization of KL divergence in Equation \ref{eq:log_sobolev_rgg} or when Jensen's inequality is used to replace $Q$ with $Q_0$ in Equations \ref{eqn:Jensens1} and \ref{eqn:Jensens2}. We also believe that the key technical Lemmas~\ref{lem:main_reduction_rgg} and~\ref{lem:expected_deviation_bound} in the proof of Theorem~\ref{thm:rgg}, which bound the deviation of $Q_0$ from its mean, are tight up to logarithmic factors. 

We now carry out Step 1 outlined above. We first establish some notation that will be carried forward throughout this section:
\begin{itemize}
\item Let $N = \binom{n}{2}$ and $X_1, X_2, \dots, X_n$ be sampled uniformly at random from the Haar measure on $\mathbb{S}^{d - 1}$ and let $X_{ij}$ denote the $j$th coordinate of $X_i$ for each $1 \le j \le d$.
\item Let $\grgg = \pr{gg}_{t_{p, d}}(X_1, X_2, \dots, X_n)$ and let $\nurgg$ denote the probability mass function of the graph $\grgg \sim \pr{rgg}(n, d, p)$. Let the probability mass function of $\mG(n, p)$ be $\mu$. Let $e_0$ denote the edge $\{1, 2\}$ and, given an edge $e$, let $\grgg_{\sim e}$ denote the set of edges in $\grgg$ other than $e$.
\item Let $\psi_d$ denote the marginal density of a coordinate $X_{11}$ of the Haar measure on $\mathbb{S}^{d - 1}$. Let $\Psi_d(x) = \int_{x}^{1}\psi_d(t)dt$ denote the tail function of $\psi_d$. Furthermore, let the standard normal tail function be given by $\bar{\Phi}(x) = \mathbb{P}[\mathcal{N}(0,1)\geq x]$.
\end{itemize}

We now define a key random variable in our proof -- the probability $Q$ that a specific edge is included in the graph given the rest of the graph. Define the $\sigma(\grgg_{\sim e_0})$-measurable random variable $Q$ taking values in $[0,1]$ as
 $$Q = \mathbb{P}\left[e_0 \in E(\grgg) \big| \grgg_{\sim e_0}\right] = \mathbb{E}\left[\mathbbm{1}(e_0 \in E(\grgg))\big| \sigma(\grgg_{\sim e_0})\right]$$
We will show that this value is approximately $p$ with high probability when $d$ grows fast enough as a function of $n$. We first reduce the total variation convergence of $\pr{rgg}(n, d, p)$ and $\mG(n, p)$ to showing this. Applying Lemma~\ref{lem:kl_tensorization}, we can upper bound $\kl$ by an expected $\kl$ of marginal distributions and then by $\chi^2$ as follows:
\begin{align*} 
\kl\left(\nurgg\bigr|\bigr|\mu\right) &\le \sum_{1 \le i < j \le n} \bE \left[ \kl\left( \mL\left( \mathbbm{1}(\{i, j\} \in E(\grgg)) \big| \sigma\left(\grgg_{\sim \{i, j\}}\right) \right) \, \bigr|\bigr| \, \text{Bern}(p) \right) \right] \\
&= N \cdot \bE \left[ \kl\left( \mL\left( \mathbbm{1}(e_0 \in E(\grgg))\big| \sigma(\grgg_{\sim e_0}) \right) \, \bigr|\bigr| \, \text{Bern}(p) \right) \right] \\
&\le N \cdot \bE \left[ \chi^2\left( \mL\left( \mathbbm{1}(e_0 \in E(\grgg))\big| \sigma(\grgg_{\sim e_0}) \right), \, \text{Bern}(p) \right) \right] \\
&= N \cdot \mathbb{E}\left[\frac{(Q-p)^2}{p(1-p)}\right] \numberthis \label{eq:log_sobolev_rgg}
\end{align*}
By Pinsker's inequality, it suffices to show the right hand side in Equation \ref{eq:log_sobolev_rgg} is $o_n(1)$. The two subsequent subsections give arguments to establish this. Before proceeding, we note some useful estimates for $\psi_d$ and $\Psi_d$ in the following two lemmas. The first item in the following lemma is discussed in Section 2 of \cite{bubeck2016testing} and shown in Section 2 of \cite{sodin2007tail}. The second item is Lemma 2 in Section 2 of \cite{bubeck2016testing}. The proofs of the other three items in the lemma are provided in Appendix \ref{subsec:psi-prop}. 

\begin{lemma}[Estimates for $\psi_d$ and $t_{p,d}$] \label{lem:psi_properties}
The marginal $\psi_d$ and $t_{p,d}$ satisfy the properties:
\begin{enumerate}
    \item For all $x\in [-1,1]$, $$\psi_d(x) = \frac{\Gamma\left(\frac{d}{2}\right)}{\Gamma\left(\frac{d-1}{2}\right)\sqrt{\pi}}(1-x^2)^{\frac{d-3}{2}}$$
    $\psi_d(x)$ is symmetric about $x= 0$ and strictly decreasing for $x \in [0,1]$.
    \item For every $0<p \leq \frac{1}{2}$ and an absolute constant $C$ we have
    $$\min\left(\frac{1}{2},C^{-1}\left(\frac{1}{2}-p \right)\sqrt{\frac{\log p^{-1}}{d}}\right)\leq t_{p,d}\leq C\sqrt{\frac{\log p^{-1}}{d}}$$
    \item Let $0 \le t \le \tfrac{1}{2}$ and $0\leq \delta \leq t$. Then, $$\frac{\psi_d(t-\delta)}{\psi_d(t)} \leq e^{2td\delta} $$
    \item For every $0 < p \leq \frac{1}{2}$, there is an absolute constant $C_1 > 0$ such that
    $$\psi_d(t_{p,d}) \leq C_1 p \cdot \max\left\{\sqrt{d},dt_{p,d} \right\}$$
    \item Let $T \sim \psi_d$. Then, for any $ 0< p \leq \frac{1}{2}$ and some constant $C > 0$
    $$\mathbb{P}\left(|T| > C\sqrt{\frac{\log p^{-1}}{d}}\right) \leq 2p$$
    \end{enumerate}
\end{lemma}
%

The following distributional approximation result is proven by Sodin \cite{sodin2007tail} and stated in \cite{bubeck2016testing}. We remark that our definition of $\Psi_d$ is scaled compared to the definition in \cite{bubeck2016testing}.

\begin{lemma}\label{lem:distributional_approx_on_the_sphere}
There exist strictly positive universal constants $C_{\mathsf{est}},C_1,C_2$ and a sequence $\epsilon_d = O\left(d^{-1} \right)$ such that the following inequalities hold for every $0 \leq t < C_{\mathsf{est}}$: 
$$ (1-\epsilon_d)\cdot \bar{\Phi}\left(t\sqrt{d}\right) \cdot e^{-C_1t^4d} \leq \Psi_d(t) \leq (1+\epsilon_d) \cdot \bar{\Phi}\left(t\sqrt{d}\right) \cdot e^{-C_2t^4d}$$
\end{lemma}

\subsection{Coupling $X_1, X_2, \dots, X_n$ to Isolate the Edge $\{1, 2\}$}

In this section, we give a coupling argument to upper bound the $\chi^2$ divergence on the right-hand side of Equation~\ref{eq:log_sobolev_rgg}. Let $X_2, X_3, \dots, X_n$ be independently and randomly chosen from the Haar measure on $\mathbb{S}^{d - 1}$. We now will describe a coupling giving an alternative way of generating $X_1$ that will give a direct description of $\mathbbm{1}(e_0 \in E(G))$ in terms of random variables introduced in the coupling. As in the statement of Theorem \ref{thm:rgg}, we assume that $d \geq n$. Note that this implies $\mathsf{span}(X_2, X_3, \dots, X_i)$ is a measure-zero subset of $\mathbb{S}^{d - 1}$ for each $1 \le i \le n - 1$. Thus the vectors $X_2, X_3, \dots, X_n$ are linearly independent almost surely. We now define the key random variables underlying our coupling.
\begin{itemize}
\item Let $Y_2, Y_3, \dots, Y_n$ be orthonormal vectors obtained by applying Gram-Schmidt to the vectors $X_2, X_3, \dots, X_n$ such that
\begin{align*}
&Y_n = X_n \quad \text{and} \\
&Y_k = \frac{X_k - \sum_{m = k+1}^n \text{Proj}_{Y_m} (X_k)}{\left\| X_k - \sum_{m = k+1}^n \text{Proj}_{Y_m} (X_k) \right\|_2}\quad \text{for all } 2 \le k \le n - 1 \numberthis \label{eqn:gramschmidt}
\end{align*}
Note that this implies
$$Y_n = X_n, \quad Y_{n-1} \in \mathsf{span}\{X_{n-1},X_{n}\}, \quad \dots, \quad Y_{2} \in \mathsf{span}\{X_2,\dots,X_n\}$$
\item Let $\Gamma_2, \Gamma_3, \dots, \Gamma_n$ be independent random variables and independent of $\sigma(X_2, X_3, \dots, X_n)$ such that $\Gamma_i \sim \psi_{d - n + i}$ for each $2 \le i \le n$.
\item Let $T_1, T_2, \dots, T_n$ be functions of $\Gamma_2, \Gamma_3, \dots, \Gamma_n$ given by
$$T_i = \Gamma_i \cdot \prod_{j = i + 1}^{n} \sqrt{1 - \Gamma_{j}^2}$$
for each $2 \le i \le n$ and
$$T_1 = \prod_{j = 2}^n \sqrt{1 - \Gamma_{j}^2}$$
\item Let $S_{d - n}$ denote the unit sphere in the $(d - n + 1)$-dimensional subspace orthogonal to $\mathsf{span}(Y_2, Y_3, \dots, Y_n)$. Let $Y_1$ be sampled from the Haar measure on $S_{d - n}$, independently of $\sigma(\Gamma_2, \dots, \Gamma_n, X_2, \dots, X_n)$, and set
$$X_1 = \sum_{i=1}^{n}T_iY_i$$
\end{itemize}
A straightforward induction shows that
$$\sum_{i=1}^{j}T_i^2 = \prod_{i = j+1}^n \left(1 - \Gamma_i^2\right)$$
for each $1 \le j \le n$. In particular, it holds that $\sum_{i=1}^{n}T_i^2 = 1$. We now will establish several key distributional properties of this coupling in the following two propositions.


\begin{proposition}\label{thm:rgg_coupling_results}
The random variables in the coupling satisfy that
\begin{enumerate}
    \item $X_1$ is independent of $\sigma(X_2,\dots,X_n)$ and is uniformly distributed on $\sphere^{d-1}$.
    \item $T_i \sim \psi_d$ for each $2 \le i \le n$.
    \item $\langle X_2,Y_j\rangle \sim \psi_d$ for $3 \le j \le n$.
\end{enumerate}
\end{proposition}

In order to prove this proposition, we will make use of the following lemma. The proof of this lemma is in Appendix \ref{subsec:coupling}.

\begin{lemma} \label{lem:sphere_uniform_induction}
The following two statements hold for the uniform distribution over unit spheres.
\begin{enumerate}
\item Let $a \in \mathbb{S}^{d - 1}$, let $a^{\perp}$ be the $(d - 1)$-dimensional space orthogonal to $a$ and let $\sphere^{a^{\perp}}$ be the unit sphere embedded in $a^{\perp}$. Let $T$ be a random variable taking values almost surely in $[-1,1]$ and let $Y$ be a random vector in $a^{\perp}$. Then, the random vector $X = Ta + \sqrt{1-T^2} \cdot Y$ is uniformly distributed over $\sphere^{d-1}$ if and only if $T \sim \psi_d$, $Y$ is uniformly distributed over $\sphere^{a^{\perp}}$ and $T$ is independent of $Y$. 
\item Let $m$ be a positive integer satisfying $m \leq d$ and $Z_1, Z_2, \dots, Z_m$ be a random set of orthonormal vectors sampled according to the Haar measure on the orthogonal group. Let $X \sim \unif(\sphere^{d-1})$ be independent of $Z_1, Z_2, \dots,Z_m$ and let $\xi \in \mathbb{R}^m$ be such that $\xi_i = \langle X, Z_i\rangle$ for each $1 \le i \le m$. Then it holds that $\xi/\|\xi\|_2 \sim \unif(\sphere^{m-1})$.
\end{enumerate}
\end{lemma}

\begin{proof}[Proof of Proposition \ref{thm:rgg_coupling_results}]
We prove the three items in the proposition separately as follows.
\begin{enumerate}
\item For each $1 \le m \le n$, define the intermediate variables
$$T_i^m = \Gamma_i \cdot \prod_{j = i + 1}^{m} \sqrt{1 - \Gamma_{j}^2} \quad \text{for } 2 \le i \le m \quad \text{and} \quad T_1^m = \prod_{j = 2}^{m} \sqrt{1 - \Gamma_{j}^2}$$
and let $X_1^m$ be
$$X_1^m = \sum_{i=1}^{m}T_i^m Y_i$$
Let $S_{d - n + m - 1}$ denote the unit sphere in the $(d - n + m)$-dimensional subspace orthogonal to $\mathsf{span}(Y_{m + 1}, Y_{m+2}, \dots, Y_n)$. We will show by induction on $m$ that $X_1^m \sim \unif(S_{d - n + m - 1})$ conditioned on any event in $\sigma(X_2,\dots,X_n)$. By definition, this holds if $m = 1$ since $T_1^1 = 1$. Now observe that since $T_{m+1}^{m+1} = \Gamma_{m+1}$ and $T^{m+1}_i = T^{m}_i \cdot \sqrt{1 - \Gamma_{m+1}^2}$ for all $m \ge 1$ and $i \le m$, we have that
\begin{align*}
X_1^{m+1} &= \sum_{i=1}^{m+1}T_i^{m+1} Y_i = \Gamma_{m+1} Y_{m+1} + \sqrt{1 - \Gamma_{m+1}^2} \cdot \sum_{i=1}^{m}T_i^{m} Y_i \\
&= \Gamma_{m+1} Y_{m+1} + \sqrt{1 - \Gamma_{m+1}^2} \cdot X_1^m
\end{align*}
for each $1 \le m \le n - 1$. The induction hypothesis implies that $X_1^m \sim \unif(S_{d - n + m - 1})$. Since $\Gamma_{m+1} \sim \psi_{d - n + m + 1}$ and $\Gamma_{m+1}$ is independent of $X_1^m \in \sigma(X_1, \dots, X_m, \Gamma_1, \dots, \Gamma_m)$, item 1 in Lemma~\ref{lem:sphere_uniform_induction} implies that $X_1^{m+1} \sim \unif(S_{d - n + m})$, completing the induction. Now setting $m = n$ yields the result since $X_1^n = X_1$.
\item Observe that $T_i = \langle X_1, Y_i \rangle$ for $2 \le i \le n$. Note that $Y_i \in \sigma(X_2,\dots, X_n)$ for each $2 \le i \le n$ and hence independent of $X_1$ by the previous item in the proposition. Since $\|Y_i\|_2 =1 $ almost surely, it follows by the definition of $\psi_d$ and the rotational invariance of $\mathsf{unif}(\sphere^{d - 1})$ that $T_i \sim \psi_d$ for each $2 \le i \le n$.
\item This follows from the rotational invariance of $\mathsf{unif}(\sphere^{d - 1})$ and the fact that $Y_j \in \sigma(X_3, \dots, X_n)$ for each $3 \le j \le n$ and thus independent of $X_2$.
\end{enumerate}
This completes the proof of the proposition.
\end{proof}

Let $\mathcal{F}$ denote the $\sigma$-algebra $\mathcal{F} = \sigma(\Gamma_3,\dots,\Gamma_n,X_2,\dots,X_n)$. The second distributional property of our coupling that we establish is that the graph other than the edge $\{1, 2\}$ is determined by $\mathcal{F}$.

\begin{proposition} \label{lem:coarse_sigma_algebra}
It holds that $\sigma(\grgg_{\sim e_0}) \subseteq \mathcal{F}$.
\end{proposition}

\begin{proof}
It suffices to show that $\grgg_{\sim e_0}$ is a deterministic function of $\Gamma_3,\dots,\Gamma_n,X_2,\dots,X_n$. Note that the events $\{i,j\} \in E(\grgg_{\sim e_0})$ for $2 \le i < j \le n$ and the random variables $Y_2, \dots, Y_n$ are determined by thresholding $\langle X_i, X_j\rangle$ and Gram-Schmidt orthogonalization, respectively, both of which are deterministic functions of $X_2, \dots, X_n$. By definition $T_3,\dots,T_n$ are deterministic functions of $\Gamma_3,\dots, \Gamma_n$. Furthermore, the $X_i$ can be expressed as $X_i = \sum_{j=i}^{n}a_{ij}Y_j$ for coefficients $a_{ij}$, which are determined by $X_2,\dots,X_n$ in Gram-Schmidt orthogonalization. Therefore, it holds that $\langle X_1,X_i\rangle = \sum_{j=i}^{n}a_{ij}T_j$ and hence the events $\{1,i \} \in E(\grgg_{\sim e_0})$ are in $\mathcal{F}$ for all $3 \le i \le n$. This completes the proof of the proposition.
\end{proof}
 
We now define the random variable 
$$Q_0 =  \mathbb{E}\left[\mathbbm{1}(e_0 \in E(\grgg))\bigr|\mathcal{F}\right]$$
Note that Proposition~\ref{lem:coarse_sigma_algebra} implies that $Q = \mathbb{E}\left[Q_0\bigr|\sigma\left(\grgg_{\sim e_0}\right)\right]$. The remainder of this section is devoted to showing that $Q_0$ concentrates near $p$. By definition, we have that $\mathbbm{1}(e_0 \in E(\grgg)) = \mathbbm{1}\left(\langle X_1,X_2\rangle \geq t_{p,d}\right)$. Furthermore, there are coefficients $a_{2j} \in \sigma(X_2, \dots, X_n)$ for $2 \le j \le n$ such that $X_2 = \sum_{j=2}^{n}a_{2j}Y_j$. It follows that $\langle X_1,X_2\rangle = \sum_{j=2}^{n}a_{2j}T_j$ and that we can rewrite $Q_0$ as
$$Q_0 = \mathbb{P}\left[\sum_{j=2}^{n}a_{2j}T_j \geq t_{p,d} \biggr| \mathcal{F}\right]$$
Rearranging Equation \ref{eqn:gramschmidt} yields that
\begin{align*}
a_{2j} &= \langle X_2, Y_j \rangle \quad \text{for all } 3 \le j \le n, \quad \text{and} \\
a_{22} &= \left\| X_2 - \sum_{j = 3}^n \text{Proj}_{Y_j}(X_2) \right\|_2 = \sqrt{1 - \sum_{j = 3}^n \langle X_2, Y_j \rangle^2}
\end{align*}
In particular, this implies that $a_{22}$ is positive almost surely. As will be shown in the lemmas later in this section, it holds that $a_{22} \approx 1$ and $a_{2j} \approx \frac{1}{\sqrt{d}}$ for $3 \le j \le n$ with high probability. Rearranging now yields that
\begin{equation}
Q_0 = \mathbb{P}\left[\Gamma_2 \ge t^{\prime}_{p, d} \biggr|\mathcal{F}\right] \quad \text{where} \quad t_{p,d}^{\prime} = \frac{t_{p,d} - \sum_{j=3}^n a_{2j} T_j}{a_{22} \cdot \prod_{j=3}^{n} \sqrt{1-\Gamma_j^2}}
\label{eq:identity_for_conditional_probab}
\end{equation}
Observe that $t_{p,d}^{\prime}$ is a $\mathcal{F}$-measurable random variable since $a_{2j} \in \sigma(X_2, \dots, X_n)$ for $2 \le j \le n$. We now will analyze a typical instance of our $\mathcal{F}$-measurable random variables. In particular, we will show that the random threshold $t_{p,d}^{\prime}$ is close to the true threshold $t_{p,d}$ with high probability. The next three lemmas primarily consist of concentration results and bounding. Their proofs can be found in Appendix~\ref{subsec:coupling}.

\begin{lemma}
\label{lem:rgg_remainder_concentration}                                                                                                                                                                                                                                                                                                                                                                                                                                                                                                                                                                                                                                                                                                                                                                                                                                                                                                                                                
Suppose that $d\gg n\log{n}$ and let $s \in (0,\infty)$ be fixed. There exists a fixed constant $C_s$ depending only on $s$ such that the following events all hold with probability at least $1- \frac{1}{n^s}$ for sufficiently large $n$:
\begin{enumerate}
    \item $\biggr|\sum_{j=3}^n a_{2j}T_j\biggr| \leq \frac{C_s \sqrt{n} \log^{3/2} n}{d}$;
    \item $a_{22} \geq \sqrt{1-\frac{C_sn\log{n}}{d}}$; and
    \item $|\Gamma_i| \leq C_s\sqrt{\frac{\log{n}}{d}}$ for every $3 \le i \le n$.
\end{enumerate}
\end{lemma}

\begin{lemma} \label{lem:main_reduction_rgg}
Suppose that $d \gg n\log^4{n}$ and $p \gg n^{-3}$. Let $C_s$ be as in Lemma~\ref{lem:rgg_remainder_concentration} and let $E_{\mathsf{rem}}$ be the event that all the three events in the statement of Lemma~\ref{lem:rgg_remainder_concentration} hold. We have the following two bounds on $|Q_0-p| \cdot \mathbbm{1}(E_{\mathsf{rem}})$
\begin{align*}
&|Q_0-p| \cdot \mathbbm{1}(E_{\mathsf{rem}})\leq O_n\left(\frac{pn}{d}\log{ p^{-1}}\right) + C_s \psi_{d}\left(t_{p,d}\right) \cdot \bigr|t^{\prime}_{p,d}-t_{p,d}\bigr| \\
&|Q_0-p| \cdot \mathbbm{1}(E_{\mathsf{rem}})\leq O_n\left(\frac{pn}{d}\log{ p^{-1}}\right) + O_n\left(p\sqrt{\frac{n\log{ p^{-1}}}{d}} \cdot \log^{3/2}{n}\right)
\end{align*}
\end{lemma}

\begin{lemma}
\label{lem:expected_deviation_bound} 
Suppose that $d \gg n\log^4{n}$ and $p \in (0, 1/2]$ satisfies that $p \gg n^{-3}$. Then we have that
$$\mathbb{E}\left[ \bigr|t^{\prime}_{p,d}-t_{p,d}\bigr|^2 \cdot \mathbbm{1}(E_{\mathsf{rem}}) \right] = O_n\left( \frac{n^2 \log^3 n}{d^3} \right) + O_n\left(\frac{n}{d^2}\right) $$
\end{lemma}

With these lemmas, we now proceed to directly bound $\kl(\nurgg||\mu)$. Applying conditional Jensen's inequality to Equation~\ref{eq:log_sobolev_rgg} yields that
\begin{equation}
\kl(\nurgg||\mu) \leq \frac{N}{p(1-p)} \cdot \mathbb{E}|Q-p|^2 \leq \frac{N}{p(1-p)} \cdot \mathbb{E}|Q_0-p|^2 \label{eqn:Jensens1}
\end{equation}
We now estimate the right-hand side above using the results in Lemma~\ref{lem:main_reduction_rgg}. Applying the bounds in the above three lemmas yields that
\allowdisplaybreaks
\begin{align*}
    \kl(\nurgg||\mu) &\leq \frac{N}{p(1 - p)} \cdot \mathbb{E}|Q_0-p|^2 \\
    &= \frac{N}{p(1 - p)} \cdot \mathbb{E} \left[ |Q_0-p|^2 \cdot \mathbbm{1}(E_{\mathsf{rem}}) \right] + \frac{N}{p(1 - p)} \cdot \mathbb{E} \left[ |Q_0-p|^2 \cdot \mathbbm{1}(E^{c}_{\mathsf{rem}}) \right] \\
    &\lesssim \frac{N}{p} \cdot \mathbb{E}\left[\left(\frac{pn}{d} \cdot \log{ p^{-1}}  + C_s \psi_{d}\left(t_{p,d}\right) \cdot \bigr|t^{\prime}_{p,d}-t_{p,d}\bigr|\right)^2\mathbbm{1}(E_{\mathsf{rem}})\right] + \frac{N}{pn^s}\\
    &\lesssim \frac{n^4p}{d^2}\log^2{ p^{-1}}  + \frac{n^2}{p} \cdot \mathbb{E} \left[ \psi^2_{d}\left(t_{p,d}\right) \cdot \bigr|t^{\prime}_{p,d}-t_{p,d}\bigr|^2 \cdot \mathbbm{1}(E_{\mathsf{rem}}) \right] + \frac{1}{pn^{s-2}}
\end{align*}
Note that in the second inequality, we used the fact that $1 - p = \Omega_n(1)$, and in the last inequality, we used the fact that $(x+y)^2 \leq 2x^2 + 2y^2$.
Applying Lemma~\ref{lem:psi_properties}, we have that $p \in (0, 1/2]$ implies that $\psi_d(t_{p,d}) \leq Cp\sqrt{d\log{ p^{-1}}}$. Combining this with Lemma \ref{lem:expected_deviation_bound} now yields that
\begin{align*}
\kl(\nurgg||\mu) &\lesssim \frac{n^4p}{d^2}\log^2{ p^{-1}}  + n^2 pd \log p^{-1} \cdot \mathbb{E} \left[ \bigr|t^{\prime}_{p,d}-t_{p,d}\bigr|^2 \cdot \mathbbm{1}(E_{\mathsf{rem}}) \right] + \frac{1}{pn^{s-2}} \\
&\lesssim \frac{n^4p}{d^2}\log^2{ p^{-1}}  + \frac{n^4 p \log p^{-1} \log^3 n}{d^2} + \frac{n^3 p \log p^{-1}}{d}+ \frac{1}{pn^{s-2}} \\
&\lesssim \frac{n^3 p \log p^{-1}}{d}+ \frac{1}{pn^{s-2}}
\end{align*}
where the last inequality follows from the fact that $d \gg n \log^4 n$,  $\log p^{-1} = O(\log n)$ and $p = \Omega_n(n^{-2} \log n)$. Taking $s = 5$ yields that $\kl(\nurgg||\mu) \to 0$ if $d \gg n^3 p \log p^{-1}$.

\subsection{Sharper Bounds in the Sparse Case}

In this section, we prove the conclusion of Theorem \ref{thm:rgg} under the condition $d \gg p^2 n^{7/2} (\log n)^3 \sqrt{\log p^{-1}}$, which is tighter in the sparse case. The argument reduces bounding $\bE |Q - p|^2$ to bounding the total variation between $\grgg_{\sim e_0}$ and $\grgg_{\sim e_0}$ conditioned on the event $e_0 \in E(G)$. This quantity is then upper bounded by an explicit coupling on the vectors $X_i$.

We begin by observing that $Q = \mathbb{E}\left[Q_0\bigr|\sigma\left(\grgg_{\sim e_0}\right)\right]$ implies by Jensen's inequality that
\begin{equation}
|Q-p| \leq \mathbb{E}\left[|Q_0-p| \, \bigr|\sigma\left(\grgg_{\sim e_0}\right)\right] \label{eqn:Jensens2}
\end{equation}
Since $Q$ is $\sigma\left(\grgg_{\sim e_0}\right)$-measurable, we have that
$$|Q-p|^2 \leq \mathbb{E}\left[|Q_0-p| \cdot |Q-p| \, \bigr|\sigma\left(\grgg_{\sim e_0}\right)\right]$$
Substituting this and the definition of $E_{\mathsf{rem}}$ into Equation~\ref{eq:log_sobolev_rgg} now yields that
\begin{align*}
    \kl\left(\nurgg\bigr|\bigr|\mu\right) &\leq \frac{N}{p(1 - p)} \cdot \mathbb{E}\left[|Q-p|^2\right] \\
    &\lesssim \frac{n^2}{p} \cdot \mathbb{E}\left[|Q-p| \cdot |Q_0-p|\right] \\
    &= \frac{n^2}{p} \cdot \mathbb{E}\left[|Q-p| \cdot |Q_0-p| \cdot \mathbbm{1}(E_{\mathsf{rem}})\right] + \frac{n^2}{p} \cdot \mathbb{E}\left[|Q-p| \cdot |Q_0-p| \cdot \mathbbm{1}(E^c_{\mathsf{rem}})\right] \\
    &\leq \frac{n^2}{p} \cdot \mathbb{E}\left[|Q-p| \cdot |Q_0-p| \cdot \mathbbm{1}(E_{\mathsf{rem}})\right] +\frac{1}{pn^{s-2}}
\end{align*}
Note that the last inequality follows from the upper bound on $\bP[E_{\mathsf{rem}}^c]$ in Lemma \ref{lem:rgg_remainder_concentration}. Applying the second bound in Lemma~\ref{lem:main_reduction_rgg} now yields that
\begin{equation} \label{eq:hybrid_bound_initial}
\kl\left(\nurgg\bigr|\bigr|\mu\right) \lesssim \left(\frac{n^3\log{ p^{-1}}}{d} + n^{5/2}\log^{3/2}(n)\sqrt{\frac{\log{ p^{-1}}}{d}}\right) \cdot \mathbb{E}\left[|Q-p|\right] +\frac{1}{pn^{s-2}}
\end{equation}

It suffices to upper bound $\mathbb{E} |Q-p|$. Recall that $\nurgg$ denotes the probability mass function of $\grgg$ and $e_0$ denotes the edge $\{1, 2\}$. Let $\nurgg_{\sim e_0}$ denote the marginal distribution of $\grgg$ restricted to all edges that are not $\{1, 2\}$, and let $(\nurgg_{\sim e_0})^+$ denote the distribution of $\grgg$ conditioned on the event $e_0 \in E(\grgg)$. We now make a simple but essential observation that will allow us to upper bound $\bE|Q - p|$ by constructing a coupling between $\nurgg_{\sim e_0}$ and $(\nurgg_{\sim e_0})^+$. 

\begin{proposition}\label{lem:l_1_identity}
It holds that
$$\mathbb{E}[|Q-p|] = 2p \cdot \tv\left(\left(\nurgg_{\sim e_0}\right)^{+},\nurgg_{\sim e_0}\right)$$
\end{proposition}

\begin{proof}
Let $\Omega_{\sim e_0}$ denote the set of simple graphs on the vertex set $[n]$ that do not include the edge $\{1, 2\}$. Note that $Q$ can be written as $Q = \nurgg_{e_0}(1 | \grgg_{\sim e_0})$ where $\nurgg_{e_0}$ denotes the probability mass function of $\mathbbm{1}(e_0 \in E(\grgg))$ conditioned on $\grgg_{\sim e_0}$. We now have that
\begin{align*}
\mathbb{E}[|Q-p|] &= \mathbb{E}_{\grgg \sim \nurgg} \left[ \left|\nurgg_{e_0}(1 | \grgg_{\sim e_0}) - p \right| \right] \\
    &= \sum_{\grgg_{\sim e_0} \in \Omega_{\sim e_0}} \nurgg_{\sim e_0}(\grgg_{\sim e_0}) \cdot \bigr|p - \nurgg_{e_0}(1 | \grgg_{\sim e_0})\bigr| \\
    &= \sum_{\grgg_{\sim e_0} \in \Omega_{\sim e_0}} p \cdot \biggr|\nurgg_{\sim e_0}(\grgg_{\sim e_0}) - \frac{\nurgg_{e_0}(1 | \grgg_{\sim e_0}) \cdot \nurgg_{\sim e_0}(\grgg_{\sim e_0})}{p}\biggr| \\
    &= \sum_{\grgg_{\sim e_0} \in \Omega_{\sim e_0}} p \cdot \biggr|\nurgg_{\sim e_0}(\grgg_{\sim e_0}) - (\nurgg_{\sim e_0})^+(\grgg_{\sim e_0})\biggr| \\
    &= 2p \cdot \tv\left(\left(\nurgg_{\sim e_0}\right)^{+},\nurgg_{\sim e_0}\right)
\end{align*}
which proves the proposition.
\end{proof}

We now will construct a coupling between $\nurgg_{\sim e_0}$ and $(\nurgg_{\sim e_0})^+$. Note that the collection of variables $\langle X_i, X_j \rangle$ for $1 \le i < j \le n$ is invariant to orthogonal rotations of the vectors $X_i$. Therefore we may assume without loss of generality that $X_1 = (1,0,\dots,0)$ and $X_2, \dots, X_n$ are sampled i.i.d. from the Haar measure on $\mathbb{S}^{d - 1}$. Now let $\psi_{d,p}^{+}$ denote the density of $Z \sim \psi_d$ conditioned on the event $Z \ge t_{p, d}$. In other words, let
$$\psi^{+}_{d,p}(x) = \frac{\mathbbm{1}(x\geq t_{p,d}) \cdot \psi_d(x)}{p}$$
for each $x \in \mathbb{R}$. Now let $X_2'$ be given by
$$X_2' = \left( \tau, \gamma X_{22}, \gamma X_{23}, \dots, \gamma X_{2d} \right) \quad \text{where} \quad \gamma = \sqrt{\frac{1 - \tau^2}{1 - X_{21}^2}} \quad \text{and} \quad \tau \sim \psi^{+}_{d,p}$$
and where $\tau$ is independent of $\sigma(X_1, X_2, \dots, X_n)$. Note that $\gamma$ is such that $\| X_2' \|_2 = 1$. Let $\pr{gg}_{t_{p,d}}^{\sim e_0}(X_1, X_2, \dots, X_n)$ denote $\pr{gg}_{t_{p,d}}(X_1, X_2, \dots, X_n)$ without the edge $\{1, 2\}$. By definition, we have that $\pr{gg}_{t_{p,d}}^{\sim e_0}(X_1, X_2, \dots, X_n)\sim \nurgg_{\sim e_0}$. Now observe that
$$\pr{gg}_{t_{p,d}}^{\sim e_0}(X_1, X_2', \dots, X_n) \sim (\nurgg_{\sim e_0})^+$$
This holds since $\mL(X_1, X_2', \dots, X_n)$ is by construction the law of $X_1, X_2, \dots, X_n$ conditioned on the event that $\{X_{21} \ge t_{p, d} \}$. Furthermore the event $\{X_{21} \ge t_{p, d} \}$ exactly coincides with the event $\{e_0 \in E(G)\}$ where $G = \pr{gg}_{t_{p,d}}(X_1, X_2, \dots, X_n)$. The coupling characterization of total variation now implies that
\begin{equation} \label{eqn:coupling}
\tv\left(\left(\nurgg_{\sim e_0}\right)^{+},\nurgg_{\sim e_0}\right) \le \bP\left[ \pr{gg}_{t_{p,d}}^{\sim e_0}(X_1, X_2', \dots, X_n) \neq \pr{gg}_{t_{p,d}}^{\sim e_0}(X_1, X_2, \dots, X_n) \right]
\end{equation}
Let $C > 0$ be a fixed constant and define the event
$$E_{\mathsf{coup}} = \biggr\{|\tau| \leq C\sqrt{\frac{\log{n}}{d}} \quad \text{and} \quad |X_{i1}| \leq C\sqrt{\frac{\log{n}}{d}} \quad \text{ for all } 2 \le i \le n \biggr\}$$
Now observe that for any fixed $x > 0$, it holds that $\bP[|\tau| > x] \le p^{-1} \cdot \Psi_d(x) \lesssim n^2 \cdot \Psi_d(x)$, as $\tau \sim \psi_{d, p}^+$ and $p^{-1} = O_n(n^2/\log n)$. Since we also have that $X_{i1} \sim \psi_d$ for each $2 \le i \le n$, Lemma \ref{lem:distributional_approx_on_the_sphere} and a union bound imply that we can choose $C$ large enough so that $\mathbb{P}[E_{\mathsf{coup}}] \geq 1- n^{-s}$ for some fixed $s >0$. We now observe that the two graphs $\pr{gg}_{t_{p,d}}^{\sim e_0}(X_1, X_2', \dots, X_n)$ and $\pr{gg}_{t_{p,d}}^{\sim e_0}(X_1, X_2, \dots, X_n)$ can only differ in edges of the form $\{2, i\}$ where $3 \le i \le n$. Furthermore, they differ in the edge $\{2, i\}$ exactly when $\mathbbm{1}( \langle X'_2, X_i \rangle \ge t_{p, d}) \neq \mathbbm{1}( \langle X_2, X_i \rangle \ge t_{p, d})$. Now note that
$$\langle X_2^{\prime}, X_i\rangle = \tau X_{i1} + \gamma \sum_{j=2}^d X_{2j} X_{ij} = \tau X_{i1} + \sqrt{\frac{1-\tau^2}{1-X_{21}^2}} \cdot \left(\langle X_2,X_i\rangle - X_{21} X_{i1}\right)$$
It therefore follows that if $E_{\mathsf{coup}}$ holds then
$$\left| \langle X_2^{\prime}, X_i\rangle - \langle X_2,X_i\rangle \right| \le |\tau | \cdot |X_{i1}| + \sqrt{\frac{1-\tau^2}{1-X_{21}^2}} \cdot |X_{21}| \cdot |X_{i1}| + \left| \sqrt{\frac{1-\tau^2}{1-X_{21}^2}} -1 \right| = O_n\left( \frac{\log n}{d} \right)$$
where here we used the fact that $|\langle X_2, X_i \rangle | \le \| X_2 \|_2 \cdot \| X_i \|_2 = 1$ by Cauchy-Schwarz. Let $\delta$ denote the upper bound in the above inequality. Observe that if $E_{\mathsf{coup}}$ holds then the only way that $\mathbbm{1}( \langle X'_2, X_i \rangle \ge t_{p, d}) \neq \mathbbm{1}( \langle X_2, X_i \rangle \ge t_{p, d})$ can hold is if $|\langle X_2, X_i \rangle - t_{p, d}| \le \delta$. Combining these observations with the fact that $\langle X_2, X_i \rangle \sim \psi_d$ now yields that
\begin{align*}
\bP\left[ \left\{ \mathbbm{1}( \langle X'_2, X_i \rangle \ge t_{p, d}) \neq \mathbbm{1}( \langle X_2, X_i \rangle \ge t_{p, d}) \right\} \cap E_{\mathsf{coup}} \right] &\le \bP\left[ |\langle X_2, X_i \rangle - t_{p, d}| \le \delta \right] \\
&= \int_{t_{p,d}-\delta}^{t_{p,d}+\delta}\psi_d(x)dx \\
&\le 2\delta \cdot \sup_{|x - t_{p, d}| \le \delta} \psi_d(x)
\end{align*}
By (3) in Lemma~\ref{lem:psi_properties}, it follows that
$$\sup_{|x - t_{p, d}| \le \delta} \psi_d(x) \leq \psi_d(t_{p,d}) \cdot e^{3t_{p,d}d\delta} = \psi_d(t_{p,d}) \cdot (1+o_n(1)) = O_n \left( p \sqrt{d \log n} \right)$$
where the second and third bounds follow since $t_{p,d} = O_n\left(\sqrt{\frac{\log n}{d}}\right)$, $\delta = O_n\left(\frac{\log n}{d}\right)$ and $d \gg \log^3{n}$. Putting this all together with Proposition \ref{lem:l_1_identity} and Equation \ref{eqn:coupling} now yields that
\begin{align*}
\mathbb{E}[|Q-p|] &\le 2p \cdot \bP\left[ \pr{gg}_{t_{p,d}}^{\sim e_0}(X_1, X_2', \dots, X_n) \neq \pr{gg}_{t_{p,d}}^{\sim e_0}(X_1, X_2, \dots, X_n) \right] \\
&\le 2p \cdot \bP\left[ E_{\mathsf{coup}}^c \right] + 2p \cdot \sum_{i = 3}^n \bP\left[ \left\{ \mathbbm{1}( \langle X'_2, X_i \rangle \ge t_{p, d}) \neq \mathbbm{1}( \langle X_2, X_i \rangle \ge t_{p, d}) \right\} \cap E_{\mathsf{coup}} \right] \\
&\lesssim 2p n^{-s} + 2p \cdot (n - 3) \cdot 2 \delta \cdot p \sqrt{d \log n} \\
&\lesssim p n^{-s} + \frac{p^2 n \log^{3/2} n}{d^{1/2}}
\end{align*}
where the second inequality follows from a union bound. Substituting this bound into Equation~\eqref{eq:hybrid_bound_initial} now yields that
\begin{align*}
\kl\left(\nurgg\bigr|\bigr|\mu\right) &\lesssim \frac{n^4p^2\log p^{-1} \log^{3/2}{n}}{d^{3/2}} + \frac{n^{7/2}p^2\log^{3}(n)\sqrt{\log{ p^{-1}}}}{d} \\
&\quad \quad + \frac{1}{pn^{s-2}} + pn^{-s} \cdot \left(\frac{n^3\log{ p^{-1}}}{d} + n^{5/2}\log^{3/2}(n)\sqrt{\frac{\log{ p^{-1}}}{d}}\right)
\end{align*}
Varying $s$ only changes the constant with which the $\lesssim$ above holds. Picking $s > 4$ thus yields that $\kl\left(\nurgg\bigr|\bigr|\mu\right) \to 0$ as $n \to \infty$ if $d \gg p^2 n^{7/2} (\log n)^3 \sqrt{\log p^{-1}}$ and $p = \Omega_n(n^{-2} \log n)$. Applying Pinsker's inequality completes the proof of Theorem \ref{thm:rgg}.

\section*{Acknowledgements}

We thank Yury Polyanskiy for inspiring discussions on related topics and thank Philippe Rigollet for directing us to \cite{bubeck2016testing}. This work was supported in part by the grants ONR N00014-17-1-2147.

\bibliography{GB_BIB.bib}
\bibliographystyle{alpha}

\begin{appendices}

\section{Appendix: Random Intersection Graphs and Matrices}
\label{sec:rim_appendix}

\subsection{Variance of the Signed Triangle Count in $\pr{RIG}(n, d, p)$}
\label{subsec:rig-triangles}

The main purpose of this section is to prove Lemma \ref{lem:rigsignedtrianglesvar}, which computes the variance of $T_s(G)$ for $G \sim \pr{rig}(n, d, p)$. The proof follows a similar structure to the proof of Lemma \ref{lem:signedtriangleexp} but is more computationally involved.

\begin{proof}[Proof of Lemma \ref{lem:rigsignedtrianglesvar}]
Let $\tau_{ijk} = (e_{ij} - p)(e_{ik} - p)(e_{jk} - p)$ for each $1 \le i < j < k \le n$. It holds that
\begin{align*}
\textnormal{Var}[T_s(G)] &= \sum_{1 \le i < j < k \le n} \sum_{1 \le i' < j' < k' \le n} \textnormal{Cov}\left[ \tau_{ijk},  \tau_{i'j'k'} \right] \\
&= \binom{n}{3} \cdot \text{Var}[\tau_{123}] + \frac{4!}{2! \cdot 2!} \cdot \binom{n}{4} \cdot \textnormal{Cov}\left[ \tau_{123},  \tau_{124} \right] + \frac{5!}{2! \cdot 2!} \cdot \binom{n}{5} \cdot \textnormal{Cov}\left[ \tau_{123},  \tau_{145} \right] \numberthis \label{eqn:varexp}
\end{align*}
The second equality follows by symmetry among vertex labels and the fact that if $\{i, j, k\} \cap \{i', j', k'\} = \emptyset$ then $\tau_{ijk}$ and $\tau_{i'j'k'}$ are independent. Note that the second coefficient is the number of ways to choose two sets of three vertices that intersect in two elements and the third coefficient is the number of ways to choose these sets so that they intersect in one element. By Lemma \ref{lem:signedtriangleexp}, we have that
$$q = \bE[\tau_{123} = 1] = (1 - p)^3 \cdot \left[ d\delta^3 + O_n(d\delta^4) \right]$$
Now note that
\allowdisplaybreaks
\begin{align*}
&\text{Var}[\tau_{123}] = \bE[\tau_{123}^2] - q^2, \quad \quad \textnormal{Cov}\left[ \tau_{123},  \tau_{124} \right] = \bE[\tau_{123}\tau_{124}] - q^2 \quad \text{and} \quad \\
&\textnormal{Cov}\left[ \tau_{123},  \tau_{145} \right] = \bE[\tau_{123}\tau_{145} = 1] - q^2
\end{align*}
We will begin by computing $\bE[\tau_{123}^2]$. Let $P$ and $Q$ be as in Lemma \ref{lem:signedtriangleexp}. Now note that
\begin{align*}
\bE[\tau_{123}^2] &= \bE\left[ (e_{12} - p)^2(e_{13} - p)^2(e_{23} - p)^2 \right] \\
&= \bE\left[ \prod_{\{i, j\} = \{1, 2\}, \{1, 3\}, \{2, 3\}} \left[ (1 - p)^2 - (1 - 2p)(1 - e_{ij}) \right] \right] \\
&= -\sum_{x \in \{0, 1\}^3} (1 - p)^{2|x|}(-1)^{|x|} (1 - 2p)^{3 - |x|} \cdot Q(x_1, x_2, x_3) \\
&= (1 - p)^6 - 3(1 - p)^4 (1 - 2p) (1 - \delta)^d(1 + \delta)^d + 3(1 - p)^2 (1 - 2p)^2 (1 - \delta)^d(1 + \delta - \delta^2)^d \\
&\quad \quad - (1 - 2p)^3 (1 + 2\delta)^d (1 - \delta)^{2d}
\end{align*}
where the last two equalities follow from the expressions for $Q$ in Lemma \ref{lem:signedtriangleexp}. Further simplifying and applying the estimates in Equations \ref{eqn:firstapprox} and \ref{eqn:secondapprox} yields that the above quantity is equal to
\allowdisplaybreaks
\begin{align*}
&(1 - p)^6 - 3(1 - p)^5 (1 - 2p) + 3(1 - p)^4 (1 - 2p)^2 \cdot \left( 1 + \frac{\delta^3}{(1 - \delta^2)(1 + \delta)} \right)^d \\
&\quad \quad \quad \quad  - (1 - 2p)^3 (1 - p)^3 \cdot \left( 1 + \frac{2\delta^3 + \delta^4}{(1 - \delta^2)(1 + \delta)^2} \right)^d \\
&\quad \quad = (1 - p)^3 \cdot \left[ (1 - p) - (1 - 2p) \right]^3 + 3(1 - p)^4 (1 - 2p)^2 \cdot \left[ \left( 1 + \frac{\delta^3}{(1 - \delta^2)(1 + \delta)} \right)^d - 1\right] \\
&\quad \quad \quad \quad - (1 - 2p)^3 (1 - p)^3 \cdot \left[ \left( 1 + \frac{2\delta^3 + \delta^4}{(1 - \delta^2)(1 + \delta)^2} \right)^d - 1 \right] \\
&\quad \quad = p^3 (1 - p)^3 + 3(1 - p)^4 (1 - 2p)^2 d \delta^3 - (1 - 2p)^3 (1 - p)^3 d\delta^3 + O_n(d\delta^4) \\
&\quad \quad = p^3 (1 - p)^3 + (2 - p)(1 - p)^3 (1 - 2p)^2 d \delta^3 + O_n(d\delta^4)
\end{align*}

We now will estimate $\bE[\tau_{123}\tau_{124}]$ using a similar method to Lemma \ref{lem:signedtriangleexp}. Let $P' : \{0, 1\}^5 \to [0, 1]$ be such that $P'(x_1, x_2, x_3, x_4, x_5)$ is the probability that $e_{12} = x_1$, $e_{13} = x_2$, $e_{14} = x_3$, $e_{23} = x_4$ and $e_{24} = x_5$. Define $Q' : \{0, 1\}^5 \to [0, 1]$ as
$$Q'(x_1, x_2, x_3, x_4, x_5) = \sum_{y \subseteq x} P'(y_1, y_2, y_3, y_4, y_5)$$
As in Lemma \ref{lem:triangleexp}, the events whose probabilities are given by the values of $Q'$ are each the product of events over the individual elements of $[d]$. For no edges to be present in the triangles $\{1, 2, 3\}$ or $\{1, 2, 4\}$, each $i \in [d]$ must be in at most one of $S_1, S_2, S_3, S_4$ or is in both of $S_3$ and $S_4$. Thus
\begin{align*}
P'(0, 0, 0, 0, 0) = Q'(0, 0, 0, 0, 0) &= \left[ (1 - \delta)^4 + 4\delta(1 - \delta)^3 + \delta^2(1 - \delta)^2 \right]^d \\
&= \left(1 + 2\delta - 2\delta^2 \right)^d(1 - \delta)^{2d}
\end{align*}
Similarly, if $|x| = x_1 + x_2 + x_3 + x_4 + x_5 = 1$, then
\begin{align*}
Q'(x_1, x_2, x_3, x_4, x_5) &= \left[ (1 - \delta)^4 + 4\delta(1 - \delta)^3 + 2\delta^2(1 - \delta)^2 \right]^d \\
&= \left(1 + 2\delta - \delta^2 \right)^d(1 - \delta)^{2d}
\end{align*}
If $|x| = 2$ and $x \neq (0, 1, 1, 0, 0), (0, 0, 0, 1, 1)$, then it follows that
\begin{align*}
Q'(x_1, x_2, x_3, x_4, x_5) &= \left[ (1 - \delta)^4 + 4\delta(1 - \delta)^3 + 3\delta^2(1 - \delta)^2 \right]^d \\
&= \left(1 + 2\delta \right)^d(1 - \delta)^{2d}
\end{align*}
If $x = (0, 1, 1, 0, 0), (0, 0, 0, 1, 1)$, then each $i \in [d]$ can also possibly be in the three sets $S_1, S_3, S_4$ and $S_2, S_3, S_4$, respectively. Therefore
\begin{align*}
Q'(0, 1, 1, 0, 0) = Q'(0, 0, 0, 1, 1) &= \left[ (1 - \delta)^4 + 4\delta(1 - \delta)^3 + 3\delta^2(1 - \delta)^2 + \delta^3(1 - \delta) \right]^d \\
&= \left(1 + \delta - 2\delta^2 + \delta^3 \right)^d(1 - \delta)^{d}
\end{align*}
We now consider the cases where $|x| = 3$. When $|x| = 3$, there are always four allowed pairs of sets that any $i \in [d]$ can be in -- the three edges of $x$ and $\{3, 4\}$. However, the number of allowed triples varies with $x$. If $x = (1, 1, 0, 0, 1), (1, 0, 1, 1, 0)$, then there are no allowed triples and
\begin{align*}
Q'(1, 1, 0, 0, 1) = Q'(1, 0, 1, 1, 0) &= \left[ (1 - \delta)^4 + 4\delta(1 - \delta)^3 + 4\delta^2(1 - \delta)^2 \right]^d \\
&= (1 + \delta)^{2d}(1 - \delta)^{2d}
\end{align*}
If $|x| = 3$ and $x \neq (1, 1, 0, 0, 1), (1, 0, 1, 1, 0)$, then there is one allowed triple and
\begin{align*}
Q'(x_1, x_2, x_3, x_4, x_5) &= \left[ (1 - \delta)^4 + 4\delta(1 - \delta)^3 + 4\delta^2(1 - \delta)^2 + \delta^3(1 - \delta) \right]^d \\
&= \left(1 + \delta - \delta^2 \right)^d(1 - \delta)^{d}
\end{align*}
If $|x| = 4$, then there is one forbidden pair of sets and two forbidden triples. Therefore
\allowdisplaybreaks
\begin{align*}
Q'(x_1, x_2, x_3, x_4, x_5) &= \left[ (1 - \delta)^4 + 4\delta(1 - \delta)^3 + 5\delta^2(1 - \delta)^2 + 2\delta^3(1 - \delta) \right]^d \\
&= \left(1 + \delta \right)^d(1 - \delta)^{d}
\end{align*}
Furthermore $Q'(1, 1, 1, 1, 1) = 1$. Now we have that
\begin{align*}
\bE[\tau_{123}\tau_{124}] &= \bE\left[ (e_{12} - p)^2 (e_{13} - p) (e_{14} - p) (e_{23} - p) (e_{24} - p) \right] \\
&= \bE\left[ \left[(1 - p)^2 - (1 - 2p)(1 - e_{12}) \right] \times \prod_{\{i, j\} = \{1, 3\}, \{1, 4\}, \{2, 3\}, \{2, 4\}} \left[ (1 - p) - (1 - e_{ij}) \right] \right] \\
&= -\sum_{x \in \{0, 1\}^5} (-1)^{|x|} (1 - p)^{2x_1} (1 - 2p)^{1 - x_1} (1 - p)^{x_2 + x_3 + x_4 + x_5} \cdot Q'(x_1, x_2, x_3, x_4, x_5) \\
&= (1 - p)^6 - (1-p)^4 \cdot \left[ 4(1 - p) + (1 - 2p) \right] \cdot \left(1 + \delta \right)^d(1 - \delta)^{d} \\
&\quad \quad + (1 - p)^3 \cdot \left[ 4(1 - p) + 4(1 - 2p) \right] \cdot \left(1 + \delta - \delta^2 \right)^d(1 - \delta)^{d} \\
&\quad \quad + 2(1 - p)^4 \cdot (1 + \delta)^{2d}(1 - \delta)^{2d} \\
&\quad \quad - (1 - p)^2 \cdot \left[ 4(1 - p) + 4(1 - 2p) \right] \cdot \left(1 + 2\delta \right)^d(1 - \delta)^{2d} \\
&\quad \quad - 2(1 - p)^2 (1 - 2p) \cdot \left(1 + \delta - 2\delta^2 + \delta^3 \right)^d(1 - \delta)^{d} \\
&\quad \quad + (1 - p) \cdot \left[ (1 - p) + 4(1 - 2p) \right] \cdot \left(1 + 2\delta - \delta^2 \right)^d(1 - \delta)^{2d} \\
&\quad \quad - (1 - 2p)\left(1 + 2\delta - 2\delta^2 \right)^d(1 - \delta)^{2d}
\end{align*}
This quantity can be rewritten as the following expression which is more convenient to estimate.
\begin{align*}
&(1 - p)^6 - (1-p)^5 (5 - 6p) + 4(1 - p)^5 (2 - 3p) \cdot \left(1 + \frac{\delta^3}{(1 + \delta)(1 - \delta^2)} \right)^d \\
&\quad \quad \quad \quad + 2(1 - p)^6 - 4(1 - p)^5 (2 - 3p) \cdot \left(1 + \frac{2\delta^3 + \delta^4}{(1 - \delta^2)(1 + \delta)^2} \right)^d \\
&\quad \quad \quad \quad - 2(1 - p)^5 (1 - 2p) \cdot \left(1 + \frac{\delta^3 - \delta^4 - \delta^5}{(1 - \delta^2)^2(1 + \delta)} \right)^d \\
&\quad \quad \quad \quad + (1 - p)^5 (5 - 9p) \cdot \left(1 + \frac{4\delta^3 + \delta^4 - 2\delta^5 - \delta^6}{(1 - \delta^2)^2(1 + \delta)^2} \right)^d \\
&\quad \quad \quad \quad - (1 - p)^5 (1 - 2p) \left(1 + \frac{6\delta^3 - 6 \delta^5 - 2\delta^6 + 2\delta^7 + \delta^8}{(1 - \delta^2)^3(1 + \delta)^2} \right)^d \\
&\quad \quad = 4(1 - p)^5(2 - 3p) \cdot d\delta^3 - 4(1 - p)^5(2 - 3p) \cdot 2d\delta^3 - 2(1 - p)^5(1 - 2p) \cdot 2\delta^3 \\
&\quad \quad \quad \quad + (1 - p)^5(5 - 9p) \cdot 4d\delta^3 - (1 - p)^5(1 - 2p) \cdot 6d\delta^3 + O_n(d\delta^4) \\
&\quad \quad = 2(1 - p)^5 (1 - 2p) d \delta^3 + O_n(d\delta^4)
\end{align*}
The second last equality follows from substituting estimates of the form
$$\left(1 + \frac{\delta^3}{(1 + \delta)(1 - \delta^2)} \right)^d = 1 + d\delta^3 + O_n(d\delta^4)$$
and analogous estimates for the other $d$th powers in the expression. These estimates can be established using similar bounds to those used to derive Equations \ref{eqn:firstapprox} and \ref{eqn:secondapprox}. Observe that the terms that are not multiples of $d\delta^3$ after substituting these estimates sum to zero.

We now will estimate $\bE[\tau_{123}\tau_{145}]$ using a slightly different method. Note that $\tau_{123}$ is $\sigma(S_1, S_2, S_3)$-measurable and $\tau_{145}$ is $\sigma(S_1, S_4, S_5)$-measurable. Thus conditioned on $S_1$, the random variables $\tau_{123}$ and $\tau_{145}$ are independent. Furthermore, because of symmetry among the elements in $[d]$, $\tau_{123}$ and $\tau_{145}$ are independent conditioned on $|S_1|$. Let $\tau_{123}^m = \bE\left[\tau_{123} \big| |S_1| = m\right]$ and observe that conditional independence yields that $\bE[\tau_{123}\tau_{145}] = \bE_{m \sim \mL(|S_1|)} \left[ (\tau_{123}^m)^2 \right]$. We now will compute $\tau^m_{123}$. Let $P_m : \{0, 1\}^3 \to [0, 1]$ be such that $P_m(x_1, x_2, x_3)$ is the probability that $e_{12} = x_1, e_{13} = x_2$ and $e_{23} = x_3$ given $|S_1| = m$. Define $Q_m : \{0, 1\}^3 \to [0, 1]$ to be
$$Q_m(x_1, x_2, x_3) = \sum_{y \subseteq x} P_m(x_1, x_2, x_3)$$
For no edges in triangle $\{1, 2, 3\}$ to be present, each of the $m$ elements of $S_1$ cannot be in either $S_2$ or $S_3$ and each of the $d - m$ remaining elements must be in at most one of $S_2$ or $S_3$. Therefore
$$Q_m(0, 0, 0) = (1 - \delta)^{2m}( 1 - \delta^2)^{d - m}$$
For either no edges or just the edge $\{1, 2\}$ to be present, each element in $S_1$ must not be in $S_3$ and each of the $d - m$ remaining elements must be in at most one of $S_2$ or $S_3$. Similar conditions hold when $\{1, 2\}$ is replaced by $\{1, 3\}$ and thus
$$Q_m(1, 0, 0) = Q_m(0, 1, 0) = (1 - \delta)^{m}( 1 - \delta^2)^{d - m}$$
For at most the edge $\{2, 3\}$ to be present, each element of $S_1$ cannot be in $S_2$ or $S_3$ and thus
$$Q_m(0, 0, 1) = (1 - \delta)^{2m}$$
For just the edge $\{1, 2\}$ to not be present, each element of $S_1$ cannot be in $S_2$. Similar conditions hold when $\{1, 2\}$ is replaced by $\{1, 3\}$ and thus
$$Q_m(0, 1, 1) = Q_m(1, 0, 1) = (1 - \delta)^{m}$$
For $\{2, 3\}$ to not be present, it must hold that each of the $d - m$ elements not in $S_1$ are in one of $S_2$ or $S_3$. Thus
$$Q_m(1, 1, 0) = (1 - \delta^2)^{d - m}$$
Furthermore $Q_m(1, 1, 1) = 1$. Now we have that
\begin{align*}
\tau_{123}^m &= -\bE\left[\left((1 - e_{12}) - (1 - p) \right) \left((1 - e_{13}) - (1 - p) \right) \left((1 - e_{23}) - (1 - p) \right) \big| |S_1| = m\right] \\
&= - \sum_{x \in \{0, 1\}^3} (-1)^{|x|} (1 - p)^{|x|} \cdot Q_m(x_1, x_2, x_3) \\
&= (1 - p)^3 - 2(1 - p)^2(1 - \delta)^{m} - (1 - p)^2(1 - \delta^2)^{d - m} + 2(1 - p) \cdot (1 - \delta)^{m}( 1 - \delta^2)^{d - m} \\
&\quad \quad + (1 - p) (1 - \delta)^{2m} - (1 - \delta)^{2m}( 1 - \delta^2)^{d - m}
\end{align*}
Now note that $|S_1| \sim \text{Bin}(d, \delta)$ and thus $\bE_{m \sim \mL(|S_1|)} \left[ x^m \right] = \left( 1 - \delta + \delta x \right)^d$ for any $x > 0$, by the form of the moment generating function of the binomial distribution. Expanding $(\tau_{123}^m)^2$ and applying this identity now yields that
\allowdisplaybreaks
\begin{align*}
\bE_{m \sim \mL(|S_1|)} \left[ (\tau_{123}^m)^2 \right] &= (1 - p)^6 - 4(1 - p)^5 (1 - \delta^2)^d - 2(1 - p)^5 (1 - \delta^2)^d \left( 1 - \delta + \frac{\delta}{1 - \delta^2} \right)^d \\
&\quad \quad + 6(1 - p)^4 (1 - 2\delta^2 + \delta^3)^d + (1 - p)^4 (1 - \delta^2)^{2d} \left( 1 - \delta + \frac{\delta}{(1 - \delta^2)^2} \right)^d \\
&\quad \quad + 8(1 - p)^4 (1 - \delta^2)^d \left( 1 - \delta + \frac{\delta(1 - \delta)}{1 - \delta^2} \right)^d \\
&\quad \quad - 12(1 - p)^3 (1 - \delta^2)^d \left(1 - \delta + \frac{\delta(1 - \delta)^2}{1 - \delta^2} \right)^d \\
&\quad \quad - 4(1 - p)^3 (1 - 3\delta^2 + 3\delta^3 - \delta^4)^d \\
&\quad \quad - 4(1 - p)^3 (1 - \delta^2)^{2d} \left(1 - \delta + \frac{\delta(1 - \delta)}{(1 - \delta^2)^2} \right)^d \\
&\quad \quad + 6(1 - p)^2 (1 - \delta^2)^{2d} \left( 1 - \delta + \frac{\delta(1 - \delta)^2}{(1 - \delta^2)^2} \right)^d \\
&\quad \quad + (1 - p)^2 (1 - 4\delta^2 + 6\delta^3 - 4\delta^4 + \delta^5)^d \\
&\quad \quad + 8(1 - p)^2 (1 - \delta^2)^d \left( 1 - \delta + \frac{\delta(1 - \delta)^3}{1 - \delta^2} \right)^d \\
&\quad \quad - 4(1 - p) (1 - \delta^2)^{2d} \left( 1 - \delta + \frac{\delta(1 - \delta)^3}{(1 - \delta^2)^2} \right)^d \\
&\quad \quad - 2(1 - p) (1 - \delta^2)^d \left( 1 - \delta + \frac{\delta(1 - \delta)^4}{1 - \delta^2} \right)^d \\
&\quad \quad + (1 - \delta^2)^{2d} \left( 1 - \delta + \frac{\delta(1 - \delta)^4}{(1 - \delta^2)^2} \right)^d
\end{align*}
This quantity can be rewritten as the following expression which is more convenient to estimate.
\begin{align*}
&(1 - p)^6 - 4(1 - p)^6 - 2(1 - p)^6 \left( 1 + \frac{\delta^3}{1 - \delta^2} \right)^d \\
&\quad \quad + 14(1 - p)^6 \left(1 + \frac{\delta^3}{(1 - \delta^2)(1 + \delta)} \right)^d + (1 - p)^6 \left( 1 + \frac{2\delta^3 - \delta^5}{(1 - \delta^2)^2} \right)^d \\
&\quad \quad - 16(1 - p)^6 \left(1 + \frac{2\delta^3 + \delta^4}{(1 - \delta^2)(1 + \delta)^2} \right)^d \\
&\quad \quad - 4(1 - p)^6 \left(1 + \frac{3\delta^3 - \delta^4 - \delta^5}{(1 - \delta^2)^2(1 + \delta)}\right)^d \\
&\quad \quad + 6(1 - p)^6 \left( 1 + \frac{3\delta^3 + \delta^4 - 2\delta^5 - \delta^6}{(1 - \delta^2)^2(1 + \delta)^2} \right)^d \\
&\quad \quad + (1 - p)^6 \left(1 + \frac{6\delta^3 - 4\delta^4 - 3\delta^5 + \delta^6 + \delta^7}{(1 - \delta^2)^3 (1 + \delta)} \right)^d \\
&\quad \quad + 8(1 - p)^6 \left( 1 + \frac{4\delta^3 + \delta^4 - 2\delta^5 - \delta^6}{(1 - \delta^2)^2(1 + \delta)^2} \right)^d \\
&\quad \quad - 4(1 - p)^6 \left( 1 + \frac{5\delta^3 + 5\delta^4 - \delta^5 - 3\delta^6 - \delta^7}{(1 - \delta^2)^2(1 + \delta)^3} \right)^d \\
&\quad \quad - 2(1 - p)^6 \left( 1 + \frac{7\delta^3 - 6\delta^4 - 2\delta^5 + 2\delta^7 + \delta^8}{(1 - \delta^2)^3(1 + \delta)^2} \right)^d \\
&\quad \quad + (1 - p)^6 \left( 1 + \frac{8\delta^3 + 6\delta^4 - 6\delta^5 - 8\delta^7 + 3\delta^8 + \delta^9}{(1 - \delta^2)^3(1 + \delta)^3} \right)^d
\end{align*}
Now substitute estimates for each of the $d$th powers of the form $1 + cd\delta^3 + O_n(d\delta^4)$ for constants $c$ varying per term. For example, the first power can be estimated to be
$$\left( 1 + \frac{\delta^3}{1 - \delta^2} \right)^d = 1 + d\delta^3 + O_n(d\delta^4)$$
These estimates can be established using the same bounding argument used to derive Equations \ref{eqn:firstapprox} and \ref{eqn:secondapprox}. Observe that the sum of the constant terms and the multiples of $d\delta^3$ are zero after substituting these estimates into the expression above for $\bE_{m \sim \mL(|S_1|)} \left[ (\tau_{123}^m)^2 \right]$. Thus we obtain that
$$\bE_{m \sim \mL(|S_1|)} \left[ (\tau_{123}^m)^2 \right] = O_n(d\delta^4)$$
Now note that $q = O_n(d\delta^3)$ and therefore we have that
\begin{align*}
\text{Var}[\tau_{123}] &= \bE[\tau_{123}^2] - q^2 =  p^3 (1 - p)^3 + O_n(d\delta^3) \\
\textnormal{Cov}\left[ \tau_{123},  \tau_{124} \right] &= \bE[\tau_{123}\tau_{124}] - q^2 = O_n(d\delta^3) \\
\textnormal{Cov}\left[ \tau_{123},  \tau_{145} \right] &= \bE[\tau_{123}\tau_{145} = 1] - q^2 = O_n(d\delta^4)
\end{align*}
since $d\delta^2 = O_n(1)$. Substituting into Equation \ref{eqn:varexp} completes the proof of the lemma.
\end{proof}

For the sake of completeness, we show how to apply the approach in Lemma \ref{lem:signedtriangleexp} to compute $\bE[T(G)]$ where $G \sim \pr{rig}(n, d, p)$ in the following lemma.

\begin{lemma} \label{lem:triangleexp}
If $G \sim \pr{rig}(n, d, p)$ where $1 - p = (1 - \delta^2)^d = \Omega_n(1)$, then it follows that
$$\bE\left[ T(G) \right] = \binom{n}{3} \cdot \left[ p^3 + d \delta^3 (1 + 2p)(1 - p)^2 + O_n\left(d\delta^4 \right) \right]$$
\end{lemma}

\begin{proof}
Given three distinct vertices in $i, j, k \in [n]$, let $T_{ijk}$ denote the indicator for the event that $i, j$ and $k$ form a triangle in $G$. Linearity of expectation yields that
\begin{equation} \label{eqn:exp}
\bE\left[ T(G) \right] = \sum_{1 \le i < j < k \le n} \bE[T_{ijk}] = \binom{n}{3} \cdot \bP[T_{123} = 1]
\end{equation}
where the second equality holds by symmetry. Let $P$ and $Q$ be as in Lemma \ref{lem:signedtriangleexp}. The principle of inclusion-exclusion now yields that
\begin{align*}
\bP[T_{123} = 1] &= P(1, 1, 1) = \sum_{x \in \{0, 1\}^3} (-1)^{3 - |x|} \cdot Q(x_1, x_2, x_3) \\
&= 1 - 3 (1 - \delta)^d(1 + \delta)^d + 3 (1 - \delta)^d(1 + \delta - \delta^2)^d - (1 + 2\delta)^d (1 - \delta)^{2d} \\
&= \left(1 - (1 - \delta^2)^d \right)^3 + 3\left( 1 - 2 \delta^2 + \delta^3 \right)^d - \left( 1 - 3\delta^2 + 2\delta^3 \right)^d \\
&\quad \quad - 3(1 - \delta^2)^{2d} + (1 - \delta^2)^{3d}
\end{align*}
since $p = 1 - (1 - \delta^2)^d$. Now observe that
\begin{align*}
3\left( 1 - 2 \delta^2 + \delta^3 \right)^d - 3(1 - \delta^2)^{2d} &= 3(1 - \delta^2)^{2d} \cdot \left[ \left( 1 + \frac{\delta^3}{(1 - \delta^2)(1 + \delta)} \right)^d - 1 \right] \\
\left( 1 - 3\delta^2 + 2\delta^3 \right)^d  - \left( 1 - \delta^2 \right)^{3d} &= (1 - \delta^2)^{3d} \cdot \left[ \left( 1 + \frac{2\delta^3 + \delta^4}{(1 - \delta^2)(1 + \delta)^2} \right)^d - 1\right] \numberthis \label{eqn:triangleprob}
\end{align*}
The same bounds as in Lemma \ref{lem:signedtriangleexp} now yield that
\begin{align*}
3\left( 1 - 2 \delta^2 + \delta^3 \right)^d - 3(1 - \delta^2)^{2d} &= 3d \delta^3 (1 - \delta^2)^{2d} + O_n\left(d\delta^4 (1 - \delta^2)^{2d} \right) \\
\left( 1 - 3\delta^2 + 2\delta^3 \right)^d  - \left( 1 - \delta^2 \right)^{3d} &= 2d \delta^3  \left( 1 - \delta^2 \right)^{3d} + O_n\left( d\delta^4  \left( 1 - \delta^2 \right)^{3d} \right)
\end{align*}
Substituting $1 - p = (1 - \delta^2)^d$ and these bounds into Equation \ref{eqn:triangleprob}, we have that
\begin{align*}
\bP[T_{123} = 1] &= \left(1 - (1 - \delta^2)^d \right)^3 + 3d \delta^3 \left(1 - \delta^2 \right)^{2d} - 2d \delta^3  \left( 1 - \delta^2 \right)^{3d} + O_n\left(d\delta^4 (1 - \delta^2)^{2d} \right) \\
&= p^3 + 3d\delta^3 \left(1 - \delta^2 \right)^{2d} - 2d\delta^3  \left( 1 - \delta^2 \right)^{3d} + O_n\left(d\delta^4 (1 - p)^2 \right) \\
&= p^3 + d \delta^3 (1 + 2p)(1 - p)^2 + O_n\left(d\delta^4 \right)
\end{align*}
Substituting this into Equation \ref{eqn:exp} now completes the proof of the lemma.
\end{proof}

\subsection{Testing for Planted Poisson Matrices}
\label{subsec:plantedpois}

In this section, we prove Lemma \ref{lem:plantedpois}. The proof uses a similar second moment method computation of $\chi^2$ divergence as in the proof of Lemma \ref{lem:pcdist}.

\begin{proof}[Proof of Lemma \ref{lem:plantedpois}]
Let $\tau = \binom{t}{2} \binom{n}{2}^{-1}$. We first carry out several preliminary computations with the laws of $\text{Poisson}(\lambda)$ and $\text{Poisson}(\lambda + \tau)$ that will be useful in simplifying subsequent $\chi^2$ divergences. Observe that the following sum has a simple closed form expression.
\begin{align*}
\sum_{k = 0}^\infty \frac{\bP\left[ \text{Poisson}(\lambda) = k - 1 \right]^2}{\bP\left[ \text{Poisson}(\lambda + \tau) = k \right]} &= \lambda^{-2} e^{-\lambda + \tau} \sum_{k = 1}^\infty \frac{k}{(k-1)!} \left( \frac{\lambda^2}{\lambda + \tau} \right)^k \\
&= \frac{e^{-\lambda + \tau}}{\lambda + \tau} \sum_{k = 1}^\infty \frac{1}{(k-1)!} \left( \frac{\lambda^2}{\lambda + \tau} \right)^{k-1} + \frac{\lambda^2 e^{-\lambda + \tau}}{(\lambda + \tau)^2} \sum_{k = 2}^\infty \frac{1}{(k-2)!}  \left( \frac{\lambda^2}{\lambda + \tau} \right)^{k-2} \\
&= \frac{e^{-\lambda + \tau}}{\lambda + \tau} \cdot e^{\frac{\lambda^2}{\lambda + \tau}} + \frac{\lambda^2 e^{-\lambda + \tau}}{(\lambda + \tau)^2} \cdot e^{\frac{\lambda^2}{\lambda + \tau}} = e^{\frac{\tau^2}{\lambda + \tau}} \cdot \frac{\lambda^2 + \lambda + \tau}{(\lambda + \tau)^2} \numberthis \label{eqn14}
\end{align*}
The following two sums can be evaluated similarly.
\begin{align*}
\sum_{k = 0}^\infty \frac{\bP\left[ \text{Poisson}(\lambda) = k - 1 \right] \cdot \bP\left[ \text{Poisson}(\lambda) = k \right]}{\bP\left[ \text{Poisson}(\lambda + \tau) = k \right]} &= \lambda^{-1} e^{-\lambda + \tau} \sum_{k = 1}^\infty \frac{1}{(k-1)!} \left( \frac{\lambda^2}{\lambda + \tau} \right)^k \\
&= \frac{\lambda}{\lambda + \tau} \cdot e^{-\lambda + \tau} \cdot e^{\frac{\lambda^2}{\lambda + \tau}} = e^{\frac{\tau^2}{\lambda + \tau}} \cdot \frac{\lambda}{\lambda + \tau} \numberthis \label{eqn15} \\
\sum_{k = 0}^\infty \frac{\bP\left[ \text{Poisson}(\lambda) = k \right]^2}{\bP\left[ \text{Poisson}(\lambda + \tau) = k \right]} &= e^{-\lambda + \tau} \sum_{k = 1}^\infty \frac{1}{k!} \left( \frac{\lambda^2}{\lambda + \tau} \right)^k \\
&= e^{-\lambda + \tau} \cdot e^{\frac{\lambda^2}{\lambda + \tau}} = e^{\frac{\tau^2}{\lambda + \tau}} \numberthis \label{eqn16}
\end{align*}
Given a fixed set $S' \subseteq [n]$ of size $t$, let $\pr{poim}_P\left( n, S', \lambda \right)$ denote the distribution of $\pr{poim}_P\left( n, t, \lambda \right)$ conditioned on the event $S = S'$. If $\mathcal{U}_t$ denotes the uniform distribution on the size $t$ subsets of $[n]$, then in particular $\pr{poim}_P\left( n, t, \lambda \right) =_d \bE_{S \sim \mathcal{U}_t} \pr{poim}_P\left( n, S, \lambda \right)$. Let $\mathcal{M}_n$ denote the set of all symmetric matrices in $\mathbb{Z}_{\ge 0}^{n \times n}$ with diagonal entries equal to zero and $X$ denote an arbitrary $X \in \mathcal{M}_n$. Let $\mP_S$, $\mP$ and $\mQ$ be shorthands for $\pr{poim}_P\left( n, S, \lambda \right)$, $\pr{poim}_P\left( n, t, \lambda \right)$ and $\pr{poim}\left( n, \lambda + \tau \right)$, respectively. Observe that these are each product distributions. Following a similar second moment method computation as in Lemma \ref{lem:pcdist}, we have that
\allowdisplaybreaks
\begin{align*} 
1 + \chi^2\left(\pr{poim}_P\left( n, t, \lambda \right), \pr{poim}\left( n, \lambda + \tau \right) \right) &= \sum_{X \in \mathcal{M}_n} \frac{\bP_{\mP}[X]^2}{\bP_{\mQ}[X]} \\
&= \bE_{S, T \sim \mathcal{U}_t} \left[\sum_{X \in \mathcal{M}_n} \frac{\bP_{\mP_S}[X] \cdot \bP_{\mP_T}[X]}{\bP_{\mQ}[X]} \right] \\
&= \bE_{S, T \sim \mathcal{U}_t} \left[\prod_{1 \le i < j \le n} \left( \sum_{k = 0}^\infty  \frac{\bP_{\mP_S}[X_{ij} = k] \cdot \bP_{\mP_T}[X_{ij} = k]}{\bP_{\mQ}[X_{ij} = k]} \right) \right]
\end{align*}
The second equality holds by linearity of expectation and because $S$ and $T$ are independent. The marginal distributions of $\mP_S, \mP_T$ and $\mQ$ combined with Equations \ref{eqn14}, \ref{eqn15} and \ref{eqn16} now imply that
\allowdisplaybreaks
\begin{align*} 
&1 + \chi^2\left(\pr{poim}_P\left( n, t, \lambda \right), \pr{poim}\left( n, \lambda + \tau \right) \right) \\
&\quad \quad = \bE_{S, T \sim \mathcal{U}_t} \left[ \left( e^{\frac{\tau^2}{\lambda + \tau}} \cdot \frac{\lambda^2 + \lambda + \tau}{(\lambda + \tau)^2} \right)^{\binom{|S \cap T|}{2}} \left( e^{\frac{\tau^2}{\lambda + \tau}} \cdot \frac{\lambda}{\lambda + \tau} \right)^{2\binom{t}{2} - \binom{|S \cap T|}{2}} \left( e^{\frac{\tau^2}{\lambda + \tau}} \right)^{\binom{n}{2} - 2 \binom{t}{2} + \binom{|S \cap T|}{2}} \right] \\
&\quad \quad = e^{\binom{n}{2} \cdot \frac{\tau^2}{\lambda + \tau}} \cdot \bE_{S, T \sim \mathcal{U}_t} \left[ \left( \frac{\lambda^2 + \lambda + \tau}{(\lambda + \tau)^2} \right)^{\binom{|S \cap T|}{2}} \left( \frac{\lambda}{\lambda + \tau} \right)^{2\binom{t}{2} - \binom{|S \cap T|}{2}} \right] \\
&\quad \quad = e^{\binom{n}{2} \cdot \frac{\tau^2}{\lambda + \tau}} \cdot \left( \frac{\lambda}{\lambda + \tau} \right)^{2\binom{t}{2}} \cdot \bE_{S, T \sim \mathcal{U}_t} \left[ \left( \frac{\lambda^2 + \lambda + \tau}{\lambda^2 + \lambda \tau} \right)^{\binom{|S \cap T|}{2}} \right] \numberthis \label{eqn17}
\end{align*}
Now fix two subsets $S, T \subseteq [n]$ of size $t$ and note that $|S \cap T| \le t = O_n(1)$. Note that $e^{\binom{n}{2} \cdot \frac{\tau^2}{\lambda + \tau}} = e^{\binom{t}{2} \cdot \frac{\tau}{\lambda + \tau}} \le e^{\binom{t}{2}} = O_n(1)$. As in Lemma \ref{lem:pcdist}, $|S \cap T|$ is distributed as $\text{Hypergeometric}(n, t, t)$ since $S, T \sim \mathcal{U}_t$ are independent. Furthermore, $\bP[|S \cap T| = k] = \binom{t}{k} \binom{n - t}{t - k} \binom{n}{t}^{-1} = O_n(n^{-k})$. Observe that
\begin{align*}
&\sum_{k = 3}^t \bP[|S \cap T| = k] \cdot e^{\binom{n}{2} \cdot \frac{\tau^2}{\lambda + \tau}} \cdot \left( \frac{\lambda}{\lambda + \tau} \right)^{2\binom{t}{2}} \cdot \left( \frac{\lambda^2 + \lambda + \tau}{\lambda^2 + \lambda \tau} \right)^{\binom{k}{2}} \\
&\quad \quad \le \sum_{k = 3}^t e^{\binom{t}{2}} \cdot \binom{t}{k} \binom{n - t}{t - k} \binom{n}{t}^{-1} \left( \frac{\lambda^2 + \lambda + \tau}{\lambda^2 + \lambda \tau} \right)^{\binom{k}{2}} \\
&\quad \quad = O_n \left( \max_{2 < k \le t} n^{-k} \cdot \left( \frac{\lambda^2 + \lambda + \tau}{\lambda^2 + \lambda \tau} \right)^{\binom{k}{2}} \right) \numberthis \label{eqn18}
\end{align*}
since $t = O_n(1)$. Also observe that
\allowdisplaybreaks
\begin{align*}
&\sum_{k = 0}^2 \bP[|S \cap T| = k] \cdot e^{\binom{n}{2} \cdot \frac{\tau^2}{\lambda + \tau}} \cdot \left( \frac{\lambda}{\lambda + \tau} \right)^{2\binom{t}{2}} \cdot \left( \frac{\lambda^2 + \lambda + \tau}{\lambda^2 + \lambda \tau} \right)^{\binom{k}{2}} \\
&\quad \quad = e^{\binom{t}{2} \cdot \frac{\tau}{\lambda + \tau}} \cdot \left( \frac{\lambda}{\lambda + \tau} \right)^{2\binom{t}{2}} \cdot \left[ \binom{n - t}{t} \binom{n}{t}^{-1} + t \binom{n - t}{t - 1} \binom{n}{t}^{-1} \right. \\
&\quad \quad \quad \quad \left. + \binom{t}{2} \binom{n - t}{t - 2} \binom{n}{t}^{-1} \left( \frac{\lambda^2 + \lambda + \tau}{\lambda^2 + \lambda \tau} \right) \right]
\end{align*}
Using the fact that $\sum_{\ell = 0}^t \binom{t}{\ell} \binom{n - t}{t - \ell} \binom{n}{t}^{-1} = 1$, this quantity simplifies to
\begin{align*}
&e^{\binom{t}{2} \cdot \frac{\tau}{\lambda + \tau}} \cdot \left( \frac{\lambda}{\lambda + \tau} \right)^{2\binom{t}{2}} \cdot \left[ 1 + \binom{t}{2} \binom{n - t}{t - 2} \binom{n}{t}^{-1} \left( \frac{\lambda^2 + \lambda + \tau}{\lambda^2 + \lambda \tau} - 1 \right) - \sum_{\ell = 3}^t \binom{t}{\ell} \binom{n - t}{t - \ell} \binom{n}{t}^{-1} \right] \\
&\quad \quad = e^{\binom{t}{2} \cdot \frac{\tau}{\lambda + \tau}} \cdot \left( \frac{\lambda}{\lambda + \tau} \right)^{2\binom{t}{2}} \cdot \left[ 1 + \binom{t}{2}^2 \binom{n}{2}^{-1} \left( \frac{\lambda^2 + \lambda + \tau}{\lambda^2 + \lambda \tau} - 1 \right) + O_n\left((1 + \lambda^{-1})n^{-3}\right) \right] \numberthis \label{eqn20}
\end{align*}
The equality above follows from: (1) $\binom{n - t}{t - 2} \binom{n}{t}^{-1} = \binom{t}{2} \binom{n}{2}^{-1} + O_n(n^{-3})$, as established in Equation \ref{eqnbin}; (2) from the fact that $\frac{\lambda^2 + \lambda + \tau}{\lambda^2 + \lambda \tau} \le 1 + \lambda^{-1}$; and (3) from $\binom{t}{\ell} \binom{n - t}{t - \ell} \binom{n}{t}^{-1} = O_n(n^{-3})$ for each $3 \le \ell \le t$ and the fact that the sum contains $t - 2 = O_n(1)$ terms. Note that $\lambda = \omega_n(n^{-2})$ and thus $\frac{\tau}{\lambda + \tau} = o_n(1)$. Since $2\binom{t}{2} = O_n(1)$, we have by Taylor expanding that
\begin{align*}
e^{\binom{t}{2} \cdot \frac{\tau}{\lambda + \tau}} \left( \frac{\lambda}{\lambda + \tau} \right)^{2\binom{t}{2}} &= e^{\binom{t}{2} \cdot \frac{\tau}{\lambda + \tau}} \left( 1 - \frac{\tau}{\lambda + \tau} \right)^{2\binom{t}{2}} \\
&= \left[ 1 + \binom{t}{2} \cdot \frac{\tau}{\lambda + \tau} + O_n\left( \frac{\tau^2}{(\lambda + \tau)^2} \right) \right] \left[ 1 - 2 \binom{t}{2} \cdot \frac{\tau}{\lambda + \tau} + O_n\left( \frac{\tau^2}{(\lambda + \tau)^2} \right) \right] \\
&= 1 - \binom{t}{2} \cdot \frac{\tau}{\lambda + \tau} + O_n\left( \frac{\tau^2}{(\lambda + \tau)^2} \right) \\
&= 1 - \binom{t}{2} \cdot \frac{\tau}{\lambda + \tau} + O_n\left( \lambda^{-2} n^{-4} \right) \numberthis \label{eqn21}
\end{align*}
Substituting $\tau = \binom{t}{2} \binom{n}{2}^{-1}$ and the estimate in Equation \ref{eqn21} into Equation \ref{eqn20} yields that
\begin{align*}
&\sum_{k = 0}^2 \bP[|S \cap T| = k] \cdot e^{\binom{n}{2} \cdot \frac{\tau^2}{\lambda + \tau}} \cdot \left( \frac{\lambda}{\lambda + \tau} \right)^{2\binom{t}{2}} \cdot \left( \frac{\lambda^2 + \lambda + \tau}{\lambda^2 + \lambda \tau} \right)^{\binom{k}{2}} \\
&\quad \quad = \left[ 1 - \binom{t}{2} \cdot \frac{\tau}{\lambda + \tau} + O_n\left( \lambda^{-2} n^{-4} \right) \right] \left[ 1 + \binom{t}{2} \tau \cdot \left( \frac{\lambda^2 + \lambda + \tau}{\lambda^2 + \lambda \tau} - 1 \right) + O_n\left((1 + \lambda^{-1})n^{-3}\right) \right] \\
&\quad \quad = 1 - \binom{t}{2} \cdot \frac{\tau}{\lambda + \tau} + \binom{t}{2} \tau \cdot \left( \frac{\lambda^2 + \lambda + \tau}{\lambda^2 + \lambda \tau} - 1 \right) + O_n\left( \lambda^{-2} n^{-4} \right) + O_n\left((1 + \lambda^{-1})n^{-3}\right) \\
&\quad \quad = 1 + \binom{t}{2} \cdot \frac{\tau^2(1 - \lambda)}{\lambda^2 + \lambda \tau} + O_n\left( \lambda^{-2} n^{-4} \right) + O_n\left((1 + \lambda^{-1})n^{-3}\right) \\
&\quad \quad = 1 + O_n\left( (1 + \lambda^{-2}) n^{-4} \right) + O_n\left((1 + \lambda^{-1})n^{-3}\right) \numberthis \label{eqn22}
\end{align*}
Substituting Equations \ref{eqn22} and \ref{eqn18} into Equation \ref{eqn17} now yields that
\begin{align*}
\chi^2\left(\pr{poim}_P\left( n, t, \lambda \right), \pr{poim}\left( n, \lambda + \tau \right) \right) &= O_n\left( (1 + \lambda^{-2}) n^{-4} \right) + O_n\left((1 + \lambda^{-1})n^{-3}\right) \\
&\quad \quad + O_n \left( \max_{2 < k \le t} n^{-k} \cdot \left( \frac{\lambda^2 + \lambda + \tau}{\lambda^2 + \lambda \tau} \right)^{\binom{k}{2}} \right)
\end{align*}
Observe that $\frac{\lambda^2 + \lambda + \tau}{\lambda^2 + \lambda \tau} = \frac{\lambda}{\lambda + \tau} + \lambda^{-1} \le 1 + \lambda^{-1}$ and, when $k = 3$ in the third term above, the bound is $n^{-3} \cdot \left(\frac{\lambda^2 + \lambda + \tau}{\lambda^2 + \lambda \tau} \right)^3 = \Omega_n\left((1 + \lambda^{-1})n^{-3}\right)$. Now applying Cauchy-Schwarz as in Lemma \ref{lem:pcdist} completes the proof of the lemma.
\end{proof}

\subsection{Total Variation Convergence of RIM and POIM}
\label{subsec:rim-poim}

In this section, we complete the proof of Theorem \ref{thm:rim}. We first deduce the following elementary upper bound on the total variation between two univariate Poisson distributions using some of the calculations in Lemma \ref{lem:plantedpois}.

\begin{lemma} \label{lem:poissontv}
If $\lambda_1 \ge \lambda_2 > 0$, then it follows that
$$\TV\left( \textnormal{Poisson}(\lambda_1), \textnormal{Poisson}(\lambda_2) \right) \le \sqrt{\frac{1}{2} \left( e^{\lambda_1^{-1} (\lambda_1 - \lambda_2)^2} - 1 \right)}$$
which is $O\left( \lambda_1^{-1} (\lambda_1 - \lambda_2)^2 \right)$ if $(\lambda_1 - \lambda_2)^2 \le \lambda_1$.
\end{lemma}

\begin{proof}
By the same computation in Equation \ref{eqn16}, we have that
$$1 + \chi^2\left( \textnormal{Poisson}(\lambda_2), \textnormal{Poisson}(\lambda_1) \right) = \sum_{k = 0}^\infty \frac{\bP\left[ \text{Poisson}(\lambda_2) = k \right]^2}{\bP\left[ \text{Poisson}(\lambda_1) = k \right]} = e^{\lambda_1 - 2\lambda_2} \sum_{k = 0}^\infty \frac{1}{k!} \left( \frac{\lambda_2^2}{\lambda_1} \right)^k = e^{\lambda_1^{-1} (\lambda_1 - \lambda_2)^2}$$
Applying Cauchy-Schwarz to obtain $\TV \le \sqrt{\frac{1}{2} \cdot \chi^2}$ now proves the lemma.
\end{proof}

Observe that $\lambda_1 \ge \frac{1}{4} \left( \sqrt{\lambda_1} + \sqrt{\lambda_2} \right)^2$, from which we obtain that
$$\TV\left( \textnormal{Poisson}(\lambda_1), \textnormal{Poisson}(\lambda_2) \right) \le \sqrt{\frac{1}{2} \left( e^{4(\sqrt{\lambda_1} - \sqrt{\lambda_2})^2} - 1 \right)}$$
This implies that if $|\lambda_1 - \lambda_2| = o(1)$, then $\TV\left( \textnormal{Poisson}(\lambda_1), \textnormal{Poisson}(\lambda_2) \right) = o(1)$. Applying the triangle inequality now yields that $\TV\left( \textnormal{Poisson}(\lambda_1), \textnormal{Poisson}(\lambda_2) \right) = o(1)$ if $\lambda_2 = \lambda_2' + o(1)$ and $(\lambda_1 - \lambda_2')^2 \ll \lambda_1$. We will use this fact in the proof of Theorem \ref{thm:rim}.

We now will prove Theorem \ref{thm:rim}, referencing parts of the proof of Theorem \ref{thm:denserig} where details are identical or similar.

\begin{proof}[Proof of Theorem \ref{thm:rim}]
First observe that $n\delta \ll n^{1/2} d^{-1/3} \ll 1$. We first summarize several observations and definitions from Theorem \ref{thm:denserig} as they apply to random intersection matrices.
\begin{itemize}
\item Let $p_k = \bP[| \{ j : i \in S_j\}| = k] = \binom{n}{k} \delta^k (1 - \delta)^{n - k}$ be the probability some $i \in [d]$ is in $k$ sets $S_j$ and let $M_k$ be the number of $i \in [d]$ in exactly $k$ sets $S_k$. Note that $(M_0, M_1, \dots, M_n) \sim \text{Multinomial}(d, p_0, p_1, \dots, p_n)$.
\item A random matrix $X \sim \pr{rim}(n, d, \delta)$ can now be generated through the procedure $\mP_{\text{gen}}$ by first setting all entries of $X$ to be zero, sampling $(M_0, M_1, \dots, M_n) \sim \text{Multinomial}(d, p_0, p_1, \dots, p_n)$ and then, for each $2 \le k \le n$, independently sampling a subset $S$ of size $k$ from $[n]$ uniformly at a random a total of $M_k$ times and increasing $X_{ij}$ by 1 for each $i, j \in S$.
\item Let $\mL_P$ denote the distribution on $(M_0, M_1, \dots, M_n)$ where the $M_k$ are mutually independent and $M_k \sim \text{Poisson}(dp_k)$. Let $\pr{rim}_P(n, d, \delta)$ be the distribution on matrices $X$ generated through $\mP_{\text{gen}}$, generating $(M_0, M_1, \dots, M_n) \sim \mL_P$ instead of from a multinomial distribution.
\item Poisson splitting implies that, after sampling $(M_0, M_1, M_2)$ from $\mL_P$ and applying $\mP_{\text{gen}}$ for $k = 2$, the resulting matrix $X_2$ is distributed as $\pr{poim}\left(n, \binom{n}{2}^{-1} dp_2 \right)$.
\item Let $\pr{rim}_P(n, d, \delta, m_3, m_4, \dots, m_K)$ denote $\pr{rim}_P(n, d, \delta)$ conditioned on the event that $M_k = m_k$ for $3 \le k \le K$ and $M_k = 0$ for $K < k \le n$. Note that $X \sim \pr{rim}_P(n, d, \delta, m_3, m_4, \dots, m_K)$ is distributed as $\pr{poim}\left(n, \binom{n}{2}^{-1} dp_2 \right)$ with $m_k$ random planted increased subsets of size $k$ for each $3 \le k \le K$ as in $\mP_{\text{gen}}$.
\end{itemize}
The argument in Proposition \ref{prop:poissonization} implies that if $n \delta \ll 1$, then
\begin{equation} \label{eqn23}
\TV\left( \pr{rim}(n, d, \delta), \pr{rim}_P(n, d, \delta) \right) = O_n \left( n^2 \delta^2 \right)
\end{equation}
as $n \to \infty$. Now let
$$E_t = \min\left\{ 1, C_t \left( \left(1 +d^{-1}\delta^{-2}(1 - \delta)^{2-n} \right) n^{-2} + \max_{2 < k \le t} n^{-k/2} \left( 1 + d^{-1}\delta^{-2}(1 - \delta)^{2-n} \right)^{\frac{1}{2} \binom{k}{2}} \right) \right\}$$
for a sufficiently large constant $C_t > 0$ so that $E_t$ is an upper bound in Lemma \ref{lem:plantedpois} when $\lambda = \binom{n}{2}^{-1} dp_2 = d\delta^2(1 - \delta)^{n - 2}$. As observed above, we have that $X$ after the step of $\mP_{\text{gen}}$ with $k = 2$ is distributed as $\pr{rim}_P(n, d, \delta, 0, 0, \dots, 0) \sim \pr{poim}\left(n, d\delta^2(1 - \delta)^{n - 2} \right)$. The same induction as in Proposition \ref{prop:unionbound} yields that
\begin{equation} \label{eqn24}
\TV\left( \pr{rim}_P(n, d, \delta, m_3, m_4, \dots, m_K), \pr{poim}\left(n, \lambda(n, d, \delta, m_3, m_4, \dots, m_K)\right) \right) \le \sum_{t = 3}^K m_t E_t
\end{equation}
for each $K \ge 1$ and $(m_3, m_4, \dots, m_K) \in \mathbb{Z}_{\ge 0}^K$, where $\lambda(n, d, \delta, m_3, m_4, \dots, m_K)$ is given by
$$\lambda(n, d, \delta, m_3, m_4, \dots, m_K) = d\delta^2(1 - \delta)^{n - 2} + \sum_{t = 3}^K m_t \binom{t}{2} \binom{n}{2}^{-1}$$
We now apply the bounding argument from the end of Proposition \ref{prop:unionbound} and the conditioning argument in the beginning of Theorem \ref{thm:denserig} to reduce the proof to comparing a $\pr{poim}$ to a mixture of $\pr{poim}$ distributions. Fix some function $w = w(n) \to \infty$ as $n \to \infty$ such that $d \gg w^2 n^3$ and $w\delta \ll d^{-1/3} n^{-1/2}$. Let $E$ be the event that $(M_3, M_4, \dots, M_n) \sim \mL_P$ satisfy all of the following inequalities
\begin{align*}
dp_k - \sqrt{wdp_k} \le M_k \le dp_k + \sqrt{wdp_k} \quad &\text{for } k \ge 3 \text{ with } dp_k > w^{-1/2} \\
M_k = 0 \quad &\text{for } k \ge 3 \text{ with } dp_k \le w^{-1/2}
\end{align*}
Now note that if $k \ge 6$, since $w\delta \ll d^{-1/3} n^{-1/2}$ and $d \gg n^3$, it follows that
$$dp_k = d\binom{n}{k} \delta^k (1 - \delta)^{n - k} \le dn^k \delta^k \ll \frac{n^{k/2}}{w^6 d^{k/3 - 1}} = o_n\left(w^{-1} \right)$$
Repeating the concentration inequalities and bounds used to establish Equation \ref{eqncondition}, we have that
\begin{align*}
\bP_{\mL_P}\left[ E^c \right] &\lesssim 3w^{-1} + 3w^{-1/2} + \sum_{k = 6}^n dp_k \\
&\le 3w^{-1} + 3w^{-1/2} + \sum_{k = 6}^n dn^k \delta^k \\
&= 3w^{-1} + 3w^{-1/2} + \frac{dn^6 \delta^6}{1 - n\delta} = o_n(1)
\end{align*}
We now bound $wdp_k E_k$ for $3 \le k \le 5$ in a similar way to Proposition \ref{prop:unionbound}. First consider the case where $d\delta^2 \ge 1$. Note that since $n \delta \ll 1$, it follows that $(1 - \delta)^{n - 2} \ge 1 - (n - 2) \delta = 1 - o_n(1)$. Therefore it follows that $d^{-1} \delta^{-2} (1 - \delta)^{2 - n} = O_n(1)$ and hence $E_k = O_n(n^{-3/2})$ for each $3 \le k \le 5$. Therefore since $n \delta \ll 1$, we have that
$$wdp_k E_k \le wd \cdot (n \delta)^k \cdot n^{-3/2} \lesssim wd \cdot (n \delta)^3 \cdot n^{-3/2} = o_n(w^{-2})$$
for each $3 \le k \le 5$, since $w\delta \ll d^{-1/3} n^{-1/2}$. Now consider the case where $d\delta^2 < 1$ and let $\delta = \gamma/\sqrt{d}$ where $\gamma < 1$. It follows that $1 + d^{-1} \delta^{-2} (1 - \delta)^{2 - n} = O_n(\gamma^{-2})$ and thus for $3 \le t \le 5$, we have that
$$wdp_t E_t \lesssim w \cdot \min\left\{ d n^t \delta^t, \sum_{k = 3}^t d n^t \delta^t \cdot n^{-k/2} \gamma^{-\binom{k}{2}} \right\} = w \cdot \min\left\{ d^{1 - t/2} n^t \gamma^t, \sum_{k = 3}^t d^{1 - t/2} n^{t-k/2} \gamma^{t -\binom{k}{2}} \right\}$$
Since $\frac{2}{t - 1} \in (0, 1]$ if $3 \le t \le 5$, we have that
\allowdisplaybreaks
\begin{align*}
wdp_t E_t &\lesssim w \cdot \sum_{k = 3}^t \left( d^{1 - t/2} n^t \gamma^t \right)^{\frac{t - 3}{t - 1}} \left( d^{1 - t/2} n^{t-k/2} \gamma^{t -\binom{k}{2}} \right)^{\frac{2}{t - 1}} \\
&= w \cdot \sum_{k = 3}^t d^{1 - t/2} n^{t - \frac{k}{t - 1}} \gamma^{t - \frac{2}{t - 1} \cdot \binom{k}{2}} \\
&\le w \cdot \sum_{k = 3}^t d^{1 - t/2} n^{t - \frac{k}{t - 1}} 
\end{align*}
where the last inequality follows from $\gamma < 1$ and $t - \frac{2}{t - 1} \cdot \binom{k}{2} \ge 0$ if $k \le t$. Hence,
\allowdisplaybreaks
\begin{align*}
wdp_3 E_3 &\lesssim w d^{-1/2} n^{3/2} = o_n(1) \\
wdp_4 E_4 &\lesssim w d^{-1} n^{3} + w d^{-1} n^{8/3} = o_n(w^{-1}) \\
wdp_5 E_5 &\lesssim wd^{-3/2} n^{17/4} + w d^{-3/2} n^{4} + w d^{-3/2} n^{15/4} = o_n(w^{-3/2})
\end{align*}
since $d \gg w^2 n^3$. In summary, $wdp_k E_k = o_n(1)$ for each $3 \le k \le 5$.

Now let $\pr{rim}_E(n, d, \delta)$ and $\mL_E$ denote the distributions of $\pr{rim}(n, d, \delta)$ and $\mL_P$ conditioned on the event $E$ holding. Note that if $E$ holds, then it follows that $M_k \le dp_k + \sqrt{wdp_k} = O_n(wdp_k)$ for each $3 \le k \le 5$ with $M_k \neq 0$ and $M_k = 0$ for all other $k \ge 3$. Combining Equation \ref{eqn24}, the conditioning property of total variation, the triangle inequality and $wdp_k E_k = o_n(1)$ for $3 \le k \le 5$ yields that
\begin{align*}
&\TV\left( \pr{rim}_P(n, d, \delta), \bE_{(m_3, m_4, m_5) \sim \mL_E} \, \pr{poim}\left(n, \lambda(n, d, \delta, m_3, m_4, m_5) \right) \right) \\
&\quad \quad \le \bP\left[ E^c \right] + \TV\left( \pr{rim}_E(n, d, \delta), \bE_{(m_3, m_4, m_5) \sim \mL_E} \, \pr{poim}\left(n, \lambda(n, d, \delta, m_3, m_4, m_5) \right) \right) \\
&\quad \quad \le \bP\left[ E^c \right] + \sup_{(m_3, m_4, m_5) \in \text{supp}(\mL_E)} \TV\left( \pr{rim}_E(n, d, \delta, m_3, m_4, m_5), \pr{poim}\left(n, \lambda(n, d, \delta, m_3, m_4, m_5) \right) \right) \\
&\quad \quad \le \bP\left[ E^c \right] + \sup_{(m_3, m_4, m_5) \in \text{supp}(\mL_E)} \sum_{k = 3}^5 m_k E_k \\
&\quad \quad \lesssim \bP\left[ E^c \right] + \sum_{k = 3}^5 wdp_k E_k = o_n(1)
\end{align*}
The triangle inequality and Equation \ref{eqn23} now imply that it suffices to show
\begin{equation} \label{eqn26}
\TV\left( \pr{poim}(n, d\delta^2), \bE_{(m_3, m_4, m_5) \sim \mL_E} \, \pr{poim}\left(n, \lambda(n, d, \delta, m_3, m_4, m_5) \right) \right) = o_n(1)
\end{equation}
Now consider a matrix $X$ sampled from either $\bE_{(m_3, m_4, m_5) \sim \mL_E} \, \pr{poim}\left(n, \lambda(n, d, \delta, m_3, m_4, m_5) \right)$ or $\pr{poim}(n, d\delta^2)$. Conditioned on the event $s = \sum_{1 \le i < j \le n} X_{ij}$, the entries $(X_{ij} : 1 \le i < j \le n)$ are distributed according to $\textnormal{Multinomial}\left(s, \binom{n}{2}^{-1} \right)$ under either distribution, by Poisson splitting. To show Equation \ref{eqn26}, the conditioning property of total variation thus implies that it suffices to bound the total variation between $\sum_{1 \le i < j \le n} X_{ij}$ under the two distributions. In other words, it suffices to show the following total variation bound
\begin{equation} \label{eqn27}
\TV\left( \textnormal{Poisson}\left(\binom{n}{2} d\delta^2 \right), \bE_{(m_3, m_4, m_5) \sim \mL_E} \, \textnormal{Poisson}\left( \binom{n}{2} \lambda(n, d, \delta, m_3, m_4, m_5) \right) \right) = o_n(1)
\end{equation}
As in the proof of Theorem \ref{thm:denserig}, let $A \subseteq \{3, 4, 5\}$ be the set of indices $k$ such that $dp_k > w^{-1/2}$ and define
$$\lambda_1 = d\delta^2(1 - \delta)^{n - 2} + \sum_{k \in A} dp_k \binom{k}{2} \binom{n}{2}^{-1} \quad \text{and} \quad \lambda_2 = d\delta^2(1 - \delta)^{n - 2} + \sum_{k = 3}^5 dp_k \binom{k}{2} \binom{n}{2}^{-1}$$
Observe that
$$\left| \lambda_1 - \lambda_2 \right| = \sum_{k \in A^c \cap \{3, 4, 5\}} dp_k \binom{k}{2} \binom{n}{2}^{-1} \le 3w^{-1/2} n^{-2}$$
Also note that
\begin{align*}
\lambda_2 &= d\delta^2(1 - \delta)^{n - 2} + \sum_{k = 3}^5 d\binom{n}{k} \binom{k}{2} \binom{n}{2}^{-1} \delta^k (1 - \delta)^{n - k} \\
&= d\delta^2 \left[ (1 - \delta)^{n - 2} + \sum_{k = 3}^5 \binom{n - 2}{k - 2} \delta^{k - 2} (1 - \delta)^{n - k - 2} \right] \\
&= d\delta^2 \left[ 1 - \sum_{\ell = 4}^{n - 2} \binom{n - 2}{\ell} \delta^{\ell} (1 - \delta)^{n - 2 - \ell} \right]
\end{align*}
Since $n \delta \ll 1$ and $\delta \ll w^{-1} d^{-1/3} n^{-1/2}$, it therefore follows that
$$\left| d\delta^2 - \lambda_2 \right| = d\delta^2 \sum_{\ell = 4}^{n - 2} \binom{n - 2}{\ell} \delta^{\ell} (1 - \delta)^{n - 2 - \ell} \le \sum_{\ell = 4}^\infty n^\ell \delta^\ell \lesssim dn^4 \delta^6 \ll   \frac{n}{w^6 d}$$
Finally note that if $(m_3, m_4, m_5) \in \text{supp}(\mL_E)$, then the triangle inequality yields that
\begin{align*}
\left| \lambda_1 - \lambda(n, d, \delta, m_3, m_4, m_5) \right| &\le \sum_{k \in A} |m_k - dp_k| \cdot \binom{k}{2} \binom{n}{2}^{-1} \le \sum_{k = 3}^5 \sqrt{wdp_k} \cdot \binom{k}{2} \binom{n}{2}^{-1} \\
&\lesssim \sum_{k = 3}^5 w^{1/2} d^{1/2} n^{k/2 - 2} \delta^{k/2} \lesssim w^{1/2} d^{1/2} n^{-1/2} \delta^{3/2}
\end{align*}
since $n \delta \ll 1$. The triangle inequality now yields that
\begin{align*}
\left| \binom{n}{2} d\delta^2 - \binom{n}{2} \lambda(n, d, \delta, m_3, m_4, m_5) \right| &\lesssim w^{-1/2} + \frac{n^3}{w^6 d} + w^{1/2} d^{1/2} n^{3/2} \delta^{3/2} \\
&= o_n(1) + w^{1/2} d^{1/2} n^{3/2} \delta^{3/2}
\end{align*}
Furthermore note that
$$\frac{\left( w^{1/2} d^{1/2} n^{3/2} \delta^{3/2} \right)^2}{\binom{n}{2} d\delta^2} \lesssim w n \delta \ll 1$$
Thus by the earlier remark on total variation distances between Poisson distributions, it follows that
$$\TV\left( \textnormal{Poisson}\left(\binom{n}{2} d\delta^2 \right), \textnormal{Poisson}\left( \binom{n}{2} \lambda(n, d, \delta, m_3, m_4, m_5) \right) \right) = o_n(1)$$
for any $(m_3, m_4, m_5) \in \text{supp}(\mL_E)$. The conditioning property of total variation then implies Equation \ref{eqn27}, completing the proof of the theorem.
\end{proof}

\section{Appendix: Random Geometric Graphs on $\mathbb{S}^{d - 1}$}
\label{sec:rgg_appendix}

\subsection{Estimates for $\psi_d$}
\label{subsec:psi-prop}

In this section, we prove Lemma~\ref{lem:psi_properties} which gives key estimates for quantities in terms of $\psi_d$ and $t_{p, d}$ in our analysis of random geometric graphs.

\begin{proof}[Proof of Lemma~\ref{lem:psi_properties}]
As mentioned previously, first item is shown in Section 2 of \cite{sodin2007tail} and the second item is Lemma 2 in Section 2 of \cite{bubeck2016testing}. We now prove the remaining three items.
\begin{enumerate}
    \item[3.] From the first item in this lemma, we have that
$$\frac{\psi_d(t -\delta)}{\psi_d(t)} = \left(\frac{1-(t-\delta)^2}{1-t^2}\right)^{\frac{d-3}{2}} = \left(1 + \frac{2t\delta - \delta^2}{1-t^2}\right)^{\frac{d-3}{2}} \leq \left(1+ \frac{8t\delta}{3}\right)^{\frac{d-3}{2}} \leq e^{2td\delta}$$
    \item[4.] Let $\delta_1 = \min\left\{\tfrac{1}{\sqrt{d}},\tfrac{1}{dt_{p,d}}\right\}$. Since $\psi_d$ is decreasing, we have that
$$p =  \int_{t_{p,d}}^{1} \psi_d(x)dx \geq \int_{t_{p,d} }^{t_{p,d}+\delta_1}\psi_d(x)dx \geq \delta_1 \psi_d(t_{p,d} + \delta_1) \geq \delta_1 \psi_d(t_{p,d})e^{-2d(t_{p,d}+\delta_1)\delta_1}$$
    Note $2d(t_{p,d}+\delta_1)\delta_1 \leq C$ for some universal constant $C > 0$, from which the result follows.
    \item[5.] Since $\psi_d$ is symmetric, we have that $\mathbb{P}\left(|T| \geq t \right) = 2\Psi_d(t)$. Combining the facts that $\Psi_d(t_{p,d}) = p$, there is a constant $C > 0$ such that $t_{p,d} \leq C\sqrt{\frac{\log p^{-1}}{d}}$ and the fact that $\Psi_d$ is a decreasing function, we now have $\Psi_d\left(C\sqrt{\frac{\log p^{-1}}{d}}\right) \leq p$. Taking $t = C\sqrt{\frac{\log p^{-1}}{d}}$ in $\mathbb{P}\left(|T| \geq t \right) = 2\Psi_d(t)$, the result follows.
\end{enumerate}
This completes the proof of the lemma.
\end{proof}

\subsection{Deferred Proofs from the Coupling Argument}
\label{subsec:coupling}

In this section, we prove Lemmas \ref{lem:sphere_uniform_induction}, \ref{lem:rgg_remainder_concentration}, \ref{lem:main_reduction_rgg} and \ref{lem:expected_deviation_bound} deferred from our coupling argument analysis of random geometric graphs on $\mathbb{S}^{d - 1}$.

\begin{proof}[Proof of Lemma \ref{lem:sphere_uniform_induction}]
We prove the two items of the lemma separately.
\begin{enumerate}
\item We will show this item in the case where $a = (1, 0, 0, \dots, 0)$. The statement for any other unit vector $a \in \mathbb{S}^{d - 1}$ will follow after applying a rotation to the $a = (1, 0, 0, \dots, 0)$ case. The isotropy of the $d$-dimensional Gaussian distribution implies that a random vector $W \sim \unif(\sphere^{d-1})$ can be generated as $W = Z/\|Z\|_2$ where $Z = (Z_1, Z_2, \dots, Z_d) \sim \mN(0, I_d)$. Now let $Z_{\sim 1} = (0, Z_2, Z_3, \dots, Z_d)$ and note that
$$W = \frac{Z_1}{\|Z\|_2} \cdot a + \sqrt{1 - \frac{Z_1^2}{\|Z\|_2^2}} \cdot \frac{Z_{\sim 1}}{\|Z_{\sim 1}\|_2}$$
Note that $Z_{\sim 1} \sim \mN(0, I_{d-1})$ by definition. The rotational invariance of $\mN(0, I_{d-1})$ implies that $Z_{\sim 1}/\|Z_{\sim 1}\|_2$ and $\|Z_{\sim 1}\|_2$ are independent. Now note that $Z_1/\|Z\|_2 = Z_1/\sqrt{Z_1^2 + \|Z_{\sim 1}\|_2^2}$ is in the $\sigma$-algebra $\sigma(Z_1, \|Z_{\sim 1}\|_2)$ and thus independent of $Z_{\sim 1}/\|Z_{\sim 1}\|_2$. Furthermore, by definition we have that $Z_1/\|Z\|_2 = W_1 \sim \psi_d$ and the isotropy of $\mN(0, I_{d-1})$ implies that $Z_{\sim 1}/\|Z_{\sim 1}\|_2 \sim \unif(\sphere^{a^\perp})$. This implies that $W$ is equal in distribution to
$$W =_d Ta + \sqrt{1 - T^2} \cdot Y$$
where $T \sim \psi_d$, $Y \sim \unif(\sphere^{a^\perp})$ and $T$ and $Y$ are independent. This proves the if direction of the item of the lemma. We now prove the only if direction. If $X = Ta + \sqrt{1 - T^2} \cdot Y$ is uniformly distributed on $\mathbb{S}^{d - 1}$, then it can be coupled to $(Z_1, Z_2, \dots, Z_d) \sim \mN(0, I_d)$ so that $X = Z/\|Z\|_2$. Now note that $T$ and $Y$ are deterministic functions of $X$ with $T = X_1 = Z_1/\|Z\|_2$ and $Y = X_{\sim 1}/\|X_{\sim 1}\|_2 = Z_{\sim 1}/\|Z_{\sim 1}\|_2$. The discussion above now shows that $(T, Y)$ satisfy the three desired conditions.

\item Note that $Z_{1}, Z_2 \dots,Z_{m}$ can be completed to an orthonormal basis $Z_1, Z_2, \dots, Z_d$. Fix a procedure to do this as a deterministic function of $Z_1, Z_2,\dots,Z_m$. Let $\alpha_i = \langle X, Z_i\rangle$ and note that $X = \sum_{i=1}^{d} \alpha_i Z_i$. Now consider conditioning on $Z_1,\dots, Z_m$. Given this conditioning, we have that $X$ is uniformly distributed on $\sphere^{d-1}$ and, by rotational invariance, also that $(\alpha_1,\dots,\alpha_d) \sim \unif(\sphere^{d-1})$. The result now follows by repeatedly applying the first item of this lemma with the last $d - m$ coordinates of $(\alpha_1,\dots,\alpha_d)$ as the choices of $a$.
\end{enumerate}
This completes the proof of the lemma.
\end{proof}

\begin{proof}[Proof of Lemma~\ref{lem:rgg_remainder_concentration}]
We again proceed item by item.
\begin{enumerate}
\item Let $\xi = (a_{23}, a_{24}, \dots,a_{2n}) \in \mathbb{R}^{n-2}$ and let $\hat{\xi} = \xi/\|\xi\|_2$. By item 2 in Lemma~\ref{lem:sphere_uniform_induction}, we have that $\hat{\xi}$ is uniformly distributed over $\sphere^{n-3}$. Similarly, let $\zeta = (T_3, T_4 \dots, T_n)$ and $\hat{\zeta} = \zeta/\|\zeta\|_2$. Observe that
\begin{equation} \label{eqn:factorization}
\sum_{j=3}^n a_{2j}T_j = \langle\xi,\zeta\rangle = \|\xi\|_2 \cdot \|\zeta\|_2 \cdot \langle \hat{\xi},\hat{\zeta}\rangle
\end{equation}
Note that $\xi$ is in $\sigma(X_2, X_3, \dots, X_n)$ and $\zeta$ is in $\sigma(\Gamma_3, \Gamma_4, \dots, \Gamma_n)$, which implies that $\xi$ and $\zeta$ are independent. Therefore, it holds that $ \langle \hat{\xi},\hat{\zeta}\rangle \sim \psi_{n-2}$. By item 5 in Lemma~\ref{lem:psi_properties}, there is a constant $C_1 > 0$ depending only on $s$ such that
$$\mathbb{P}\left[\left|\langle \hat{\xi},\hat{\zeta}\rangle\right| \leq C_1 \sqrt{\frac{\log{n}}{n-2}}\right] \ge 1- \frac{1}{9n^s}$$
Since $a_{2j} = \langle X_2, Y_j \rangle$, it follows that $a_{2j} \sim \psi_d$ for each $3 \le j \le n$. Thus for some for some constant $C_2 > 0$ depending only on $s$, item 5 of Lemma~\ref{lem:psi_properties} again implies that
$$\mathbb{P}\left[a_{2j}^2 > \frac{C^2_2\log n}{d}\right] \leq \frac{1}{9n^{s+1}}$$
for each $3 \le j \le n$. Since $\|\xi\|^2_2 = \sum_{j=3}^n a_{2j}^2$, if $\| \xi \|_2 > \sqrt{\tfrac{(n-2)\log{n}}{d}}$, then it must follow that $a_{2j}^2 > C^2_2 \cdot \tfrac{\log n}{d}$ for some $j$. A union bound now yields that
\begin{equation}
\mathbb{P}\left[\|\xi\|_2 > C_2 \sqrt{\frac{(n-2)\log{n}}{d}}\right] \le \sum_{j = 3}^n \mathbb{P}\left[a_{2j}^2 > \frac{C^2_2\log n}{d}\right] \le \frac{1}{9n^s}
\label{eq:remainder_norm_concentration}
\end{equation}
By item 2 of Proposition~\ref{thm:rgg_coupling_results}, we have that $T_j \sim \psi_d$ for each $3 \le j \le n$. Repeating the same union bound argument above yields that there is a constant $C_3 > 0$ depending only on $s$ such that
$$\mathbb{P}\left[\|\zeta\|_2 > C_3 \sqrt{\frac{(n-2)\log{n}}{d}}\right] \le \frac{1}{9n^s}$$
Therefore each of the following events have probability at least $1 - \frac{1}{9n^s}$ for some constants $C_1,C_2$ and $C_3$ which depend only on $s$.
$$\left\{\left| \langle \hat{\xi},\hat{\zeta}\rangle \right| \leq C_1 \sqrt{\frac{\log{n}}{n-2}}\right\}, \quad \left\{\|\xi\|_2 \leq C_2 \sqrt{\frac{(n-2)\log{n}}{d}}\right\} \quad \text{and} \quad \left\{\|\zeta\|_2 \leq C_3 \sqrt{\frac{(n-2)\log{n}}{d}}\right\}$$
The result follows from union bound and combining these inequalities with Equation \ref{eqn:factorization}.
\item By the definition of $a_{22}$, we have that $a_{22}  = \sqrt{1 - \|\xi\|^2_2}$. If $C_2$ is as in Equation~\ref{eq:remainder_norm_concentration}, then the two events $\left\{a_{22} > \sqrt{1 - C_2^2 \cdot \frac{(n-2)\log{n}}{d}} \right\}$ and $\left\{\|\xi\|_2 \leq C_2 \sqrt{\frac{(n-2)\log{n}}{d}}\right\}$ coincide. The result now follows from Equation \ref{eq:remainder_norm_concentration}. 
\item By definition, we have that $\Gamma_i \sim \psi_{d-n+i}$. Item 5 of Lemma~\ref{lem:psi_properties} implies that
$$\mathbb{P}\left[|\Gamma_i| > C_4\sqrt{\frac{\log{n}}{d-n+i}}\right] \leq \frac{1}{3n^{s+1}}\,. $$
Using the fact that $d \gg n\log{n}$ and a union bound, we conclude the result.
\end{enumerate}
Now taking $C_s = \max(C_1C_2C_3, C_2^2, C_4)$ completes the proof of the lemma.
\end{proof}

\begin{proof}[Proof of Lemma~\ref{lem:main_reduction_rgg}]
From Equation \ref{eq:identity_for_conditional_probab}, we have that
$$Q_0 = \mathbb{P}\left[\Gamma_2 \ge t^{\prime}_{p, d} \biggr|\mathcal{F}\right] \quad \text{where} \quad t_{p,d}^{\prime} = \frac{t_{p,d} - \sum_{j=3}^n a_{2j} T_j}{a_{22} \cdot \prod_{j=3}^{n} \sqrt{1-\Gamma_j^2}}$$
Since $\Gamma_2$ is independent of $\mathcal{F}$ and $t_{p,d}^{\prime}$ is $\mathcal{F}$-measurable, we conclude by Fubini's theorem that
$$Q_0 = \Psi_{d-n+2}\left( t_{p, d}' \right)$$
Note that by definition, $p = \Psi_d(t_{p,d})$. By the triangle inequality, we have that 
\begin{equation} \label{eq:probability_deviation_triangle_inequality}
    |Q_0 - p| \leq \left|\Psi_{d-n+2}\left(t_{p, d}'\right) - \Psi_{d-n+2}(t_{p,d})\right| + \left|\Psi_{d-n+2}(t_{p,d}) -\Psi_d(t_{p,d})\right|
\end{equation}
We first will apply Lemma~\ref{lem:distributional_approx_on_the_sphere} to bound $\left|\Psi_{d-n+2}(t_{p,d}) -\Psi_d(t_{p,d})\right|$. By monotonicity, we have $\bar{\Phi}(t\sqrt{d}) \leq \bar{\Phi}(t\sqrt{d-n+2})$. Now observe that
\begin{align}
   \bar{\Phi}\left(t\sqrt{d-n+2}\right) &= \bar{\Phi}\left(t\sqrt{d}\right) + \frac{1}{\sqrt{2\pi}} \int_{t\sqrt{d-n+2}}^{t\sqrt{d}}e^{-\frac{x^2}{2}} dx \nonumber \\
   &\leq \bar{\Phi}\left(t\sqrt{d}\right) + \frac{1}{\sqrt{2\pi}} \cdot t\left(\sqrt{d}-\sqrt{d-n+2}\right) \cdot e^{-\frac{(d-n+2)t^2}{2}} \nonumber \\
   &\leq \bar{\Phi}\left(t\sqrt{d}\right) + \frac{C_1nt}{\sqrt{d}} \cdot e^{-\frac{t^2(d-n)}{2}} \label{eq:gaussian_anti_conc}
\end{align}
Here, we have used the fact that $d \gg n\log^3{n}$. Applying the standard estimate for the Gaussian CDF when $x \geq 1$ given by
$\bar{\Phi}(x) \geq \frac{1}{\sqrt{2\pi}}\left(\frac{1}{x} - \frac{1}{x^3}\right)e^{-\frac{x^2}{2}}$, we now have that
\begin{equation*}
    \bar{\Phi}\left(t\sqrt{d}\right) \geq \begin{cases} 
    \bar{\Phi}(2) &\quad \text{if } t\sqrt{d} \leq 2
    \\ \frac{1}{2t\sqrt{2\pi d}} \cdot e^{-\frac{dt^2}{2}} &\quad \text{otherwise}
    \end{cases}
\end{equation*}
Combining these inequalities with Equation \ref{eq:gaussian_anti_conc} and the fact that $d\gg n$ yields
\begin{equation}
    1 \le \frac{\bar{\Phi}\left(t\sqrt{d-n+2}\right)}{\bar{\Phi}\left(t\sqrt{d}\right)} \leq \begin{cases}
    1 + \frac{Cn}{d} &\quad \text{if } t\sqrt{d} \leq 2 \\
    1 + Cnt^2 \cdot e^{\frac{nt^2}{2}} &\quad \text{otherwise}
    \end{cases}
\end{equation}
for an absolute constant $C > 0$. Let $C_{\mathsf{est}}$ be the positive constant given in Lemma \ref{lem:distributional_approx_on_the_sphere}. Since $p \gg n^{-3}$ and $d \gg n \log n$, we have that $t_{p,d} < C_{\mathsf{est}}$ for sufficiently large $n$ by item 2 of Lemma~\ref{lem:psi_properties}. Using the distributional approximation in Lemma~\ref{lem:distributional_approx_on_the_sphere}, we can bound $\Psi_d$ and $\Psi_{d-n+2}$ as follows in terms of $\bar{\Phi}$. Since $p = \Psi_d(t_{p,d})$, we have
\begin{align}
    \bigr|\Psi_{d-n+2}(t_{p,d}) -\Psi_d(t_{p,d})\bigr| &= p \cdot \left|\frac{\Psi_{d-n+2}(t_{p,d})}{\Psi_d(t_{p,d})}-1\right| \nonumber\\
    &= p \cdot \left|\left(1+ O_n(d^{-1}) \right) \cdot e^{O_n(dt^4_{p,d})} \cdot \frac{\bar{\Phi}\left(t_{p,d}\sqrt{d-n+2}\right)}{\bar{\Phi}\left(t_{p,d}\sqrt{d}\right)} -1\right| \nonumber\\
    &= p \cdot \left|\left(1+ O_n(d^{-1}) \right) \cdot e^{O_n(dt^4_{p,d})} \cdot \left(1 +  O_n\left(nt_{p,d}^2\right) \right)-1\right| \nonumber\\
    &= O_n\left(pdt_{p,d}^4+pnt_{p,d}^2 + \frac{p}{d}\right) \label{eq:dimension_shift_density_comparison}
\end{align}
We now will bound the term $|\Psi_{d-n+2}(t_{p, d}') - \Psi_{d-n+2}(t_{p,d})|$ by approximating the density $\psi_{d-n+2}$ in the neighborhood of $t_{p,d}$. First observe that combining the items in Lemma~\ref{lem:rgg_remainder_concentration} with $d \gg n \log n$ implies that
\begin{equation} \label{eqn:denombound}
a_{22} \cdot \prod_{j=3}^{n} \sqrt{1-\Gamma_j^2} = 1 - O_n\left( \frac{n \log n}{d} \right)
\end{equation}
on the event $E_{\mathsf{rem}}$. Combining this bound with the expression for $t_{p, d}'$, the fact that $p \gg n^{-3}$ and the bounds in item 2 of Lemma~\ref{lem:psi_properties} now yields that
\begin{equation}
|t_{p,d}-t_{p,d}^{\prime}| \cdot \mathbbm{1}(E_{\mathsf{rem}}) = O_n\left(\frac{\sqrt{n}\log^{3/2}{n}}{d} \right)
\label{eq:thresold_difference_bound}
\end{equation}
Observe that this difference is $O_n\left( \sqrt{\tfrac{\log n}{d}} \right)$ on the event $E_{\mathsf{rem}}$ since $d \gg n \log^2 n$. Let
$$u = \underset{x \in \left[t_{p, d}, t_{p, d}'\right]}{\text{argmin}} \, |x|$$
Note that $u = O_n\left( \sqrt{\tfrac{\log n}{d}} \right)$ conditioned on $E_{\mathsf{rem}}$. Thus given $E_{\mathsf{rem}}$ holds,
\allowdisplaybreaks
\begin{align*}
\bigr|\Psi_{d-n+2}\left(t^{\prime}_{p,d}\right) - \Psi_{d-n+2}(t_{p,d})\bigr| &= \left| \int_{t_{p,d}}^{t'_{p,d}} \psi_{d-n+2}(x)dx \right| \leq \psi_{d-n+2}\left(u \right) \cdot \left|t^{\prime}_{p,d}-t_{p,d}\right| \\
    &= \frac{\psi_{d-n+2}\left(u \right)}{\psi_d(u)} \cdot \frac{\psi_d(u)}{\psi_d(t_{p,d})} \cdot \psi_d(t_{p,d}) \cdot \left|t^{\prime}_{p,d}-t_{p,d}\right|
\end{align*}
where the inequality follows from the fact that $\psi_{d - n + 2}(t)$ is monotonically decreasing in $|t|$. Furthermore, we have that
\allowdisplaybreaks
\begin{align*}
    &\bigr|\Psi_{d-n+2}\left(t^{\prime}_{p,d}\right) - \Psi_{d-n+2}(t_{p,d})\bigr| \\
    &\quad \quad \lesssim \left( \sqrt{\frac{d - n + 2}{d}} \cdot \left( 1 - u^2 \right)^{-n/2} \right) e^{2dt_{p,d} | u - t_{p, d} |} \cdot \psi_d(t_{p,d}) \cdot \left|t^{\prime}_{p,d}-t_{p,d}\right| \\
    &\quad \quad \lesssim \sqrt{\frac{d - n}{d}} \cdot \left( 1 + O_n\left( \frac{n \log n}{d} \right) \right) e^{2dt_{p,d} | t'_{p,d} - t_{p, d} |} \cdot \psi_d(t_{p,d}) \cdot \left|t^{\prime}_{p,d}-t_{p,d}\right| \\
    &\quad \quad = \left( 1 + O_n\left( \frac{n \log n}{d} \right) \right) \exp\left( O_n\left( d \cdot \sqrt{ \frac{\log p^{-1}}{d}} \cdot \frac{\sqrt{n}\log^{3/2}{n}}{d} \right) \right) \cdot \psi_d(t_{p,d}) \cdot \left|t^{\prime}_{p,d}-t_{p,d}\right| \\
    &\quad \quad = \left( 1 + O_n\left( \frac{n \log n}{d} + \sqrt{\frac{n \log^4 n}{d}} \right) \right) \cdot \psi_d(t_{p,d}) \cdot \left|t^{\prime}_{p,d}-t_{p,d}\right| \\
    &\quad \quad = \left(1+o_n(1)\right) \cdot \psi_d(t_{p,d}) \cdot \left|t^{\prime}_{p,d}-t_{p,d}\right|
\end{align*}
The second inequality follows from items 1 and 3 of Lemma~\ref{lem:psi_properties} and using $\Gamma\left(\frac{d}{2}\right)/\Gamma\left(\frac{d-1}{2}\right)\sqrt{\pi} = \Theta(\sqrt{d})$. The third inequality follows from the fact that Bernoulli's inequality implies that $(1 - u^2)^{-n/2} \le 1 + nu^2$ if $nu^2 \le 1$. The third last equality follows from item 2 of Lemma~\ref{lem:psi_properties}, the fact that $p \gg n^{-3}$ and Equation \ref{eq:thresold_difference_bound}. The final estimate follows from the fact that $d \gg n\log^4 n$. Let $C > 0$ be the constant in the $\lesssim$ above. Substituting this bound into Equation~\ref{eq:probability_deviation_triangle_inequality}, we have 
$$|Q_0-p| \cdot \mathbbm{1}(E_{\mathsf{rem}})\leq O_n\left(pdt_{p,d}^4+pnt_{p,d}^2 + \frac{p}{d}\right) + C(1+o_n(1)) \cdot \psi_{d}\left(t_{p,d}\right) \cdot \bigr|t^{\prime}_{p,d}-t_{p,d}\bigr|$$
Now note that
$$pdt_{p,d}^4+pnt_{p,d}^2 + \frac{p}{d} = O_n\left(\frac{pn}{d}\log{ p^{-1}}\right)$$
Therefore we have that for sufficiently large $n$,
$$|Q_0-p| \cdot \mathbbm{1}(E_{\mathsf{rem}})\leq O_n\left(\frac{pn}{d}\log{ p^{-1}}\right) +2 C\psi_{d}\left(t_{p,d}\right) \cdot \bigr|t^{\prime}_{p,d}-t_{p,d}\bigr| $$
This proves the first claim in the lemma. Using the fact that $\psi_d(t_{p,d})\leq Cp\sqrt{d\log p^{-1}}$ and Equation~\ref{eq:thresold_difference_bound}, we conclude that
$$|Q_0-p| \cdot \mathbbm{1}(E_{\mathsf{rem}})\leq O_n\left(\frac{pn}{d}\log{ p^{-1}}\right) + O_n\left( p\sqrt{\frac{n\log{ p^{-1}}}{d}} \cdot\log^{3/2}{n} \right)$$
which proves the second claim in the lemma.
\end{proof}

\begin{proof}[Proof of Lemma~\ref{lem:expected_deviation_bound}]
Given the event $E_{\mathsf{rem}}$, Equation \ref{eqn:denombound} and the expression for $t_{p,d}'$ imply that
\begin{align*}
    \bigr|t^{\prime}_{p,d}-t_{p,d}\bigr| \cdot \mathbbm{1}(E_{\mathsf{rem}}) &\leq O_n\left(t_{p,d} \cdot \frac{n\log{n}}{d}\right) + (1+o_n(1)) \cdot \left|\sum_{j=3}^n T_j a_{2j}\right| \\
    &\leq O_n\left(n\left(\frac{\log{n}}{d}\right)^{\frac{3}{2}}\right) + (1+o_n(1)) \cdot \left|\sum_{j=3}^n T_j a_{2j}\right| 
\end{align*}
using the upper bound on $t_{p,d}$ in item 2 of Lemma~\ref{lem:psi_properties}. The inequality $(x+y)^2 \leq 2x^2 + 2y^2$ yields
\begin{align}
    \mathbb{E}\left[ \bigr|t^{\prime}_{p,d}-t_{p,d}\bigr|^2 \cdot \mathbbm{1}(E_{\mathsf{rem}}) \right] \lesssim n^2\left(\frac{\log{n}}{d}\right)^{3}+ \mathbb{E}\left[ \left|\sum_{j=3}^n T_j a_{2j}\right|^2 \right]
\label{eq:main_eq_around_here}
\end{align}
Now recall that $Y_j$ is a unit norm random vector in the $\sigma$-algebra $\sigma(X_j, \dots, X_n)$. Therefore, for $j \geq 3$, $Y_j$ is independent of $X_2$. Also note that $a_{2j} = \langle X_2,Y_j\rangle$. Furthermore, the $T_j$ are independent of $X_2,\dots,X_n$ and, since the random variable $T_j T_k$ is symmetric about zero, we have that $\mathbb{E}[T_jT_k] = 0$ for $k\neq j$. Thus $\mathbb{E}[T_jT_k a_{2j}a_{2k}] = \mathbb{E}[T_jT_k] \cdot \bE[a_{2j}a_{2k}] = 0$ if $j\neq k$, and hence
$$\mathbb{E}\left[ \left|\sum_{j=3}^n T_j a_{2j}\right|^2 \right] = \sum_{j=3}^{n} \mathbb{E}\left[T_j^2 a_{2j}^2\right] = \sum_{j=3}^{n} \mathbb{E}\left[T_j^2\right] \cdot \mathbb{E}\left[a_{2j}^2\right] = \frac{n - 2}{d^2} \le \frac{n}{d^2}$$
Where the equality holds because $\mathbb{E}\left[a_{2j}^2\right] = \mathbb{E}\left[T_j^2\right] = \frac{1}{d}$ for $3 \le j \le n$, since $a_{2j}, T_j \sim \psi_d$ by item 2 in Proposition~\ref{thm:rgg_coupling_results}. Substituting this into Equation \ref{eq:main_eq_around_here} completes the proof of the lemma.
\end{proof}
\end{appendices}
\end{document}